\documentclass[11pt]{amsart}
\usepackage{amscd,amsmath,amssymb,amsfonts,verbatim}
\usepackage[cmtip, all]{xy}
\usepackage{MnSymbol}
\usepackage[a4 paper, left=3cm, right=3cm, top = 2.5cm, bottom = 2.5cm ]{geometry}
\usepackage{comment}
\usepackage{tikz-cd}
\newtheorem{thm}{Theorem}[section]
\newtheorem{prop}[thm]{Proposition}
\newtheorem{lem}[thm]{Lemma}
\newtheorem{cor}[thm]{Corollary}
\newtheorem{conj}[thm]{Conjecture}
\theoremstyle{definition}
\newtheorem{ques}[thm]{Question}

\newtheorem{defn}[thm]{Definition}
\theoremstyle{remark}
\newtheorem{remk}[thm]{Remark}
\newtheorem{remks}[thm]{Remarks}

\newtheorem{exm}[thm]{Example}
\newtheorem{exms}[thm]{Examples}
\newtheorem{notat}[thm]{Notation}
\numberwithin{equation}{section}

{\hfill$\square$\end{defn}}
{\hfill$\square$\end{remk}}
{\hfill$\square$\end{remks}}
{\hfill$\square$\end{exm}}
{\hfill$\square$\end{exms}}
{\hfill$\square$\end{notat}}

\newcommand{\thmref}{Theorem~\ref}
\newcommand{\propref}{Proposition~\ref}
\newcommand{\corref}{Corollary~\ref}

\newcommand{\lemref}{Lemma~\ref}

\newcommand{\sC}{{\mathcal C}}
\newcommand{\sD}{{\mathcal D}}

\newcommand{\sF}{{\mathcal F}}
\newcommand{\sG}{{\mathcal G}}
\newcommand{\sH}{{\mathcal H}}
\newcommand{\sI}{{\mathcal I}}
\newcommand{\sJ}{{\mathcal J}}
\newcommand{\sK}{{\mathcal K}}
\newcommand{\sL}{{\mathcal L}}
\newcommand{\sM}{{\mathcal M}}

\newcommand{\sO}{{\mathcal O}}

\newcommand{\sS}{{\mathcal S}}

\newcommand{\sU}{{\mathcal U}}

\newcommand{\sX}{{\mathcal X}}

\newcommand{\sZ}{{\mathcal Z}}
\newcommand{\A}{{\mathbb A}}

\newcommand{\F}{{\mathbb F}}
\newcommand{\G}{{\mathbb G}}
\renewcommand{\H}{{\mathbb H}}

\newcommand{\N}{{\mathbb N}}
\renewcommand{\P}{{\mathbb P}}
\newcommand{\Q}{{\mathbb Q}}

\newcommand{\Z}{{\mathbb Z}}

\newcommand{\dm}{{\mathbf{DM}}}
\newcommand{\dr}{{\mathbf{D}}}

\newcommand{\fm}{{\mathfrak m}}

\newcommand{\fp}{{\mathfrak p}}

\newcommand{\ff}{{\mathfrak f}}

\newcommand{\hs}{\heartsuit}
\newcommand{\lw}{{\rm lw}}

\newcommand{\Lci}{{\rm lci}}

\newcommand{\Ker}{{\rm Ker}}
\newcommand{\gr}{{\rm gr}}

\newcommand{\Alb}{{\rm Alb}}
\newcommand{\CH}{{\rm CH}}

\newcommand{\surj}{\twoheadrightarrow}
\newcommand{\inj}{\hookrightarrow}
\newcommand{\red}{{\rm red}}

\newcommand{\Pic}{{\rm Pic}}

\newcommand{\Hom}{{\rm Hom}}

\newcommand{\Spec}{{\rm Spec \,}}
\newcommand{\sing}{{\rm sing}}
\newcommand{\Char}{{\rm char}}

\newcommand{\Tr}{{\rm Tr}}

\newcommand{\ab}{\rm ab}
\newcommand{\nr}{\rm nr}

\newcommand{\divf}{{\rm div}}

\newcommand{\id}{{\operatorname{id}}}

\newcommand{\Sch}{{\operatorname{\mathbf{Sch}}}} 
\newcommand{\Reg}{{\operatorname{\mathbf{Reg}}}}

\newcommand{\<}{\langle}
\renewcommand{\>}{\rangle}

\newcommand{\Sm}{{\mathbf{Sm}}}

\newcommand{\cyc}{{\operatorname{\rm cyc}}}

\newcommand{\et}{{\text{\'et}}}

\newcommand{\ds}{{/\kern-3pt/}}

\newcommand{\Nsch}{{\operatorname{\mathbf{NSch}}}}

\newcommand{\Br}{{\operatorname{Br}}}
\newcommand{\sm}{{\operatorname{sm}}}

\newcommand{\tr}{{\operatorname{tr}}}

\newcommand{\un}{\underline}
\newcommand{\ov}{\overline}

\renewcommand{\dim}{\text{\rm dim}}

\newcommand{\tuborg}{\left\{\begin{array}{ll}}
\newcommand{\sluttuborg}{\end{array}\right.}

\newcommand{\zar}{{\rm zar}}

\newcommand{\reg}{{\rm reg}}

\newcommand{\tor}{{\rm tor}}

\newcommand{\Cores}{{\rm Cores}}
\newcommand{\inv}{{\rm inv}}
\newcommand{\tm}{{\rm t}}

\newcommand{\wt}{\widetilde}
\newcommand{\wh}{\widehat}
\newcommand{\cont}{{\rm cont}}

\newcommand{\coker}{{\rm Coker}}
\newcommand{\Fil}{{\rm fil}}

\newcommand{\htt}{{\rm ht}}

\newcommand{\etl}{{\acute{e}t}}

\newcounter{elno}

\newcounter{elno-abc}   

\newcounter{elno-abc-prime}

\begin{document}
\title[Unramified cohomology and B-M pairing]{Unramified cohomology and Brauer--Manin pairing}
\author{Amalendu Krishna, Jitendra Rathore}
\address{Department of Mathematics, South Hall, Room 6607, University of California
  Santa Barbara, CA, 93106-3080, USA.}
\email{amalenduk@math.ucsb.edu}
\address{Department of Mathematics, South Hall, Room 6607, University of California
  Santa Barbara, CA, 93106-3080, USA.}
\email{jitendra@math.ucsb.edu}




\maketitle
\vskip .4cm

\begin{quote}\emph{Abstract.}
  We show that the Tate module of the $\ell$-adic unramified cohomology
  $H^3_{\nr}(X, {\Q_\ell}/{\Z_\ell})$ of
  a smooth projective surface $X$ over a local or strictly local field  $k$ of
  residue characteristic prime to $\ell$ vanishes under suitable conditions.
  We apply this to answer some questions concerning the left and
  right kernels of the Brauer--Manin pairing for open subvarieties of $X$.
  To prove the vanishing theorem, we establish Bloch's formula for the Chow group of
  codimension-two cycles on a semistable model of $X$, and derive applications to
  the restriction map for such cycles to the closed
  fiber of the model.
\end{quote}

\setcounter{tocdepth}{1}
\tableofcontents

\vskip .4cm

\section{Introduction}\label{sec:Intro}
\subsection{Objective of the paper}\label{sec:Background}
The Chow group of zero-cycles and the Brauer group of algebraic varieties are
fundamental objects of study in arithmetic geometry.
An important tool for investigating their
relationship is the Brauer--Manin pairing, which directly links the group of
zero-cycles with the Brauer group of a smooth variety over a non-archimedean local
field. This pairing is well understood for curves, whereas for smooth projective
surfaces, it has been the subject of extensive study, leading to several deep results.

The main objective of this paper is to extend these results from smooth projective
surfaces to smooth quasi-projective surfaces. To this end, we prove a vanishing
theorem for the Tate module of the third unramified cohomology
of smooth projective surfaces and establish
Bloch's formula for the Chow group of codimension-two cycles on their semistable
models. As part of the proof, we carry out a detailed study of the Gersten complexes
for the $K_2$-sheaves with integral and finite coefficients on a given model,
obtaining new results on their cohomology. In this section, we provide the necessary
background and state our main results.

\subsection{Vanishing of unramified cohomology}\label{sec:VRC}
Let $R$ be an excellent henselian discrete valuation ring whose residue field $\F$ is
either finite or separably closed of exponential characteristic $p > 0$. Let $k$ denote
the quotient field of $R$. Let $X$ be a connected, smooth, projective surface
over $k$ and let $\Alb_X$ denote the albanese variety of $X$.
Given a prime $\ell \neq p$ and  $i \ge 1$,  let $\sH^i({\Z}/{\ell^n}(i-1))$ be
the Zariski sheaf on $X$ associated to the presheaf
$U \mapsto H^{i}_\et(U, \mu^{\otimes {i-1}}_{\ell^n})$.
We let $H^i_{\nr}(X, {\Z}/{\ell^n}) := H^0_\zar(X, \sH^i({\Z}/{\ell^n}(i-1)))$ denote the
$i$-th unramified cohomology of $X$ with ${\Z}/{\ell^n}$-coefficients
(cf. \cite[\S~4.1]{CT-SB}). We let
$H^i_{\nr}(X, \Z_\ell) = \varprojlim_n  H^i_{\nr}(X, {\Z}/{\ell^n})$.
This coincides with the Tate module of the unramified cohomology
$H^i_{\nr}(X, {\Q_\ell}/{\Z_\ell}) := \varinjlim H^i_{\nr}(X, {\Z}/{\ell^n})$
(cf. \lemref{lem:Chow-finite}).

Let $\sX$ be a connected, regular, proper scheme of finite type over $R$ whose generic
fiber is $X$ and whose reduced closed fiber $Y$ is a simple normal crossings 
divisor on $\sX$ (cf. \S~\ref{sec:Gersten-model}).
To state our main results, we introduce the following notation.

{\sl {\bf Notation:}} Given an integer $n \ge 1$, we shall say that the
$(\star_n)$-condition holds for $\sX$
if either of the following conditions is satisfied.
\begin{enumerate}
  \item
$n$ is odd.
  \item
    $\sqrt{-1} \in \F$.
 \item
  $\sX$ is a smooth $R$-scheme.
\item
  $R$ is equicharacteristic.
  \end{enumerate}

Given a prime $\ell$, we shall say that the $(\star\star_\ell)$-condition holds
for $\sX$ if either of the following conditions is satisfied.
  \begin{enumerate}
  \item
  $\F$ is finite, the $(\star_\ell)$-condition holds,
  $\Alb_X$ has potentially good reduction and the irreducible
    components of $Y$ satisfy the Tate conjecture.
  \item
    $\F$ is finite, $\Alb_X$ has potentially good reduction and the irreducible
    components of $Y_{\F'}$ satisfy the Tate conjecture, where $\F' = \F(\sqrt{-1})$.
  \end{enumerate}

\vskip.2cm

  We now fix a prime $\ell \neq p$.
The first two results of this paper are the following refinements of 
the seminal theorems of Esnault--Wittenberg from \cite{Esnault-Wittenberg}.

\begin{thm}\label{thm:Main-0}
  Assume that the $(\star\star_\ell)$-condition holds.
  Then $H^3_{\nr}(X, \Z_\ell) =0$.
\end{thm}

A result of Parimala--Suresh (cf. \cite[Cor.~7.6]{Parimala-Suresh}) shows that the
assumption on $\Alb_X$ in the above theorem can not be omitted.
We do not know to what extent the assumption on $Y$ can be weakened.

\vskip.2cm

 Before we state the analogous result when $\F$ is separably closed, recall from
\cite[\S~4]{Esnault-Wittenberg} that if $V$ is a variety over a separably closed
field, then one says that $H^2_\et(V, {\Q_\ell}(1))$ is algebraic if the natural
injection $\Pic(V)^{(\ell)}\otimes_{\Z_\ell} \Q_\ell \to H^2_\et(V, {\Q_\ell}(1))$, induced
by the Kummer sequence, is an isomorphism, where $A^{(\ell)} =
\varprojlim_m {A}/{\ell^m}$ for any abelian group $A$. For a $k$-scheme $Z$, we let
$\ov{Z}$ denote the base change of $Z$ by a separable closure of $k$.

\begin{thm}\label{thm:Main-12}
  Assume that $\F$ is separably closed and $H^2_{{\rm {\etl}}}(\ov{X}, {\Q_\ell}(1))$
  is algebraic. Then $H^3_{\nr}(X, \Z_\ell) =0$.
  
\end{thm}

\vskip.2cm

We expect Theorems~\ref{thm:Main-0} and ~\ref{thm:Main-12}
to have a number of applications, especially in the
study of motivic and {\'e}tale cohomology of open surfaces over local and strictly
local fields. Our results on the Brauer--Manin pairing (see below) rely heavily on
both the statement and the proof of \thmref{thm:Main-0}.

\subsection{Bloch's formula for $\CH_1(\sX)$ and applications}\label{sec:BF}
Let $R$ be an excellent discrete valuation ring with an arbitrary residue field $\F$.
Let $\sX$ be a
finite type, connected, regular $R$-scheme of dimension three which is faithfully
flat over $R$ and whose reduced closed fiber $Y$ is a simple normal crossings divisor
on $\sX$. Let $\sK_{i,\sX}$ (resp. $\sK^M_{i, \sX}$) denote the Zariski sheaf of $i$-th
Quillen (resp. Milnor) $K$-theory on $\sX$. In \cite{Kerz-JAG}, Kerz introduced
an improved version of $\sK^M_{i, \sX}$ which is denoted by $\wh{\sK}^M_{i, \sX}$.
However, the canonical map $\wh{\sK}^M_{i, \sX} \to \sK_{i,\sX}$ is an isomorphism
for $i \le 2$ (cf. Prop.~14 of op. cit.). Let  $\CH_1(\sX)$ denote the Chow
group of 1-cycles on $\sX$ (cf. \cite[Chap.~1]{Fulton}). As a key ingredient for proving
Theorems~\ref{thm:Main-0} and ~\ref{thm:Main-12}, we prove the
following Bloch--Kato formula for $\CH_1(\sX)$. This formula is well-known
when $R$ has equal characteristic (cf. \cite{Panin}) but has remained unknown
otherwise.

\begin{thm}\label{thm:Main-1}
  There are canonical isomorphisms
  \[
  H^2_\zar(\sX, \sK^M_{2,\sX}) \xrightarrow{\cong} H^2_\zar(\sX, \sK_{2,\sX})
  \xrightarrow{\cong}  \CH_1(\sX).
  \]
\end{thm}

\vskip.2cm

With finite coefficients, we prove the following improved version of
\thmref{thm:Main-1}
which provides a Gersten type presentation of $H^1_\zar(\sX, {\sK_{2, \sX}}/n)$ as
well. This plays a key role in the proof of \thmref{thm:Main-0}.
Similar to \thmref{thm:Main-1}, this is a new result when $R$ has 
mixed characteristic. 

Let $n \in \F^\times$ and let 
\begin{equation}\label{eqn:Ger*}
  {K_2(k(\eta))}/n \xrightarrow{\partial_0} {\underset{x \in \sX^{(1)}}\bigoplus}
  {k(x)^\times}/n \xrightarrow{\partial_1} {\underset{x \in \sX^{(2)}}\bigoplus}
  {\Z}/n \to 0
\end{equation}
denote the complex obtained by taking the global sections of
sheaves in the Gersten complex for ${\sK_{2, \sX}}/n$, where $\eta$ denotes the
generic point of $\sX$. 

\begin{thm}\label{thm:Main-11}
  Assume that the $(\star_n)$-condition holds. Then there exist exact sequences
  \[
    {K_2(k(\eta))}/n \xrightarrow{\partial_0}  \Ker(\partial_1) \to
    H^1_\zar(\sX, {\sK_{2, \sX}}/n) \to 0;
\]
\[
{\underset{x \in \sX^{(1)}}\bigoplus}
  {k(x)^\times}/n \xrightarrow{\partial_1} {\underset{x \in \sX^{(2)}}\bigoplus}
  {\Z}/n \to H^2_\zar(\sX, {\sK_{2, \sX}}/n) \to 0.
    \]
\end{thm}

By assigning a rational function (or a pair of rational functions) its graph,
it is not hard to see using \thmref{thm:Main-11} that there exists a canonical map
$H^1_\zar(\sX, {\sK_{2, \sX}}/n) \to \CH^2(\sX, 1; {\Z}/n)$, where
the latter is one of the higher Chow groups of $\sX$ with finite coefficients,
defined by Levine \cite{Levine-JAG}. Following the strategy of \cite[\S~2.3]{Landsburg},
it should be possible to prove that this map is an isomorphism.
We do not pursue this question here.

\vskip.3cm

\subsection{Restriction map for relative zero-cycles}\label{sec:Res-map}
We give an application of \thmref{thm:Main-1} to a question of
Esnault--Kerz--Wittenberg \cite{EKW}.
Let $n$ be an integer not divisible by $\Char(\F)$.
If $\sX$ is  smooth and projective over $R$ and $\F$ is either finite or
separably closed, Saito--Sato (cf. \cite[Cor.~0.10]{Saito-Sato-Ann}) showed without
any condition on $\dim(\sX)$ that the intersection of a cycle with $Y$ defines
a restriction map
$\rho \colon {\CH_1(\sX)}/n \to {\CH_0(Y)}/n$. Subsequently, Esnault--Kerz--Wittenberg
\cite{EKW} showed that the restriction map of Saito--Sato exists more generally,
namely, if $\F$ is either finite
or algebraically closed, or $(d-1)!$ is prime to $n$ (where $d = \dim(Y)$), or $\sX$
is smooth over $R$ or $R$ has equal characteristic, provided one replaces
${\CH_0(Y)}/n$ by 
the Friedlander--Voevodsky \cite{Friedlander-Voevodsky} motivic cohomology
$H^{2d}(Y, {\Z}/n(d))$ (cf. Definition~\ref{defn:MC-defn}).
In particular, unlike \cite{Saito-Sato-Ann}, the result of
\cite{EKW} assumes no condition on $\F$ when $d \le 2$.

In the words of Esnault--Kerz--Wittenberg (cf. \cite[\S~1]{EKW}), it is
expected that the restriction map $\rho$,
which they constructed with ${\Z}/n$-coefficients, exists integrally.
We note that, although Saito--Sato and Esnault--Kerz--Wittenberg proved
that $\rho$ is an isomorphism, the principal difficulty in their arguments lies in the
construction of the map $\rho$, rather than proving its bijectivity.
As a first application of \thmref{thm:Main-1}, we establish the following integral
version of the map $\rho$, thereby answering to the
expectation of  Esnault--Kerz--Wittenberg in the case $d =2$. 

Let $\sX$ be as in \thmref{thm:Main-1} (in particular, $\F$ is
arbitrary). Let $Y_\reg$ denote the regular locus of $Y$.
Let $Z^g_1(\sX)$ denote the subgroup of $Z_1(\sX)$ generated by integral cycles
which are finite (hence flat) over $R$ and do not meet the singular locus of $Y$
(such cycles will be said to be in good position). Let $\CH^{\Lci}_0(X)$ denote the
(cohomological) Chow group of 0-cycles on $Y$
(cf. \cite{Binda-Krishna-Comp}, \S~\ref{sec:Prf-2}).

\begin{thm}\label{thm:Main-2}
  There exists a commutative diagram
\begin{equation}\label{eqn:Main-2-0}
\xymatrix@C.8pc{
Z^g_1(\sX) \ar[r]^-{\wt{\rho}} \ar[d] & 
Z_0(Y_\reg) \ar@{->>}[d] \\
\CH_1(\sX) \ar[r]^-{\rho} & \CH^{\Lci}_0(Y).}
\end{equation}
Here, the vertical arrows are the canonical projections, and $\wt{\rho}$ is the group
homomorphism given by taking an $1$-cycle in good position and intersecting it with
$Y$. If $\sX$ is quasi-projective over $R$, then the left vertical arrow is
also surjective.
\end{thm}

\vskip.2cm

If we let $\Lambda = \Z$ if $\Char(\F) = 0$, and
$\Lambda = \Z[\tfrac{1}{p}]$ if $\Char(\F) = p > 0$, then there is
a canonical homomorphism $\CH^{\Lci}_0(Y) \otimes_{\Z} \Lambda \to H^{4}(Y, \Lambda(2))$
by \cite[Thm.~1.6]{Binda-Krishna-SNS}.
We can therefore replace $\CH^{\Lci}_0(Y)$ by $H^{4}(Y, \Lambda(2))$ in
~\eqref{eqn:Main-2-0} if we work with $\Lambda$-coefficients
(cf. \corref{cor:Main-2*}).

\vskip.2cm

Combining \thmref{thm:Main-2} with \cite[Cor.~1.2]{Binda-Krishna-SNS}, we obtain the
following corollary which yields a quick proof of the $d =2$ case of the main result
of \cite{EKW}.
\begin{cor}\label{cor:Main-3}
  Assume in \thmref{thm:Main-2} that $\F$ is perfect and
  $\sX$ is projective over $R$. Then
  the map $\rho \colon {\CH_1(\sX)}/n \to H^{4}(Y, {\Z}/n(2))$
  is an isomorphism.
  \end{cor}

\begin{remk}\label{remk:Finiteness}
  An immediate consequence of Theorems~\ref{thm:Main-1} and ~\ref{thm:Main-11}
  is a simple proof of the
  fact that if $\sX$ is a proper $R$-scheme and $\F$ is either finite or separably
  closed, then $\CH_1(\sX)[n]$ is finite for any integer $n \in \F^\times$. This
  is likely to be known to experts and can be independently derived using the
  corresponding results for the generic (due to Suslin \cite[Cor.~4.4]{Suslin-KT})
  and closed fibers.
  \end{remk}

\subsection{Cycle class map and left kernel of Brauer--Manin pairing for open surfaces}\label{sec:BML}
We now state our applications of \thmref{thm:Main-0} to the Brauer--Manin pairing and
the cycle class map for open surfaces over local and strictly local fields.
Let $R$ be an excellent henselian discrete valuation ring whose residue field $\F$
is finite of characteristic $p > 0$. Let $k$ denote the quotient field of $R$.

Recall (cf. \cite{Manin}) that, for a smooth projective variety $X$ over $k$,
evaluating Brauer classes at the closed points of $X$ defines a pairing
\begin{equation}\label{eqn:BM-pair-0}
  \CH_0(X) \times \Br(X) \to {\Q}/{\Z}
\end{equation}
between the Chow group of 0-cycles $\CH_0(X)$ and the Brauer group $\Br(X)$.
Known as the Brauer--Manin pairing, ~\eqref{eqn:BM-pair-0}
has been studied extensively by many
authors (see, for instance, \cite{CT-LNM, CT-Saito, Parimala-Suresh, Saito-Sato-ENS,
  Yamazaki-MA}). A central problem in the study of Brauer--Manin pairing and
its applications is the description of its left and right kernels.

When $X$ is a curve, the Brauer--Manin pairing has trivial left and right kernels,
thanks to the pioneering work of Lichtenbaum \cite{Lichtenbaum} (see 
\cite{Saito-Invent-1} for the corresponding result in positive characteristic).
However,  Parimala--Suresh \cite{Parimala-Suresh} showed that the Brauer--Manin pairing
has a non-zero left kernel, even modulo the maximal prime-to-$p$ divisible subgroup
of $\CH_0(X)$, when $\dim(X) \ge 2$. In the latter case, Colliot-Th{\'e}l{\`e}ne--Saito
\cite{CT-Saito} showed that the Brauer--Manin pairing also has a non-zero right kernel
in general. An interesting question to ask now is: how large these kernels
are in general, and under what conditions, either of these kernels has the
expected size. 

Suppose that there is a regular scheme $\sX$ which is finite type, proper and flat
over $R$ whose generic fiber is $X$. In this case,
Saito--Sato showed in \cite{Saito-Sato-ENS} that the right kernel of
~\eqref{eqn:BM-pair-0} coincides with $\Br(\sX)$. The prime-to-$p$ part of this
result was earlier shown by Colliot-Th{\'e}l{\`e}ne--Saito \cite{CT-Saito}.
Regarding the left kernel of ~\eqref{eqn:BM-pair-0}, Saito showed in
\cite{Saito-Invent} that it is trivial modulo the maximal
prime-to-$p$ divisible subgroup of $\CH_0(X)$ if $\Char(k) = 0, \
H^2_\zar(X, \sO_X) = 0$ and $\Alb_X$ has potentially good reduction.

In \cite{Esnault-Wittenberg}, Esnault--Wittenberg obtained a substantial improvement
of Saito's theorem by showing that the result remains valid (in any characteristic)
when $\Alb_X$ has
potentially good reduction and $\sX$ is a semistable model of $X$ whose irreducible
components of the reduced closed fiber satisfy the Tate conjecture. The latter
condition is implied by $H^2_{\zar}(X,\sO_X)=0$ when $\Char(k)=0$. The aforesaid
result of Parimala--Suresh implies that the assumption on $\Alb_X$ in the results of
Saito and Esnault--Wittenberg can not be removed.

In \cite{Yamazaki-Nagoya}, Yamazaki showed that, when $\Char(k)=0$, the Brauer--Manin
pairing extends to open varieties. More precisely, if $U$ is a smooth quasi-projective
variety over $k$, then evaluating Brauer classes at the closed points of $U$ defines a
pairing
\begin{equation}\label{eqn:BM-pair-1}
  H^S_0(U) \times \Br(U) \to {\Q}/{\Z},
\end{equation}
where $H^S_0(U)$ is the Suslin homology of $U$ (see \cite[Defn.~10.8]{MVW}, where
it is called the 0-th algebraic singular homology,
see also \cite[Thm.~5.1]{Schmidt-ANT}).
When $\Char(k) = 0$ and $U$ is a curve,  Scheiderer-van Hamel \cite{Sh-Hamel}
constructed the pairing ~\eqref{eqn:BM-pair-1} 
before Yamazaki considered the general case. They also showed that
~\eqref{eqn:BM-pair-1} has trivial left and right kernels for open curves.
The positive characteristic version of \cite{Sh-Hamel} was
established in \cite{KRS}.

In higher dimensions, the only known result is due to
Yamazaki in characteristic zero (cf. \cite[Thm.~6.3]{Yamazaki-Nagoya}).
He proved that the pairing~\eqref{eqn:BM-pair-1} has
trivial left kernel modulo the maximal divisible subgroup of $H^S_0(U)$, provided that
$\dim(U)=2$ and $U$ admits a smooth compactification $X$ such that $X$ is
geometrically rational and the irreducible components of $\ov{X} \setminus \ov{U}$
generate the N{'e}ron--Severi group of $\ov{X}$, where $\ov{X}$ is the base change of
$X$ by an algebraic closure of $k$. Beyond this, no general results are
known about the pairing~\eqref{eqn:BM-pair-1}.
Our goal in this paper is to extend Yamazaki's construction to positive characteristic,
and to investigate whether the known results for smooth projective surfaces continue
to hold for open surfaces.

We now state the new results concerning the kernels of
~\eqref{eqn:BM-pair-1}. We begin with the left kernel.
Let $R$ be an excellent henselian discrete valuation ring whose residue field $\F$ is
either finite or separably closed of exponential characteristic $p > 0$. Let $k$ denote
the quotient field of $R$. Let $X$ be a connected, smooth, projective surface
over $k$ and let $U \subset X$ be any open surface.

In this paper, we define an idele class group $C_0(U)$. When $U$ is projective, this
group coincides with $\CH_0(U)$. Moreover, there is a canonical surjection
$C_0(U) \surj H^S_0(U)$ whose kernel is prime-to-$\Char(k)$ divisible
(cf. \propref{prop:Tame-part}).
As with the Chow group of zero-cycles, $C_0(U)$ admits a canonical cycle class map
to {\'e}tale cohomology with compact support.
We fix a prime $\ell \neq p$. 
Our main results on the Brauer--Manin pairing for $U$ and its left kernel are as
follows.

\begin{thm}\label{thm:Main-4}
  Assume that $\F$ is finite. Then we have the following.
  \begin{enumerate}
  \item
    There exists a pairing $C_0(U) \times \Br(U) \to {\Q}/{\Z}$.
  \item
    Assume that the $(\star\star_\ell)$-condition holds.
    Then the cycle class map
  \[
  \cyc_U \colon C_0(U)^{(\ell)} \to H^4_{{{\rm {\etl}}},c}(U, \Z_\ell(2))
  \]
  is injective. In particular, the Brauer--Manin pairing for $U$ has trivial left
  kernel modulo the maximal $\ell$-divisible subgroup of $C_0(U)$.
  \end{enumerate}
\end{thm}

\begin{remk}\label{remk:Main-4-char-0}
  As stated above, item (1) of \thmref{thm:Main-4} was shown by Yamazaki 
  \cite{Yamazaki-Nagoya} when $\Char(k) = 0$.
\end{remk}

\begin{thm}\label{thm:Main-5}
  Assume that $\F$ is separably closed and $H^2_{\rm {\etl}}(\ov{X}, {\Q_\ell}(1))$ is
  algebraic. Then the cycle class map
  \[
  \cyc_U \colon C_0(U)^{(\ell)}  \to H^4_{{\rm {\etl}},c}(U, \Z_\ell(2))
  \]
  is injective.
\end{thm}

\begin{remk}\label{remk:Main-5-char-0}
  When $\F$ is separably closed, $\Char(k) = 0$ and $H^2_\zar(X, \sO_X) = 0$, then the
algebraicity condition in \thmref{thm:Main-5} is satisfied. In particular, the theorem
applies in this case. 
\end{remk}


\subsection{Cycle class map for normal projective surfaces}\label{sec:BML-Nor}
 For a singular variety $X$
over a field, let $\CH^{\lw}_0(X)$ denote the Levine--Weibel Chow group of 0-cycles on
$X$ (cf. \cite{Levine-Weibel}, \S~\ref{sec:Prf-2}, \S~\ref{sec:Normal-surface}).
This Chow group for normal projective varieties has been extensively
studied by several authors. Its structure over algebraically closed fields is
well understood (cf. \cite{BPW}, \cite{Krishna-Srinivas}).
If $X$ is a normal projective surface over a finite field, then Ghosh--Krishna
\cite{Ghosh-Krishna-Jussieu} showed that the cycle class map
${\CH^{\lw}_0(X)}/n \to H^4_{\et}(X, {\Z}/n(2))$ is an isomorphism of finite groups
if $n$ is invertible in the field. The nature of the cycle class map for normal
surfaces over other fields is yet unknown. 

As applications of Theorems~\ref{thm:Main-4} and ~\ref{thm:Main-5}, we obtain the
following extensions of the main results of Esnault--Wittenberg
\cite{Esnault-Wittenberg} to normal projective surfaces over local fields.

\begin{thm}\label{thm:Main-6}
  Assume that $\F$ is finite and let $X$ be a connected, normal projective surface
  over $k$. Then we have the following.
  \begin{enumerate}
  \item
    There exists a pairing $\CH^{\lw}_0(X) \times \Br(X_\reg) \to {\Q}/{\Z}$.
  \item
    Assume that $X$ admits a desingularization $\wt{X}$ such that the hypotheses of
    \thmref{thm:Main-4} are satisfied for $U = X_\reg \subset \wt{X}$.
    Then the cycle class map
  \[
  \cyc_X \colon \CH^{\lw}_0(X)^{(\ell)} \to H^4_{{\rm {\etl}}}(X, \Z_\ell(2))
  \]
  is injective. Equivalently, the Brauer--Manin pairing for $X$ has trivial left kernel
  modulo the maximal $\ell$-divisible subgroup of $ \CH^{\lw}_0(X)$.
  \end{enumerate}
\end{thm}

\begin{thm}\label{thm:Main-7}
  Assume that $\F$ is separably closed. Let $X$ be a connected, normal
  projective surface over $k$. Assume that $X$ admits a desingularization $\wt{X}$
  which satisfies the
  hypothesis of \thmref{thm:Main-5}. Then the cycle class map
  \[
  \cyc_X \colon \CH^{\lw}_0(X)^{(\ell)} \to H^4_{\rm {\etl}}(X, \Z_\ell(2))
  \]
  is injective.
\end{thm}

\begin{remk}\label{remk:Main-7-remk}
Assume that $\Char(k) = 0$ and $X$ is a connected, normal
projective surface over $k$ such that $H^2_\zar(X, \sO_X) = 0$.
One can then show that the hypothesis of \thmref{thm:Main-7} is satisfied, and
consequently the theorem applies in this setting.
\end{remk}

\subsection{Right kernel of Brauer--Manin pairing for open surfaces}\label{sec:BMR}
We shall state two results concerning the right kernel of ~\eqref{eqn:BM-pair-1}.
We assume in this subsection that  $\F$ is finite.
We let $D = X \setminus U$ with the reduced closed subscheme structure.
Let $\sD \subset \sX$ denote the closure of $D$ in $\sX$ with the reduced closed
subscheme structure. We let $\sU = \sX \setminus \sD$.
Then one checks that $\sU$ is the
largest open subscheme of $\sX$ whose generic fiber is $U$.
We let $\Br(U)^{\perp} = \Ker(\Br(U) \to \Hom(C_0(U), {\Q}/{\Z}))$ denote the
right kernel of ~\eqref{eqn:BM-pair-1}. Let
$A\{p'\}$ denote the prime-to-$p$ torsion subgroup of an abelian group $A$.

By the main theorem of \cite{Saito-Sato-ENS}, we
have $\Br(\sX) = \Br(X)^{\perp}$. In particular, there is an inclusion
$\Br(\sX) \subseteq \Br(U)^{\perp}$.
By Corollary~3.2.3 of op. cit. (when $\Char(k) = 0$) and
\cite[Thm.~1.1]{KM-0} (when $\Char(k) > 0$), there is an inclusion
$\Br(U)^{\perp} \subseteq \Br(\sU)$.
We remark here that in \cite[Cor.~3.2.3]{Saito-Sato-ENS}, it is assumed that
$\sU$ is proper over $R$. However, what it actually proves without the assumption of
properness of $\sX$ is that if the evaluation of a Brauer class in $X$ is zero
at every closed point in $X$ whose closure in $\sX$ is finite over $R$, then the
Brauer class lies in $\Br(\sU)$. This finiteness condition is automatic when $\sX$
is proper over $R$.

As a consequence of the above, we have the inclusions of Brauer groups
\begin{equation}\label{eqn:RK-0}
\Br(\sX) \subseteq \Br(U)^\perp \subseteq \Br(\sU).
\end{equation}

Our first result concerning the right kernel of the Brauer--Manin pairing of
$U$ is the following.
\begin{thm}\label{thm:Main-8}
  Each of the inclusions $\Br(\sX)\{p'\} \subseteq \Br(U)^\perp\{p'\} \subseteq
  \Br(\sU)\{p'\}$ is strict in general.
  \end{thm}

For the second result, we let $D^o \subset D$ be any dense open regular
subscheme, and  set $X^o = X \setminus (D \setminus D^o)$.
We let $\rho^\star_{D^o} \colon H^1_{\et}(D^o, {\Q}/{\Z}) \to \Hom(C_1(D^o), {\Q}/{\Z})$
denote the dual of the reciprocity map in the class field theory of
$D^o$ (cf. \cite[\S~4]{GKR}, \S~\ref{sec:Et-realn}). There is also a push-forward map
$\iota_* \colon  H^1_{\et}(D^o, {\Q}/{\Z}) \to H^3_\et(X^o, {\Q}/{\Z}(1))$,
where $\iota \colon D^o \inj X^o$ is the inclusion.
We define $H^1_{\et}(D^o)^{\perp} = \Ker(\rho^\star_{D^o}) \bigcap \Ker(\iota_*)$.

Unlike the case of $\Br(X)^{\perp}$, we do not know whether there is an explicit
description of $\Br(U)^{\perp}$. We propose the following.

\begin{conj}\label{conj:Main-9}
There exists a short exact sequence
  \[
  0 \to \Br(\sX)\{p'\} \to \Br(U)^{\perp}\{p'\} \to H^1_{{\rm {\etl}}}(D^o)^{\perp}\{p'\} \to 0.
  \]
  \end{conj}

\vskip.3cm

Our final result is the following evidence of this conjecture.

\begin{thm}\label{thm:Main-10}
  Assume that the $(\star\star_\ell)$-condition holds for every $\ell \neq p$.
  Then Conjecture~\ref{conj:Main-9} is true.
  \end{thm}

\subsection{Overview of proofs}\label{sec:Outline}
In this paper, Theorems~\ref{thm:Main-0} and~\ref{thm:Main-12} are at the heart of
the proofs of all our results concerning the Brauer--Manin pairing and the cycle
class map for open surfaces. Theorems~\ref{thm:Main-1} and~\ref{thm:Main-11}, in
turn, are at the heart of the proofs of Theorems~\ref{thm:Main-0} and
~\ref{thm:Main-12}. This explains how we were led to the question of the existence
of Gersten and Bloch--Ogus resolutions for the $K$-theory and {\'e}tale cohomology
sheaves on non-smooth schemes over discrete valuation rings of mixed characteristic.

The proofs of Theorems~\ref{thm:Main-1} and~\ref{thm:Main-11} are based on results
of Gillet--Levine \cite{Gillet-Levine}, Geisser \cite{Geisser-MZ}, and Sakagaito
\cite{Sakagaito} concerning the Gersten and Bloch--Ogus resolutions. We combine
these results with Hoobler's proof of the norm--residue isomorphism theorem for
arbitrary regular local rings \cite{Hoobler}, together with a comparison theorem
of Deligne \cite{Hodge-2} relating the coniveau and Leray filtrations on étale
cohomology. This ultimately reduces the
problem to a careful analysis of the maps in the Quillen localization sequence for
$K$-theory and the induced differentials in the Gersten complex.

To prove Theorems~\ref{thm:Main-0} and~\ref{thm:Main-12}, we use the localization
sequences for motivic and étale cohomology and compare them via the {\'e}tale
realization maps. We then combine the injectivity of the cycle class map for $X$,
proved by Esnault--Wittenberg \cite{Esnault-Wittenberg}, with the fact that the
cycle class map for the model $\sX$ is an isomorphism, as shown by Saito--Sato
\cite{Saito-Sato-Ann}. This reduces the proofs to establishing the existence of
Gersten resolutions for the Zariski sheaves  $\sK_2$ and ${\sK_2}/n$ on the model.
The latter follows from Theorems~\ref{thm:Main-1} and~\ref{thm:Main-11}.

To prove Theorems~\ref{thm:Main-4} and~\ref{thm:Main-5}, we first identify
the class group $C_0(U)$ with the motivic cohomology of $U$ with compact support.
This allows us to reduce the problem to proving the injectivity of the
{\'e}tale realization map. We achieve this by comparing the localization
sequences for
motivic and {\'e}tale cohomology and using Theorems~\ref{thm:Main-0} and
~\ref{thm:Main-12} to reduce to the case of projective surfaces. The latter case
is a result of Esnault--Wittenberg \cite{Esnault-Wittenberg}. 

We prove \thmref{thm:Main-10} by following a similar strategy and using
Theorem~\ref{thm:Main-0} once again.
An additional ingredient here is a class field theory for
regular, but not necessarily smooth, curves over local fields, which we establish in
this paper. Another key ingredient for proving \thmref{thm:Main-10}
is an extension of the duality theorem of
Saito \cite{Saito-Duality} to regular but not necessarily smooth schemes over
local fields, which we also establish in this paper.

We prove \thmref{thm:Main-1} in \S~\ref{sec:BFM}, \thmref{thm:Main-11} in
\S~\ref{sec:K-fin} and Theorems~\ref{thm:Main-0} and~\ref{thm:Main-12} in
\S~\ref{sec:Unram-V}. In \S~\ref{sec:Suslin-review}, we recall Suslin homology and
establish its identification with the motivic cohomology with compact support.
In \S~\ref{sec:BM-left}, we prove Theorems~\ref{thm:Main-4} and~\ref{thm:Main-5}.
In \S~\ref{sec:ST-duality}, we establish an extension of the Saito-Tate
duality \cite{Saito-Duality} to regular, not necessarily smooth, schemes over
local fields. 
We prove a class field theory for regular curves over local fields and
complete the proof of \thmref{thm:Main-10} in \S~\ref{sec:R-ker}.
Finally, we complete the proofs of Theorems~\ref{thm:Main-6},
~\ref{thm:Main-7} and ~\ref{thm:Main-8} in \S~\ref{sec:Remarks}.

\subsection{Notation}\label{sec:Notn}
For a Noetherian scheme $X$, we shall denote the category of
separated and finite type schemes over $X$ by $\Sch_X$.
We shall let $\Sm_X$ denote the full subcategory $\Sch_X$ consisting
of smooth schemes over $X$. If $A \to A'$ is a ring homomorphism and $X$ is an
$A$-scheme, we shall let $X_{A'}$ denote the base change $X \times_{\Spec(A)} \Spec(A')$.
 Unless mentioned otherwise, all intersections of the
closed subschemes of a scheme will be assumed to be scheme-theoretic intersections. 
A Zariski closed subset of a scheme will be assumed to be
endowed with the reduced closed subscheme structure unless we mention otherwise.

For a Noetherian scheme $X$, we shall let $X_\red$ denote the reduced part of $X$ and
$X_n$ denote the normalization of $X_\red$. We let $k(X)$ be the total ring of fractions
for $X$.
If $X$ is reduced, $X_\sing$ will denote the singular locus of $X$
and $X_\reg$ the complementary open subscheme.
We shall let $X^{(i)}$ (resp. $X_{(i)}$) denote
the set of codimension (resp. dimension) $i$ points on $X$. 
We let $\sC(X)$ denote the set of closed integral curves in $X$.
We let $X_\zar$ (resp. $X_\et$) denote the Zariski (resp.
{\'e}tale) site of $X$. 

For abelian groups $A, B$, the notation $A \otimes B$ will indicate
tensor product of $A$ and $B$ over $\Z$ and $\Hom(A,B)$ will indicate
the set of $\Z$-linear maps from $A$ to $B$.
For $m\in \N$, we let 
$A[m]$ (resp. $A/m$) denote the kernel (resp. cokernel)
of the multiplication map $A \xrightarrow{m} A$. For a prime $p$,
we shall let $A\{p\}$ (resp. $A\{p'\}$) denote the $p$-primary (resp. prime-to-$p$
primary) component of $A$. We let $A_\tor$ denote the torsion subgroup of $A$.
We let $A^\star = \Hom(A, {\Q}/{\Z})$. For any prime $\ell$, we let
$A^{(\ell)} = {\varprojlim}_n {A}/{\ell^n}$ denote the $\ell$-completion of $A$.

If $A$ is a topological abelian group, we let $A^\vee = \Hom_{\cont}(A, {\Q}/{\Z})$
denote the subgroup of $A^\star$ consisting of continuous homomorphisms,
where ${\Q}/{\Z}$ has discrete topology.
If $F$ is a functor from a category of schemes to another category, and $X = \Spec(A)$
is an affine scheme, we shall usually write $F(X)$ as $F(A)$.

\section{Bloch's formula for the regular model}\label{sec:BFM}
In this section, we recall the Gersten complex for the Zariski sheaf
$\sK_{2, \sX}$ on an essentially of finite type regular scheme over a discrete
valuation ring and prove some preliminary results about this complex.
We conclude this section with the proofs of
Theorems~\ref{thm:Main-1} and ~\ref{thm:Main-2}.
We begin by proving some $K$-theoretic results for regular local rings.

\subsection{Some results for regular local rings}\label{sec:Local*}
We will use the following general construction from \cite{Quillen}, together with some
related results, to study the exactness of the Gersten complex for regular
schemes over a discrete valuation ring (dvr).

Let $Y$ be a Noetherian scheme. For $i \ge 0$, we let $M^i(Y)$ denote the Serre
   subcategory of the abelian category of coherent sheaves on $Y$ which are supported
   on closed subsets of codimension at least $i$ in $Y$. By Quillen's
   localization and d{\'e}vissage theorems, there is a homotopy fiber sequence of
   $K$-theory spectra
   \begin{equation}\label{eqn:Quillen-0}
     K\left(M^{i+1}(Y)\right) \to K\left(M^i(Y)\right) \to
     K\left({M^i(Y)}/{M^{i+1}(Y)}\right)
   \end{equation}
   and a weak equivalence $K\left({M^i(Y)}/{M^{i+1}(Y)}\right)
\simeq {\underset{y \in Y^{(i)}}\coprod} K(k(y))$.

From the above, we get a commutative diagram of homotopy groups
\begin{equation}\label{eqn:Quillen-1}
\xymatrix@C1pc{
  & &  K_{j-1}\left({M^{i+1}(Y)}\right) \ar[d]^-{\psi_i} & \\
  K_j\left(M^{i-1}(Y)\right) \ar[r]^-{\phi_{i-1}} & K_j\left({M^{i-1}(Y)}/{M^{i}(Y)}\right)
  \ar[r] \ar[dr]_-{\partial_{i-1}} &
  K_{j-1}\left({M^{i}(Y)}\right) \ar[dr]^-{\delta_i} \ar[d]^-{\phi_i} & \\
  & &  K_{j-1}\left({M^{i}(Y)}/{M^{i+1}(Y)}\right) \ar@{->>}[r] \ar[d] &
  \coker(\partial_{i-1}) \\
  & & K_{j-2}\left({M^{i+1}(Y)}\right) &}
\end{equation}
for $i \ge 0$ and $j \in \Z$, in which the second row from the top and the third column
from the left are exact.

We let $Z_i(Y)$ denote the (free abelian) group of dimension $i$
cycles on $Y$ and $\CH_i(Y)$ its quotient by the rational equivalence
\cite[Chap.~1]{Fulton}. If $Y$ is equidimensional, we let $Z^i(Y)$ denote the
(free abelian group) of codimension $i$ cycles on $Y$ and $\CH^i(Y)$ its
quotient by the rational equivalence.

\begin{lem}\label{lem:Gersten-1}
  Let $A$ be a regular local ring. Then the map
  $\divf \colon {\underset{{\htt}(\fp) =1}\bigoplus} k(\fp)^\times
  \to  {\underset{{\rm ht}(\fp) =2}\bigoplus} \Z$ is surjective.
\end{lem}
\begin{proof}
The lemma is equivalent to the statement that $\CH^2(A) = 0$.
 To show the latter, we shall use ~\eqref{eqn:Quillen-1} with $Y = \Spec(A), \ j = 1$
 and $i = 2$. Since a Serre subcategory of an abelian category
is abelian, and the first negative $K$-theory of an abelian category is trivial
(cf. \cite{Schlichting-MZ}), it follows that $\delta_2$  in ~\eqref{eqn:Quillen-1}
is surjective. 

We let $W^2(A)$ denote the cokernel of $\partial_1$ in ~\eqref{eqn:Quillen-1}.
We let $C^2(A)$ denote the quotient of $Z^2(A)$ by
the subgroup generated by $\cyc(Z)$, where $J \subset A$ is a complete intersection
ideal of height two and $Z = \Spec(A/J)$. Here, $\cyc(Z)$ is the cycle associated to
the closed subscheme $Z \subset \Spec(A)$ in $Z^2(A)$ (cf. \cite[\S~1.5]{Fulton}).
These groups were defined by Claborn--Fossum \cite{Claborn-Fossum}
who denoted them by $W_2(A)$ (resp. $C_2(A)$).
They defined $W_i(A)$ and $C_i(A)$ more generally for every $i \ge 1$.

Nori showed that the canonical map $C_i(A) \to \CH^i(A)$ is an isomorphism
for all $i \ge 1$ (cf. \cite[Thm.~5.4]{Murthy-Survey}).
In particular, $C_1(A) \cong \Pic(A) = 0$. 
It follows therefore from \cite[Prop.~3.2]{Claborn-Fossum} that the identity map of
$Z^2(A)$ induces canonical isomorphisms
\begin{equation}\label{eqn:Gersten-1-0}
  C^2(A) \xrightarrow{\cong} W^2(A) \xrightarrow{\cong} \CH^2(A).
  \end{equation}

Using these isomorphisms, we are reduced to showing that $W^2(A) = 0$. By the
surjectivity of $\delta_2$, this is equivalent to showing that $\delta_2$ is the zero
map. But this follows from \cite[Thm.~5.2]{Smoke} which implies that
$K_0\left({M^{2}(A)}\right)$ (which is denoted by $K^2(A)$ in op. cit.) is generated by
elements $[A/J]$, where $J \subset A$ is a height two complete intersection ideal. 
In particular, $\delta_2([A/J]) = 0$ by the first isomorphism in
~\eqref{eqn:Gersten-1-0}. This completes the proof.
\end{proof}

\begin{lem}\label{lem:Gersten-2}
  Let $A$ be a regular local ring of dimension $d$. Then the map
  \[
  \phi_{d-1} \colon K_j\left(M^{d-1}(A)\right) \to {\underset{\htt(\fp) = d-1}\coprod}
  K_j(k(\fp))
  \]
  is injective for $j \le 2$.
  \end{lem}
\begin{proof}
By ~\eqref{eqn:Quillen-1}, the lemma is equivalent to the assertion that
  the map $\psi_{d-1} \colon K_j\left(M^d(A)\right) \to K_j\left(M^{d-1}(A)\right)$ is
  zero for $j \le 2$. To show this assertion, we let $\fm = (x_1, \ldots , x_d)$
  be the maximal ideal of $A$ and let $\fp = (x_1, \ldots , x_{d-1})$ if $d \ge 2$
  and $\fp = (0)$ if $d =1$. We set $A' = A/{\fp}$ and let $K'$ denote the
  quotient field of $A'$.
We now note by \cite[(5.1)]{Quillen} that $K_j\left(M^d(A)\right) \cong K_j(\ff)$
  and the map $\psi_{d-1}$ has a factorization by the push-forward maps
  $K_j(\ff) \to K_j(A') \to K_j\left(M^{d-1}(A)\right)$, where $\ff$ is the residue field
  of $A$. It suffices therefore to show that the push-forward map
  $K_j(\ff) \to K_j(A')$ is zero for $j \le 2$.

We let $Y = \Spec(A')$ and $Y^o = Y \setminus \{y\}$, where $y \in Y$ is the closed
  point. By the Quillen localization theorem, it suffices to show that the
  map $K_j(A') \to K_j(Y^o)$ is injective for $j \le 2$. But this is true because 
  $A'$ is a dvr, and therefore, the composite map
  $K_j(A') \to K_j(Y^o) \to K_j(K')$ is injective for $j \le 2$ by
  \cite[Chap.~V, Cor.~6.6.2]{Weibel-K-book}. This completes the proof.
\end{proof}

\begin{cor}\label{cor:Gersten-3}
  Let $A$ be as in \lemref{lem:Gersten-2} and let $j \le 3$. Then we have an
  exact sequence
  \[
  K_j\left(M^{d-2}(A)\right) \xrightarrow{\phi_{d-2}}
  {\underset{\htt(\fp) = d-2}\coprod}
   K_{j}(k(\fp)) \xrightarrow{\partial_{d-2}} {\underset{\htt(\fp) = d-1}\coprod}
   K_{j-1}(k(\fp)).
   \]
\end{cor}
\begin{proof}
  Follows immediately from ~\eqref{eqn:Quillen-1} and \lemref{lem:Gersten-2}.
  \end{proof}

\subsection{Gersten complex in general}\label{sec:Gersten}
Let $Y$ be a Noetherian separated scheme. We let $K'_i(Y)$ (resp. $K_i(Y)$) denote
the Quillen $K$-theory of the exact category of coherent (resp. finite rank locally
free)
sheaves on $Y$.  Recall that the canonical map $K_i(Y) \to K'_i(Y)$ is an isomorphism
for every $i \in \Z$ when $Y$ is regular. We let $\sK'_{i,Y}$ denote the Zariski sheaf
on $Y$ associated to the presheaf $U \mapsto K'_i(U)$. We define the sheaf $\sK_{i,U}$
similarly. For $i \ge 0$, we let $\sK^M_{i,Y}$ denote the Milnor $K$-theory sheaf,
and let $\wh{\sK}^M_{i,Y}$ denote the improved Milnor $K$-theory sheaf introduced in
\cite{Kerz-JAG}.

The following result is \cite[Lem.~4]{Kerz-Strunk} which we shall use
frequently.
\begin{lem}\label{lem:Kerz-Strunk}
    Let $r\ge 0$ be an integer. Let be $\sF$ a Zariski sheaf on $Y$ such that
    $\sF_x = 0$ for all points $x \in Y$ with $\dim\left(\ov{\{x\}}\right) > r$. Then
    $H^i_\zar(Y, \sF) = 0$ for $i > r$.
    \end{lem}

By \lemref{lem:Kerz-Strunk} and a combination of
Prop.~10, Thm.~13 and \cite[Prop.~14]{Kerz-JAG}, we get the following.

\begin{lem}\label{lem:Kerz-0}
For all $j \ge 0$, there are canonical maps of Zariski sheaves
\begin{equation}\label{eqn:K-groups}
  \sK^M_{j,Y} \surj \wh{\sK}^M_{j,Y} \to \sK_{j,Y},
  \end{equation}
in which the second map is an isomorphism for $j \le 2$.
The first map is an isomorphism if either $Y$ is a spectrum of a field or
if the residue fields of $Y$ are infinite. In particular,
the canonical maps $H^i_\zar(Y, \sK^M_{j,Y}) \to H^i_\zar(Y, \wh{\sK}^M_{j,Y}) \to
H^i_\zar(Y, \sK_{j,Y})$ are isomorphisms for $j \le 2$ and $i \ge 1$. 
\end{lem}

By \lemref{lem:Kerz-0} and
\cite[Prop.~1]{Kato-contem-86}, we have a chain complex of Zariski sheaves
\begin{equation}\label{eqn:Gersten-0}
  \sK_{2,Y} \xrightarrow{\epsilon} {\underset{x \in Y^{(0)}}\bigoplus} i_{x *} K_2(k(x))
    \xrightarrow{\partial_0} {\underset{x \in Y^{(1)}}\bigoplus} i_{x *} k(x)^\times
      \xrightarrow{\partial_1} {\underset{x \in Y^{(2)}}\bigoplus} i_{x *} \Z \to 0,
\end{equation}
where $\epsilon$ is the canonical pull-back map, $\partial_0$ is the Tame symbol
and $\partial_1$ is the valuation map. The latter coincides with the map which
locally associates to a rational function on a codimension one integral closed
subscheme of $Y$, its divisor class on the closed subscheme.
We shall therefore write $\partial_1(f)$ also as $\divf(f)$.

The claim that $\partial_1 \circ \partial_0 = 0$ in ~\eqref{eqn:Gersten-0}
can be checked directly (e.g., see \cite[Prop.~1]{Kato-contem-86}).
To see that $\partial_0 \circ \epsilon$ is zero, we first note that $\epsilon$ factors
through $\sK_{2, Y_n}$, where $Y_n$ is the normalization of $Y_\red$
(cf. \S~\ref{sec:Notn}). This allows
us to assume that $Y$ is normal and integral. Now, to check that
$\partial_0 \circ \epsilon =0$, we can check it by localizing $Y$ at each
codimension one point, which allows us to further assume that $Y$ is the spectrum
of a dvr. In the latter case, the claim follows from
the definition of the Tame symbol. We shall refer to ~\eqref{eqn:Gersten-0} as the
Gersten complex for $\sK_{2,Y}$. Our goal in \S~\ref{sec:BFM} is to investigate the
exactness of this complex in a special case. One case where this is already known is
the following.

\begin{lem}\label{lem:Gersten-4}
  Let $A$ be a regular local ring which is smooth and essentially of
  finite type algebra over a dvr. Let $K$ denote the quotient
  field of $A$. Then the Gersten complex
  \[
  0 \to K_j(A) \to K_j(K) \to  {\underset{\htt(\fp) = 1}\coprod}
   K_{j-1}(k(\fp)) \to {\underset{\htt(\fp) = 2}\coprod} K_{j-2}(k(\fp)) \to 0
   \]
   is exact for $j \le 2$.
\end{lem}
\begin{proof}
  By \cite[Cor.~6]{Gillet-Levine} and its proof (cf. ~\eqref{eqn:Quillen-1}),
  one only has to verify the lemma when $A$ is a dvr.
  But the latter case is obvious when $j \le 1$, and follows from
  \cite[Chap.~V, Cor.~6.6.2]{Weibel-K-book} when $j =2$.
  \end{proof}

\subsection{The Gersten complex for a regular model}\label{sec:Gersten-model}
For the rest of \S~\ref{sec:BFM}, the following is our setting.
We let $R$ be an excellent dvr with quotient field $k$ and residue
field $\F$ (which is allowed to be arbitrary).
Let $\sX$ be a connected, regular, essentially of
finite type $R$-scheme of Krull dimension three. We assume that
$\sX$ is faithfully flat over $R$ and its reduced closed fiber $Y$ is a simple normal
crossings (snc) divisor on $\sX$. By the latter condition, we shall mean that
for any finite collection of pairwise distinct irreducible components
$Y_1, \ldots , Y_r$ of $Y$, the scheme-theoretic intersection
${\underset{1 \le i \le r}\bigcap} Y_i$ is a smooth $\F$-scheme and has
codimension $r-1$ in $Y$ (cf. \cite[\S~1]{Esnault-Wittenberg}).

We let $X$ denote the generic fiber of $\sX$.
We let $\eta$ denote the generic point of $\sX$ and write $K = k(\eta)$.
We let $\{Y_1, \ldots , Y_m\}$ be the set of irreducible components of $Y$ and let
$\Sigma = {\underset{i \neq j \neq l \neq i}\bigcup} (Y_i \bigcap Y_j \bigcap Y_l)$.
Then $\Sigma$ is a finite set of closed points in $Y_\sing$.
We let $\sX^o = \sX \setminus \Sigma$ and $\sX_\sm = \sX \setminus Y_\sing$
so that there are open immersions $\sX_\sm \subset \sX^o \subset \sX$.
Note that $\sX_\sm$ may not be a smooth $R$-scheme.

We write out the Gersten complex (cf. ~\eqref{eqn:Gersten-0}) separately for $\sX$
as we shall need it repeatedly.
\begin{equation}\label{eqn:Quillen-2}
  0 \to \sK_{2,\sX} \xrightarrow{\epsilon} i_{\eta *} K_2(k(\eta))
  \xrightarrow{\partial_0}
  {\underset{x \in \sX^{(1)}}\bigoplus} i_{x *} k(x)^\times
      \xrightarrow{\partial_1} {\underset{x \in \sX^{(2)}}\bigoplus} i_{x *} \Z \to 0
\end{equation}

We begin with the following result.

\begin{lem}\label{lem:Panin*}
  The Gersten complex ~\eqref{eqn:Gersten-0} is exact for $X$ and $Y_\reg$.
If $R$ has equal characteristic, then ~\eqref{eqn:Quillen-2}
is exact even if $X$ is not smooth over $k$ or $Y$ is not an snc divisor on $\sX$.
\end{lem}
\begin{proof}
  See \cite[Thm.~A]{Panin} (see also \cite[Prop.~10, 14]{Kerz-JAG}).
\end{proof}

In view of this lemma, we shall assume in this subsection that
$\Char(k) = 0$ and $\Char(\F) = p > 0$. This implies in particular that
$X$ is a smooth $k$-scheme and $\sX_\sm$ is a smooth $R$-scheme. 
We fix a point $x \in \sX$ and let $A = \sO_{\sX,x}, \ \sX_x = \Spec(A), \
\sX^o_x = \sX_x \times_R \sX^o, \ \sU_x = \sX_x \times_R \sX_\sm$ and
$Z_x = \sX_x \setminus \sU_x$.

\begin{lem}\label{lem:Gersten-5}
  The complex
  \begin{equation}\label{eqn:Gersten-5-6}
  K_2(K) \xrightarrow{\partial_0}
  {\underset{\htt(\fp) =1}\bigoplus} k(\fp)^\times
  \xrightarrow{\partial_1} {\underset{\htt(\fp) = 2}\bigoplus} \Z
  \end{equation}
  is exact in the middle, where the prime ideals are taken in $A$.
\end{lem}
\begin{proof}
  If $x \in (Y_\reg)_{(0)}$, then $x \in \sX_\sm$ and we can apply
  \lemref{lem:Gersten-4} to conclude.
  We now consider the second case in which we assume $x \in \Sigma$.
  Since $Y$ is an snc divisor on
  $\sX$, there exist in this case, exactly three distinct irreducible components
  $C_1, C_2$ and $C_3$ of $Y_\sing$ such that
  $\{x\} = \sX_x \times_R (C_1 \bigcap C_2 \bigcap C_3)$. Let $\eta_i$ denote the
  generic point of $C_i$.

 We look at the morphism of chain complexes
  \begin{equation}\label{eqn:Gersten-5-0}
    \xymatrix@C1.6pc{
      K_2(K) \ar[r]^-{\partial_0} \ar[d]_-{\id} & 
      {\underset{\htt(\fp) =1}\bigoplus} k(\fp)^\times \ar[r]^-{\partial_1}
      \ar[d]^-{\id} & 
      {\underset{\htt(\fp) = 2}\bigoplus} K_0(k(\fp)) \ar[d] \\
      K_2(K) \ar[r]^-{\partial_0|_{\sU_x}} &
      {\underset{y \in \sU^{(1)}_x}\bigoplus} k(y)^\times \ar[r]^-{\partial_1|_{\sU_x}} &
      {\underset{y \in \sU^{(2)}_x}\bigoplus} K_0(k(y)),}
    \end{equation}
 where the vertical arrows are obtained by restriction via the open immersion
  $\sU_x \inj \sX_x$.

The left and the middle vertical arrows in the above diagram are bijective while the
  right vertical arrow is surjective whose kernel is
  $\stackrel{3}{\underset{i = 1}\bigoplus} K_0(k(\eta_i))$. 
  Since $\partial_1$ is surjective by \lemref{lem:Gersten-1}, it follows that
  $\partial_1|_{\sU_x}$ is also surjective. We therefore get a commutative diagram

  \begin{equation}\label{eqn:Gersten-5-1}
    \xymatrix@C1.6pc{
      0 \ar[r] & {\rm Image}(\partial_0) \ar[r] \ar[d]_{\cong} & \Ker(\partial_1) \ar[r]
      \ar@{^{(}->}[d] & F \ar[r] \ar[d] & 0 \\
      0 \ar[r] & {\rm Image}(\partial_0|_{\sU_x}) \ar[r] & \Ker(\partial_1|_{\sU_x})
      \ar[r] & H^1_\zar(\sU_x, \sK_{2, \sU_x}) \ar[r] & 0,}
    \end{equation}
  where $F$ is defined so that the top row becomes exact. The bottom row can be
  seen to be exact by
  considering the sheafified version of \lemref{lem:Gersten-4}.

  By a diagram chase of
  ~\eqref{eqn:Gersten-5-0} and  ~\eqref{eqn:Gersten-5-1}, we obtain an exact
  sequence
 \begin{equation}\label{eqn:Gersten-5-2} 
   0 \to F \to H^1_\zar(\sU_x, \sK_{2, \sU_x}) \xrightarrow{\partial} \
   \stackrel{3}{\underset{i = 1}\bigoplus} K_0(k(\eta_i)) \to 0.
   \end{equation}

We now recall that the Brown-Gersten spectral sequence (cf. \cite[Rem.~3.5]{Quillen},
\cite[Thm.~10.3]{TT}) provides a natural map
$K_1(\sU_x) \to H^0_\zar(\sU_x, \sK_{1, \sU_x})$ which is a split surjection
(cf. the proof of \cite[Lem.~2.1]{Krishna-ANT}). Letting $SK_1(\sU_x)$ denote the
kernel of this map, another application of the same spectral sequence yields a
natural surjection $SK_1(\sU_x) \surj H^1_\zar(\sU_x, \sK_{2, \sU_x})$.

We next look at the diagram
 \begin{equation}\label{eqn:Gersten-5-3}
   \xymatrix@C1pc{
 & & 0 \ar[d] & & \\     
& & K_1(A) \ar[r]^-{\cong} \ar[d] & A^\times \ar[d]^-{\cong} & \\
0 \ar[r] & SK_1(\sU_x) \ar[r]  \ar[d]_-{\partial'} & K_1(\sU_x) \ar[r] \ar[d] &
H^0_\zar(\sU_x, \sK_{1, \sU_x}) \ar[r] & 0 \\
& \stackrel{3}{\underset{i = 1}\bigoplus} K_0(k(\eta_i)) & K'_0(Z_x) \ar[d]
\ar[l]_-{j^*} & & \\
& & 0 & &}
 \end{equation}

The middle column is the localization  exact sequence in which the 
map $K_1(\sU_x) \to K'_0(Z_x)$ is surjective because the next map
$K_0(A) \to K_0(\sU_x)$ in the long exact sequence is clearly injective.
The middle long row is exact as explained above.
The right vertical arrow is an isomorphism by the Hartogs Lemma
 (cf. \cite[Lem.~4.4]{Ghosh-Krishna-Jussieu}) because $Z_x$ has codimension at least
 two in $\sX_x$. A diagram chase shows that the composite map
 $SK_1(\sU_x) \to  K_1(\sU_x) \to  K'_0(Z_x)$ is an isomorphism.

We now claim that the map $K'_0(Z_x) \to
 \stackrel{3}{\underset{i = 1}\bigoplus} K_0(k(\eta_i)) = K_0(k(Z_x))$,
 induced by the flat pull-back to the generic points, is an isomorphism.
To show this, we first note that $(Z_x)_n = \stackrel{3}{\underset{i = 1}\amalg}
\sX_x \times_R C_i$ is regular since $Y$ is an snc divisor on $\sX$.
We therefore get isomorphisms
\begin{equation}\label{eqn:Gersten-5-4}
  K_0((Z_x)_n) \cong \ \stackrel{3}{\underset{i = 1}\bigoplus}K_0(\sX_x \times_R C_i)
  \xrightarrow{\cong}
 \ \stackrel{3}{\underset{i = 1}\bigoplus} K_0(k(\eta_i)) \cong  K_0(k(Z_x)).
\end{equation}

We now let $\pi \colon (Z_x)_n \to Z_x$ be the normalization map. Let
$E = \pi^{-1}((Z_x)_\sing) = \pi^{-1}(\{x\})$ and $Z^o_x = Z_x \setminus \{x\} =
(Z_x)_\reg$. Since $\Spec(k((Z_x)_n)) = \pi^{-1}(Z^o_x) \to Z^o_x = \Spec(k(Z_x))$
is an isomorphism, the commutativity of the localization sequence with the
proper forward maps (cf. \cite[Prop.~3.18]{TT})
yields a commutative diagram of localization sequences
\begin{equation}\label{eqn:Gersten-5-5}
 \xymatrix@C1pc{
   K_0(E) \ar[r]^-{\iota'_*} \ar[d]_-{\pi_*} & K_0((Z_x)_n) \ar[r]^-{j'^*}
   \ar[d]^-{\pi_*} & K_0(k((Z_x)_n)) \ar[r] \ar[d]^-{\pi_*} & 0 \\
   K_0(k(x)) \ar[r]^-{\iota_*} & K'_0(Z_x) \ar[r]^-{j^*} & K_0(k(Z_x)) \ar[r] & 0.}
\end{equation}

Since $E = \stackrel{3}{\underset{i = 1}\amalg} \Spec(k(x))$ and
$\pi \colon E \to \Spec(k(x))$ is the projection, the left vertical arrow in
~\eqref{eqn:Gersten-5-5} is the summation map $\Z^3 \to \Z$. In particular, it is
surjective. Since $(Z_x)_n$ is a disjoint union of the spectra of three discrete
valuation rings, the pull-back map $j'^*$ is an isomorphism. In particular, $\iota'_*$
is zero. This implies that $\iota_*$ is also zero, i.e., $j^*$ is an isomorphism.
Returning back to ~\eqref{eqn:Gersten-5-3}, we conclude that the boundary map
$\partial'$ is an isomorphism. In particular, $SK_1(\sU_x) \cong \Z^3$.

Putting everything together, we find that there are surjective  maps
\begin{equation}\label{eqn:Gersten-5-7}
SK_1(\sU_x) \surj H^1_\zar(\sU_x, \sK_{2, \sU_x}) \stackrel{\partial}{\surj} \
\stackrel{3}{\underset{i = 1}\bigoplus} K_0(k(\eta_i))
\end{equation}
between abelian groups in which the first and the third groups are free and have
the same rank (three in this case).
An easy exercise then shows that both maps must be isomorphisms. We have thus shown
that the map $\partial$ in ~\eqref{eqn:Gersten-5-2} is bijective. Equivalently,
$F = 0$. By ~\eqref{eqn:Gersten-5-1}, this is equivalent to saying that 
~\eqref{eqn:Gersten-5-6} is exact in the middle.

The third case is when $x \in (Y_\sing)_{(0)} \setminus \Sigma$. In this case,
there exist unique irreducible components $Y_i \neq Y_j$ of $Y$ such that
$Y_i \bigcap Y_j = C_1 \amalg \cdots \amalg C_s$ is a disjoint union of
irreducible curves and $x \in C_1$. This implies that $Z_x = \sX_x \times_R C_1$
is regular. Letting $\eta_1$ denote the generic point of $C_1$ and repeating the
above proof verbatim, we arrive at the same situation as in
~\eqref{eqn:Gersten-5-7} (except that the rank of the corner free abelian groups is
now one). We conclude again that ~\eqref{eqn:Gersten-5-6} is exact in the middle.

The fourth case to consider is when $x \in \sX^{(2)}$. If $x \in X \bigcup Y_\reg$, then
the exactness of ~\eqref{eqn:Gersten-5-6} follows from
Lemmas~\ref{lem:Gersten-4} and ~\ref{lem:Panin*}.
If $x \in Y_\sing$, then $x$ is the generic point of
an irreducible component $C$ of $Y_\sing$ which is regular and $Z_x = \{x\}$.
We repeat the argument of the second case to get surjective maps
as in  ~\eqref{eqn:Gersten-5-7}, where now the first map is an isomorphism
and all groups are isomorphic to $\Z$. It follows that $\partial$ is bijective whence
$F = 0$. We conclude that ~\eqref{eqn:Gersten-5-6} is exact in the middle.
Finally, we are left with the case $x \in \sX^{(i)}$ with $i \le 1$ when the
desired exactness follows from \lemref{lem:Gersten-4}.
This completes the proof of the lemma.
\end{proof}

\subsection{Proof of Theorem~\ref{thm:Main-1}}\label{sec:Prf-1}
We continue to work in the setting of \S~\ref{sec:Gersten-model} but do not assume
anymore that $\Char(k) =0$ or $\Char(\F) >0$.
We let $\ov{\sK}_{2,\sX}$ denote the image of the map $\epsilon \colon \sK_{2,\sX} \to
i_{\eta *} K_2(k(\eta))$ and let $\sK^{\nr}_{2,\sX}$ denote the kernel of the
map $\partial_0 \colon i_{\eta *} K_2(k(\eta)) \to
{\underset{x \in \sX^{(1)}}\bigoplus}  i_{x*} k(x)^\times$ in the Gersten complex
~\eqref{eqn:Quillen-2} for $\sX$. On $\sX$, we then have the maps of Zariski sheaves
\begin{equation}\label{eqn:Gersten-8}
  {\sK}_{2,\sX} \surj \ov{\sK}_{2,\sX} \inj \sK^{\nr}_{2,\sX}
  \end{equation}
and a complex of Zariski sheaves
\begin{equation}\label{eqn:Gersten-9}
  0 \to \sK^{\nr}_{2,\sX} \to i_{\eta *} K_2(k(\eta)) \xrightarrow{\partial_0}
  {\underset{x \in \sX^{(1)}}\bigoplus} i_{x*} k(x)^\times \xrightarrow{\partial_1}
  {\underset{x \in \sX^{(2)}}\bigoplus} i_{x*} \Z \to 0.
\end{equation}

By combining Lemmas~\ref{lem:Gersten-1}, ~\ref{lem:Panin*} and ~\ref{lem:Gersten-5},
we get the following.

\begin{thm}\label{thm:Gersten-10}
  The complex ~\eqref{eqn:Gersten-9} is a flasque resolution of $\sK^{\nr}_{2,\sX}$.
\end{thm}

\begin{lem}\label{lem:Bloch-0}
  The canonical map $H^i_\zar(\sX, {\sK}_{2,\sX}) \to  H^i_\zar(\sX, {\sK}^{\nr}_{2,\sX})$
  is surjective for $i = 1$ and bijective for $i \ge 2$. In particular,
  $H^i_\zar(\sX, \sK_{2, \sX}) = 0$ for $i \ge 3$.
\end{lem}
\begin{proof}
  In view of ~\eqref{eqn:Quillen-2}, we can assume that $\Char(k) = 0$ and
  $\Char(\F) >0$.
  If we now let $\sF$ denote the kernel of ${\sK}_{2,\sX} \surj \ov{\sK}_{2,\sX}$ and
  $\sG$ denote the cokernel of $\ov{\sK}_{2,\sX} \inj \sK^{\nr}_{2,\sX}$, then
  $\sF_x = 0 = \sG_x$ for every $x \in \sX_{(i)}$ with $i \ge 2$ by
  \lemref{lem:Gersten-4}. Furthermore, $\sG_x = 0$ for every $x \in \sX_{(1)}$
  by \corref{cor:Gersten-3}. We now apply
  \lemref{lem:Kerz-Strunk} to get $H^i_\zar(\sX, \sF) = 0$ for $i \ge 2$ and
  $H^i_\zar(\sX, \sG) = 0$ for $i \ge 1$. From this, we deduce
  $H^1_\zar(\sX, {\sK}_{2,\sX}) \surj  H^1_\zar(\sX, \ov{\sK}_{2,\sX})
  \surj  H^1_\zar(\sX, {\sK}^{\nr}_{2,\sX})$ and
  $H^i_\zar(\sX, {\sK}_{2,\sX}) \xrightarrow{\cong}  H^i_\zar(\sX, \ov{\sK}_{2,\sX})
  \xrightarrow{\cong}  H^i_\zar(\sX, {\sK}^{\nr}_{2,\sX})$ for $i \ge 2$.
  The last part of the lemma now follows from \thmref{thm:Gersten-10}.
  \end{proof}

The following result proves \thmref{thm:Main-1}.

\begin{cor}\label{cor:Bloch-2}
  There are canonical isomorphisms
  \[
  H^2_\zar(\sX, {\sK}^M_{2,\sX}) \xrightarrow{\cong} 
  H^2_\zar(\sX, \wh{\sK}^M_{2,\sX}) \xrightarrow{\cong}
  H^2_\zar(\sX, {\sK}_{2,\sX}) \xrightarrow{\cong} \CH_1(\sX).
  \]
\end{cor}
\begin{proof}
  The first two isomorphisms follow from \lemref{lem:Kerz-0}.
  This last isomorphism follows directly from \lemref{lem:Bloch-0} since it is clear
  from \thmref{thm:Gersten-10} and ~\eqref{eqn:Gersten-9} that
  \begin{equation}\label{eqn:Bloch-2-0}
  H^2_\zar(\sX, {\sK}^{\nr}_{2,\sX}) \cong \coker
  \left({\underset{x \in \sX^{(1)}}\bigoplus} k(x)^\times \xrightarrow{\divf}
       {\underset{x \in \sX^{(2)}}\bigoplus} \Z\right) = \CH_1(\sX).
       \end{equation}
       \end{proof}

\vskip.2cm

\subsection{Proof of \thmref{thm:Main-2}}\label{sec:Prf-2} 
Let the setting be as in \S~\ref{sec:Gersten-model}.
We begin by recalling the modified version of the Levine--Weibel Chow group of $Y$
from \cite{Binda-Krishna-Comp}.
Let $C$ be a reduced curve over $\F$. One says that $C$ is a good curve
(relative to $Y_\sing$) if there is a finite local complete
intersection (lci) morphism $\nu \colon C \to Y$ such that $\nu^{-1}(Y_\sing)$ is
nowhere dense in $C$. For a good curve $C$, we let $k(C, \nu^{-1}(Y_\sing))^{\times}$
be the image of the natural map
$\sO^{\times}_{C,S} \to \stackrel{s}{\underset{i =1}\oplus}
\sO^{\times}_{C, \eta_i}$, where $\{\eta_1, \ldots , \eta_s\}$ is the set of generic
points of $C$ and $S$ is the union of the closed subset $\nu^{-1}(Y_\sing)$ and the
set of generic points $\eta_i$ of $C$ such that $\ov{\{\eta_i\}}$ is disjoint from
$\nu^{-1}(Y_\sing)$. 

The lci Chow group of 0-cycles $\CH^{\Lci}_0(Y)$ is the quotient of
$Z_0(Y_\reg)$ by the subgroup $R^{\Lci}_0(X)$ generated by $\nu_*(\divf(f))$, where
$\nu \colon C \to Y$ is a good curve and $f \in k(C, \nu^{-1}(Y_\sing))^{\times}$.
When we replace $R^{\Lci}_0(X)$ by the subgroup generated by $\nu_*(\divf(f))$, where
$\nu \colon C \inj Y$ is an inclusion of a good curve, the corresponding quotient
is the original Levine--Weibel Chow group $\CH^{\lw}_0(Y)$, introduced in
\cite{Levine-Weibel}. We thus have surjections
$\CH^{\lw}_0(Y) \surj \CH^{\Lci}_0(Y) \stackrel{\pi^*}{\surj} \CH_0(Y_n)
\stackrel{\pi_*}{\surj} \CH_0(Y)$,
where $\pi \colon Y_n \to Y$ is the normalization map.

Let $Z^g_1(\sX)$ denote the subgroup of $Z_1(\sX)$ generated by integral cycles
which are finite (hence flat) over $R$ and do not meet $Y_\sing$.
Using \corref{cor:Bloch-2} and ~\eqref{eqn:Bloch-2-0}, we get a diagram
\begin{equation}\label{eqn:Rest-0}
  \xymatrix@C1pc{
    \sZ^g_1(\sX) \ar[d]_-{\wt{\rho}} \ar[r]^-{\cong} &
       {\underset{C \in Z^g_1(\sX)}\bigoplus}
       H^2_C(\sX, \sK^{\nr}_{2,\sX}) \ar[d] \ar[r] &
       H^2_\zar(\sX, \sK^{\nr}_{2,\sX})
       & H^2_\zar(\sX, \sK_{2,\sX}) \ar[l]_-{\cong} \ar[d]^-{\rho'} \\
Z_0(Y_\reg) \ar[r]^-{\cong} & {\underset{y \in Y^{(2)}_\reg}\bigoplus}
H^2_y(Y, \sK_{2,Y}) \ar[rr] & & H^2_\zar(Y, \sK_{2,Y}).}
\end{equation}

The left horizontal arrow in the top row is an isomorphism by \thmref{thm:Gersten-10}
and the middle horizontal arrow is the sum of the forget support maps.
The right horizontal arrow on the top is an isomorphism by \lemref{lem:Bloch-0}.
The left horizontal arrow in the bottom row is
an isomorphism by using excision and Gersten resolution for $\sK_{2, Y_\reg}$ and the
right horizontal arrow is the sum of the forget support maps. The arrow
$\wt{\rho}$ is the group homomorphism given by taking an irreducible $1$-cycle in
$Z^g_1(\sX)$ and intersecting it with $Y$. The middle vertical arrow and $\rho'$
are the restriction maps. Having the description of all maps at hand, it is clear that
~\eqref{eqn:Rest-0} is commutative.

We now look at the diagram
\begin{equation}\label{eqn:Rest-1}
  \xymatrix@C1pc{
    Z^g_1(\sX) \ar[d]_-{\wt{\rho}} \ar[r] &
    H^2_\zar(\sX, \sK_{2,\sX}) \ar[d]^-{\rho'} \ar[r]^-{\cong}
    & \CH_1(\sX) \ar@{.>}[d]^-{\rho} \\
    Z_0(Y_\reg) \ar[r] & H^2_\zar(Y, \sK_{2,Y}) \ar[r]^-{\cong} & \CH^{\Lci}_0(Y).}
\end{equation}

The left square is commutative by ~\eqref{eqn:Rest-0} and the right horizontal arrow
in the bottom row is an isomorphism by \cite[Thm.~8.1]{Binda-Krishna-Saito}.
It follows that there is a unique restriction map $\rho \colon  \CH_1(\sX) \to
\CH^{\Lci}_0(Y)$ which makes the outer square
commute.  If $\sX$ is quasi-projective over $R$, then
the composite horizontal arrow in the top row of ~\eqref{eqn:Rest-1} is surjective
by the moving lemma \cite[Thm.~2.3]{GLL} of Gabber-Liu-Lorenzini.
This completes the proof of \thmref{thm:Main-2}.
\qed

Letting $\Lambda = \Z$ if $\Char(\F) = 0$ and $\Lambda = \Z[\tfrac{1}{p}]$ if
$\Char(\F) = p > 0$, we get the following.

\begin{cor}\label{cor:Main-2*}
  In the setting of \thmref{thm:Main-2}, we have a commutative diagram
 \begin{equation}\label{eqn:Main-2*-0}
\xymatrix@C.8pc{
Z^g_1(\sX)_\Lambda \ar[r]^-{\wt{\rho}} \ar[d] & 
Z_0(Y_\reg)_\Lambda \ar@{->>}[d] \\
\CH_1(\sX)_\Lambda \ar[r]^-{\rho} & H^4(Y, \Lambda(2)).}
 \end{equation}
 whose left vertical arrow is surjective if $\sX$ is quasi-projective over $R$.
\end{cor} 
\begin{proof}
Combine \thmref{thm:Main-2} with \cite[Thm.~1.6]{Binda-Krishna-SNS}.
\end{proof}

\section{Cohomology of ${\sK_{2, \sX}}/n$}\label{sec:K-fin}
In this section, our goal is to study the Gersten complex for the Zariski sheaf
${\sK_{2, \sX}}/n$ when $n \in \sO^{\times}_{\sX}$ and use it to prove
\thmref{thm:Main-11}.
We begin by proving a general result concerning the relation between the Leray
and coniveau spectral sequences for the {\'e}tale cohomology.

\subsection{Coniveau and Leray spectral sequences for {\'e}tale cohomology}
\label{sec:LCSS}
Let $Z$ be an integral Noetherian scheme of Krull dimension $d$
which we assume to be regular. We let $q \ge 0$ be an integer and let $n$ be an
integer such that $1/n \in \sO_Z$. We let $\Lambda_n(q)$ denote the
{\'e}tale  sheaf $\mu^{\otimes q}_n$ on $Z$. For integers $i, q$, we let $\sH^i(n,q)$
denote the Zariski sheaf on $Z$ associated to the presheaf
$U \mapsto H^i_\et(U, \Lambda_n(q))$.

Recall that by Gabber's  absolute purity theorem \cite{Fujiwara},
the coniveau spectral sequence for the {\'e}tale cohomology of
$\Lambda_n(q)$ (cf. \cite[\S~1]{CHK}) has the form
\begin{equation}\label{eqn:Unram-2-0}
    E^{i,j}_1(c) = {\underset{x \in Z^{(i)}}\bigoplus} H^{j-i}_\et(k(x), \Lambda_n(q-i))
    \Rightarrow H^{i+j}_\et(Z, \Lambda_n(q)).
\end{equation}
On the other hand, associated to the morphism of
sites $\epsilon \colon Z_\et \to Z_\zar$, one has the Leray spectral sequence
  \begin{equation}\label{eqn:LSS-0}
    E^{i,j}_2(l) = H^i_\zar(Z, \sH^j(n,q)) \Rightarrow H^{i+j}_\et(Z, \Lambda_n(q)).
\end{equation}

  Let $F^\bullet_C H^{i}_\et(Z, \Lambda_n(q))$
  (resp. $F^\bullet_L H^{i}_\et(Z, \Lambda_n(q))$) denote the filtration of
  $H^{i}_\et(Z, \Lambda_n(q))$ arising from the coniveau (resp. Leray) spectral
  sequence.

Recall the Bloch-Ogus complex 
\begin{equation}\label{eqn:BO-0}
  0 \to \sH^j(n,q) \xrightarrow{\epsilon^{\et}} i_{\eta *} H^j_\et(k(\eta), \Lambda_n(q))
  \xrightarrow{\partial^{\et}_0} \hspace*{4cm}
\end{equation}
\[
\hspace*{2cm} {\underset{x \in Z^{(1)}}\bigoplus}
  i_{x *} H^{j-1}_\et(k(x), \Lambda_n(q-1)) \xrightarrow{\partial^{\et}_1}
  \cdots \xrightarrow{\partial^{\et}_{q-1}}
  {\underset{x \in Z^{(q)}}\bigoplus} i_{x *} H^{j-q}_\et(k(x), \Lambda_n(0)) \to 0
  \]
  of Zariski sheaves on $Z$, where $\eta$ is the generic point of $Z$.
This complex is constructed in the same manner as Quillen constructed his Gersten
  complex for $K$-theory using the localization sequence, which is valid also for the
  {\'e}tale cohomology. We refer to \cite[\S~1, Cor.~2.1.3]{CHK}
  (see also \cite[Defn.~2.1]{Saito-Sato-Ann}) for its construction.
We let $\sH^j_{\nr}(n,q)$ denote the kernel of $\partial^{\et}_0$.

If we let $\sG^j_\bullet(n,q)$ denote the complex of sheaves obtained by leaving out
$\sH^j(n,q)$ in ~\eqref{eqn:BO-0} and consider the global sections, we get a
complex
 \begin{equation}\label{eqn:BO-1}             
G^j_\bullet(n,q):= \left(H^j_\et(k(\eta), \Lambda_n(q))
  \xrightarrow{\partial^{\et}_0} \cdots \xrightarrow{\partial^{\et}_{q-1}}
 {\underset{x \in \sX^{(q)}}\bigoplus} H^{j-q}_\et(k(x), \Lambda_n(0))\right)
 \end{equation}
 which coincides with the complex $(E^{\bullet, j}_1(c), d^{\bullet, j}_1(c))$.
 Since the terms of the complex $\sG^j_\bullet(n,q)$ are all flasque,
the morphism of complexes
of Zariski sheaves $\epsilon^{\et} \colon \sH^j(n,q) \to \sG^j_\bullet(n,q)$
therefore defines a canonical homomorphism of abelian groups
\begin{equation}\label{eqn:BO-2} 
  \lambda_Z(i,j) \colon E^{i,j}_2(l) = H^{i}_\zar(Z, \sH^j(n,q)) \to
  H_{i}(G^j_\bullet(n,q)) = E^{i,j}_2(c).
\end{equation}

The following result is a refinement of \cite[Cor.~4.4]{Paranjape}, and 
also an abelian version of \cite[Thm.~2]{Gillet-Soule}.

\begin{prop}\label{prop:LC-SS}
  Given $q \ge 0$, there exists a natural morphism of spectral sequences
  $E^{i,j}_2(l) \to E^{i,j}_2(c)$ which induces the identity map on their abutment
  $H^{i+j}_{\rm {\etl}}(Z, \Lambda_n(q))$ and induces the map $\lambda_Z(i,j)$ on the
  $(i,j)$-th terms of the spectral sequences. In particular, we have a natural
  inclusion $F^\bullet_L H^{i+j}_{\rm {\etl}}(Z, \Lambda_n(q)) \subseteq
  F^\bullet_C H^{i+j}_{\rm {\etl}}(Z, \Lambda_n(q))$.
\end{prop}
\begin{proof}
We let $\Lambda_n(q) \to \sI^\bullet$ be an injective resolution of $\Lambda_n(q)$
on $Z_\et$. We let $\sJ^\bullet = \epsilon_*(\sI^{\bullet})$ and let
$\sJ^\bullet \to \sL^{\bullet, \bullet}$ be a
  Cartan--Eilenberg resolution of $\sJ^{\bullet}$ on $Z_\zar$. Let
  $\sJ^\bullet \to \sL^\bullet$ be the induced map, where $\sL^\bullet$
  denotes the totalization of $\sL^{\bullet, \bullet}$. For $r \ge 0$, we let
  $\sL^{\ge r, \bullet} \subset \sL^{\bullet, \bullet}$ be the subcomplex obtained
  by taking the brutal truncation of $\sL^{\bullet, \bullet}$ along the first
  coordinate and let $\sL^{\ge r} \subset \sL^{\bullet}$ denote the totalization of
  $\sL^{\ge r, \bullet}$. We let $F = F^\star \sJ^{\bullet}$ denote the
  trivial filtration of $\sJ^{\bullet}$ (see
  \cite[\S~1.1]{Hodge-2} for the definition of filtered complexes and
  morphisms between them) given by $F^0\sJ^{\bullet} = \sJ^{\bullet}$
  and $F^1\sJ^{\bullet} = 0$. We let $G = G^\star \sL^{\bullet}$ denote the
  filtration of $\sL^{\bullet}$ given by $G^r \sL^{\bullet} = \sL^{\ge r}$.

 By \cite[Fact~3.4]{Paranjape}, the canonical morphism
  $\phi \colon (\sJ^{\bullet}, F) \to (\sL^{\bullet}, G)$ of filtered complexes is a
  level-2 filtered injective resolution in the sense of Definition~3.2 of
  op. cit.. In particular, it follows from Theorem~1.1 of op. cit. that
  we have a $E_0$-spectral sequence
\begin{equation}\label{eqn:LC-SS-0}
  E^{i,j}_0(\sJ^\bullet, F) \Rightarrow H^{i+j}_\et(Z, \Lambda_n(q))
\end{equation}
such that the resulting $E_2$-spectral sequence $E^{i,j}_2(\sJ^\bullet, F)$
coincides with the Leray spectral sequence $E^{i,j}_2(\sJ^\bullet, G) = E^{i,j}_2(l)$ of
~\eqref{eqn:LSS-0}.

We next let $H = H^\star \sJ^{\bullet}$ be the coniveau filtration of
$\sJ^{\bullet}$, where $H^r \sJ^{\bullet}$ is the homotopy colimit
of the homotopy fibers of the unit of adjunction maps
$\sJ^{\bullet} \to \sJ^\bullet|_{Z \setminus W}$ as
$W \subset Z$ runs through all closed subsets of $Z$ of codimension at least $r$.
We then have the canonical inclusion
$\psi \colon (\sJ^{\bullet}, F) \to (\sJ^{\bullet}, H)$ of filtered
complexes. This induces a morphism of spectral sequences
\begin{equation}\label{eqn:LC-SS-1}
\xymatrix@C1.6pc{
  E^{i,j}_0(\sJ^\bullet, F) \ar[d]_-{\psi} & \Rightarrow  & H^{i+j}_\et(Z, \Lambda_n(q))
  \ar@{=}[d] \\
  E^{i,j}_0(\sJ^{\bullet}, H) & \Rightarrow & H^{i+j}_\et(Z, \Lambda_n(q)).}
\end{equation}

It is clear from the definition of the {\'e}tale cohomology with support that
the $E_1$-spectral sequence arising from the bottom row of ~\eqref{eqn:LC-SS-1}
is the coniveau spectral sequence of ~\eqref{eqn:Unram-2-0}.

Considering the $E_2$-spectral sequences arising from ~\eqref{eqn:LC-SS-1}, we
therefore get a morphism of spectral sequences
\begin{equation}\label{eqn:LC-SS-2}
\xymatrix@C1.6pc{
  E^{i,j}_2(l) \ar[d]_-{\psi} & \Rightarrow & H^{i+j}_\et(Z, \Lambda_n(q))
  \ar@{=}[d] \\
  E^{i,j}_2(c) & \Rightarrow & H^{i+j}_\et(Z, \Lambda_n(q)),}
\end{equation}
and this furnishes the desired morphism of spectral sequences.
The remaining part of the proposition is clear from the above construction
of the spectral sequences and the morphisms between them.
\end{proof}

\subsection{Gersten resolution for ${\sK^{\nr}_{2,\sX}}/n$}\label{sec:Ger-n-1}
We return to the setting of \S~\ref{sec:Gersten-model} for the remainder of
\S~\ref{sec:K-fin}.
In this subsection, we prove several results concerning the cohomology of
${\sK_{2,\sX}}/n$ and $\sH^j(n,2)$ on $\sX$. Recall that $\F$ is an arbitrary
field, $n \in \F^\times$ and $\Lambda_n = {\Z}/n$.
We begin with the following general result.

\begin{lem}\label{lem:Unram-2}
  Let $A$ be a regular local ring with quotient field $F$ such that $n \in A^\times$.
  Assume that either $n$ is odd or $\sqrt{-1} \in A$. Then the map
  \[
  \partial^{{\rm {\etl}}}_0 \colon H^1_{\rm {\etl}}(F, \Lambda_n(2))
  \to {\underset{\htt(\fp) = 1}\bigoplus} H^{0}_{\rm {\etl}}(k(\fp), \Lambda_n(1))
  \]
  is surjective.
\end{lem}
\begin{proof}
In the coniveau spectral sequence ~\eqref{eqn:Unram-2-0} for $Z = \Spec(A)$,
  we see that $E^{i,j}_r = 0$ if $i > j$ or $i < 0$. 
If we run this spectral sequence whose abutment is
  $H^2_\et(A, \Lambda_n(2))$, we get an exact sequence
  \begin{equation}\label{eqn:Unram-2-1}
    0 \to E^{0,2}_\infty \to  H^{2}_\et(F, \Lambda_n(2)) \xrightarrow{d^{0,2}_1}
    {\underset{\htt(\fp) =1}\bigoplus} H^{1}_\et(k(\fp), \Lambda_n(1)).
    \end{equation} 
  We also get $E^{1,1}_2 \cong E^{1,1}_\infty$ and an exact sequence
  \begin{equation}\label{eqn:Unram-2-2}
H^1_\et(F, \Lambda_n(2)) \xrightarrow{d^{0,1}_1 = - \partial^{\et}_0}
 {\underset{\htt(\fp) = 1}\bigoplus} H^{0}_\et(k(\fp), \Lambda_n(1)) \to
E^{1,1}_\infty \to 0.
\end{equation}

 Since $E^{i,2-i}_1 = 0$ for $i \ge 2$, we get an exact sequence
  \begin{equation}\label{eqn:Unram-2-3}
    0 \to E^{1,1}_\infty \to H^2_\et(A, \Lambda_n(2)) \xrightarrow{\epsilon^{\et}}
    H^{2}_\et(F, \Lambda_n(2)) \xrightarrow{d^{0,2}_1}
  {\underset{\htt(\fp) =1}\bigoplus} H^{1}_\et(k(\fp), \Lambda_n(1)).
  \end{equation}

By ~\eqref{eqn:Unram-2-2}, proving the lemma is equivalent to proving that
  $E^{1,1}_\infty =0$. The latter is equivalent to showing that the pull-back map
  $\epsilon^{\et}$ in ~\eqref{eqn:Unram-2-3} is injective.
  But this follows from \cite[Lem.~2]{Hoobler}.
\end{proof}

As an application of \lemref{lem:Unram-2}, we shall now prove a mod-$n$ analogue of
\thmref{thm:Gersten-10} together with several of its consequences. To simplify the
exposition, we first introduce the following notation. It is slightly different
from the similar notation in \S~\ref{sec:VRC}.

\vskip.2cm

{\sl {\bf Notation:}} We shall say that the $(\star_n)$-condition holds for $\sX$
if one of the following conditions is satisfied.
\begin{enumerate}
  \item
$n$ is odd.
  \item
    $\sqrt{-1} \in R$.
  \item
    $R$ is henselian and $\sqrt{-1} \in \F$.
\item
  $\sX$ is a smooth $R$-scheme.
\item
  $R$ is equicharacteristic.
  \end{enumerate}

\begin{thm}\label{thm:Unram-3}
Under the $(\star_n)$-condition, the Gersten complexes
  \begin{equation}\label{eqn:Unram-3-2}
    0 \to {\sK^{\nr}_{2,\sX}}[n] \to  i_{\eta *} {K_2(k(\eta))}[n]
    \xrightarrow{\partial_0}
    {\underset{x \in \sX^{(1)}}\bigoplus} i_{x*} {k(x)^\times}[n] \to  0;
  \end{equation}
 
\begin{equation}\label{eqn:Unram-3-0}
   0 \to {\sK^{\nr}_{2,\sX}}/n \to i_{\eta *} {K_2(k(\eta))}/n \xrightarrow{\partial_0}
  {\underset{x \in \sX^{(1)}}\bigoplus} i_{x*} {k(x)^\times}/n \xrightarrow{\partial_1}
  {\underset{x \in \sX^{(2)}}\bigoplus} i_{x*} {\Z}/n \to 0
\end{equation}
are exact.
\end{thm}
\begin{proof}
 We look at the commutative diagram
  \begin{equation}\label{eqn:Unram-3-1}
    \xymatrix@C1pc{
       & 0 \ar[d] & 0 \ar[d] & 0 \ar[d] & & \\
      0 \ar[r] & {\sK^{\nr}_{2,\sX}}[n] \ar[r] \ar[d] &  i_{\eta *} {K_2(k(\eta))}[n]
      \ar[r]^-{\partial_0} \ar[d] &
         {\underset{x \in \sX^{(1)}}\bigoplus} i_{x*} {k(x)^\times}[n]
        \ar[r] \ar[d] & 0 \ar[d] & \\
        0 \ar[r] & {\sK^{\nr}_{2,\sX}} \ar[r] \ar[d]_-{n} &
        i_{\eta *} {K_2(k(\eta))} \ar[r]^-{\partial_0} \ar[d]^-{n} \ar[r] & 
        {\underset{x \in \sX^{(1)}}\bigoplus} i_{x*} {k(x)^\times} \ar[r]^-{\partial_1}
        \ar[d]^-{n} &
  {\underset{x \in \sX^{(2)}}\bigoplus} i_{x*} {\Z} \ar[r] \ar[d]^-{n} & 0 \\
  0 \ar[r] & {\sK^{\nr}_{2,\sX}} \ar[r] \ar[d] &
  i_{\eta *} {K_2(k(\eta))} \ar[r]^-{\partial_0} \ar[d] \ar[r] & 
  {\underset{x \in \sX^{(1)}}\bigoplus} i_{x*} {k(x)^\times} \ar[r]^-{\partial_1}
  \ar[d] & {\underset{x \in \sX^{(2)}}\bigoplus} i_{x*} {\Z} \ar[r] \ar[d] & 0 \\
  0 \ar[r] & {\sK^{\nr}_{2,\sX}}/n \ar[r] \ar[d] &
  i_{\eta *} {K_2(k(\eta))}/n \ar[r]^-{\partial_0} \ar[r] \ar[d] & 
 {\underset{x \in \sX^{(1)}}\bigoplus} i_{x*} {k(x)^\times}/n \ar[r]^-{\partial_1} \ar[d]
 & {\underset{x \in \sX^{(2)}}\bigoplus} i_{x*} {\Z}/n \ar[r] \ar[d] & 0. \\
  & 0 & 0 & 0 & 0 &} 
\end{equation}

Since all columns are exact, the two middle rows are exact by
  \thmref{thm:Gersten-10}, and the maps between these rows are all multiplication by
  $n$, a diagram chase tells us that the following conditions are equivalent.
  \begin{enumerate}
  \item
    ~\eqref{eqn:Unram-3-2} and ~\eqref{eqn:Unram-3-0} are exact.
  \item
    The top row of ~\eqref{eqn:Unram-3-1} is exact.
  \item
    The bottom row of ~\eqref{eqn:Unram-3-1} is exact.
  \item
    The map $\partial_0$ in the top row is surjective.
    \end{enumerate}
We shall verify either (3) or (4), depending on which of the
   $(\star_n)$-condition holds.

  We let $x \in \sX$ and $A = \sO_{\sX, x}$.
  We assume first that either the first or the third $(\star_n)$-condition holds. By
  \cite[Lem.~3.14]{Suslin-KT}, there is a commutative diagram
  \begin{equation}\label{eqn:Unram-3-1-1}
    \xymatrix@C1pc{
H^1_\et(k(\eta), \Lambda_n(2)) \ar[r]^-{\partial^{\et}_0} \ar[d] & 
{\underset{\htt(\fp) = 1}\bigoplus} H^{0}_\et(k(\fp), \Lambda_n(1))
\ar[d]^-{\cong} \\
  K_2(k(\eta))[n] \ar[r]^-{\partial_0} & {\underset{\htt(\fp) = 1}\bigoplus}
  k(\fp)^\times[n],}
  \end{equation}
  where the right vertical arrow is the canonical isomorphism induced by the
  Kummer sequence
  (Suslin also shows that the left vertical arrow is surjective but we don't need it).
 Since the top horizontal arrow is surjective by \lemref{lem:Unram-2}, it follows
 that so is the bottom horizontal arrow. This completes the proof of the theorem when
 either the first or the third $(\star_n)$-condition holds. Since the
 third $(\star_n)$-condition implies the second, the theorem
 follows in the latter case too.

 We assume now that $\sX$ is a smooth $R$-scheme. We shall show in this case that
 the bottom row of ~\eqref{eqn:Unram-3-1} is exact, which will prove the theorem
 under the fourth $(\star_n)$-condition.
 Since the last term of ~\eqref{eqn:Gersten-9} is torsion-free, it suffices to show
 that the map ${K_2(A)}/n \to {K_2(k(\eta))}/n$ is injective. To show this, we
 look at the commutative diagram
\begin{equation}\label{eqn:Unram-3-1-2}
    \xymatrix@C1pc{
      {K_2(A)}/n \ar[r] \ar[d] & {K_2(k(\eta))}/n \ar[d] \\
      K_2(A, \Lambda_n) \ar[r] & K_2(k(\eta), \Lambda_n),}
\end{equation}
where $K(A, \Lambda_n)$ is the $K$-theory with finite coefficient spectrum and
the vertical arrows are the canonical inclusions given by the universal
coefficient theorem. Since the bottom horizontal arrow is injective by
\cite[Cor.~B]{Gillet-JA}, it follows that the top horizontal arrow is also injective.
This proves the theorem if the fourth $(\star_n)$-condition holds.

Finally, we assume that $R$ is an equicharacteristic ring. As in the previous case,
it suffices to show that the map ${K_2(A)}/n \to {K_2(k(\eta))}/n$ is injective.
To show this, we apply \cite[Thm.~2.2]{KP-Comp} to reduce to the
case when $A$ is essentially of finite type and smooth over a perfect field. 
In this case, we can use ~\eqref{eqn:Unram-3-1-2} and \cite[Cor.~B]{Gillet-JA}
(which holds for smooth algebras over fields too) once again to conclude the
desired injectivity. This completes the proof of the theorem.
\end{proof}

\begin{cor}\label{cor:Unram-4}
 Under the $(\star_n)$-condition, we have an exact sequence
  \begin{equation}\label{eqn:Unram-4-0}
  0 \to {H^1_\zar(\sX, {\sK}^{\nr}_{2,\sX})}/n \to H^1_\zar(\sX, {{\sK}^{\nr}_{2,\sX}}/n)
  \to \CH_1(\sX)[n] \to 0.
  \end{equation}
  \end{cor}
\begin{proof}
  It follows from \thmref{thm:Unram-3} that
  $H^i_\zar(\sX, {\sK}^{\nr}_{2,\sX}[n]) = 0$ for $i \ge 2$.
Letting $n{\sK}^{\nr}_{2,\sX}$ denote the image of the `multiplication by $n$' map on
  ${\sK}^{\nr}_{2,\sX}$, this implies that the map $H^i_\zar(\sX, {\sK}^{\nr}_{2,\sX}) \to
  H^i_\zar(\sX, n{\sK}^{\nr}_{2,\sX})$ is surjective for $i =1$ and bijective
  for $i \ge 2$. Using ~\eqref{eqn:Bloch-2-0}, this easily results in the exactness
  of ~\eqref{eqn:Unram-4-0}.
\end{proof}

\subsection{Bloch--Ogus resolution for $\sH^2_{\nr}(n,2)$}\label{sec:ECC}
We continue to work in the setting of \S~\ref{sec:Gersten-model}. 
We shall write $H^{i,j}_n(A) = H^i_\et(A, \Lambda_n(j))$ for any ring $A$.

We consider the diagram
\begin{equation}\label{eqn:Unram-6-0}
    \xymatrix@C.7pc{
 {\sK_{2,\sX}}/n \ar[r]^-{\epsilon} \ar[d]_-{\beta_0} &
      i_{\eta *} {K_2(k(\eta))}/n \ar[r]^-{\partial_0}
 \ar[d]^-{\beta_0} & {\underset{x \in \sX^{(1)}}\bigoplus} i_{x*} {k(x)^\times}/n
 \ar[r]^-{\partial_1} \ar[d]^-{\beta_1} &
    {\underset{x \in \sX^{(2)}}\bigoplus} i_{x*} {\Z}/n
 \ar[r] \ar[d]^-{\beta_2} &  0 \\
\sH^2(n,2) \ar[r]^-{\epsilon^{\et}} & i_{\eta *} H^{2,2}_n(k(\eta))
  \ar[r]^-{\partial^{\et}_0} & {\underset{x \in \sX^{(1)}}\bigoplus}
  i_{x *} H^{1,1}_n(k(x)) \ar[r]^-{\partial^{\et}_1}
  & {\underset{x \in \sX^{(2)}}\bigoplus} i_{x *} H^{0,0}_n(k(x)) \ar[r] & 0,}
\end{equation}
whose top row is the (mod-$n$) Gersten complex ~\eqref{eqn:Gersten-0}, the
  bottom row is the Bloch-Ogus complex ~\eqref{eqn:BO-1}, and the vertical
  arrows are the Norm-residue maps.

\begin{lem}\label{lem:Unram-6}
    The diagram ~\eqref{eqn:Unram-6-0} is commutative and its vertical arrows are
    bijective.
  \end{lem}
  \begin{proof}
The commutativity of the diagram
    is well-known (cf. \cite[Proof of Thm.~(1.1), P.~942]{Bloch-BAMS}).
    We give a proof sketch as op. cit. does not provide a proof.

The left square in  ~\eqref{eqn:Unram-6-0} clearly commutes.
  The middle square commutes by a combination of \cite[Thm.~3.11]{Gillet-Adv}
  and \cite[Chap.~V, Lem.~11.13]{Weibel-K-book}. To see that the right square
  commutes, note that for any $x \in  \sX^{(1)}$ and any codimension one point $y$ of
  $\ov{\{x\}}$, one has that the $x$-component of $\beta_1$ is the Kummer map, i.e.,
  the first boundary map arising from the Kummer sequence
  (cf. \cite[Chap.~V, Exm.~11.10]{Weibel-K-book})
  while the $y$-component of $\beta_2$ is the identity map.

  We now let $Z = \ov{\{x\}}, \ A = \sO_{Z,y}$ and let $A_n$ denote the normalization
  of $A$. Let $y_1, \ldots , y_r$ denote the closed points of $\Spec(A_n)$
  (note that $A$ is excellent because $R$ is).
   We then have a commutative diagram
\begin{equation}\label{eqn:Unram-6-1}
    \xymatrix@C2pc{
      A^\times \ar[r] \ar[d] &  k(x)^\times \ar[r]^-{\rm valuation} \ar[d] &
      \stackrel{r}{\underset{i =1}\oplus} K_0(k(y_i))
      \ar[r] \ar@{->>}[d] & 0 \\
      H^{1,1}_n(A) \ar[r] &  H^{1,1}_n(k(x)) \ar[r] &
      \stackrel{r}{\underset{i =1}\oplus} H^{0,0}_n(k(y_i)) \ar[r] & 0,}
\end{equation}
in which the left and the middle vertical arrows are the Kummer maps and the
right vertical arrow is the canonical quotient map.

The commutativity of the right square in ~\eqref{eqn:Unram-6-0} now follows because
its $(x,y)$-component is nothing but (the reduction mod-$n$ of)
the composition of the two squares in the diagram
\begin{equation}\label{eqn:Unram-6-2}
    \xymatrix@C1pc{
      k(x)^\times \ar[r] \ar[d] & \stackrel{r}{\underset{i =1}\oplus} K_0(k(y_i))
      \ar[r] \ar@{->>}[d] & K_0(k(y)) \ar@{->>}[d] \\
      H^{1,1}_n(k(x)) \ar[r] &
      \stackrel{r}{\underset{i =1}\oplus} H^{0,0}_n(k(y_i)) \ar[r] &
      H^{0,0}_n(k(y)),}
\end{equation}
whose horizontal arrows in the second square are the corestriction maps.
In particular, this square commutes. The vertical arrows of ~\eqref{eqn:Unram-6-0}
are bijective by \cite[Thm.~3]{Hoobler}.
  \end{proof}

\begin{lem}\label{lem:Unram-8}
  We have the following.
  \begin{enumerate}
  \item
    There exists a commutative diagram
  \begin{equation}\label{eqn:Unram-8-0}
    \xymatrix@C1pc{
      {\sK_{2,\sX}}/n \ar[r]^-{\epsilon} \ar[d]_-{\beta_0} & {{\sK}^{\nr}_{2, \sX}}/n
      \ar[d]^-{\beta'_0} \\
      \sH^2(n,2) \ar[r]^-{\epsilon^{{\rm {\etl}}}} & \sH^2_{\nr}(n,2),}
  \end{equation}
  whose left vertical arrow is bijective and the right vertical arrow is surjective.
  If the $(\star_n)$-condition holds, then the right vertical arrow is bijective.
\item
If the $(\star_n)$-condition holds, then the Bloch-Ogus sequence
  \[
  \hspace*{.6cm}
  0 \to \sH^2_{\nr}(n,2) \xrightarrow{\epsilon^{\rm {\etl}}} i_{\eta *} H^{2,2}_n(k(\eta))
  \xrightarrow{\partial^{\rm {\etl}}_0} {\underset{x \in \sX^{(1)}}\bigoplus}
  i_{x *} H^{1,1}_n(k(x)) \xrightarrow{\partial^{\rm {\etl}}_1}
  {\underset{x \in \sX^{(2)}}\bigoplus} i_{x *} H^{0,0}_n(k(x)) \to 0
  \]
  is a flasque resolution of $\sH^2_{\nr}(n,2)$.
  \end{enumerate}
\end{lem}
\begin{proof}
  Item (1) follows from \thmref{thm:Unram-3} and \lemref{lem:Unram-6} if we let
  $\beta'_0$ be the map induced on the kernels of the horizontal arrows in
  the middle square of ~\eqref{eqn:Unram-6-0}. Item (2) also follows by a
  combination of \thmref{thm:Unram-3} and \lemref{lem:Unram-6}.
  \end{proof}

\subsection{Proof of \thmref{thm:Main-11}}\label{sec:ECC-prf}
We continue to work in the setting of \S~\ref{sec:Gersten-model}.
We shall prove few lemmas before we come to the proof of \thmref{thm:Main-11}.

\begin{lem}\label{lem:BO-exact}
  For every $j, q \ge 0$, the Bloch--Ogus sequence
  \[
  0 \to \sH^j(n,q) \xrightarrow{\epsilon^{{\rm {\etl}}}} i_{\eta *} H^{j,q}_n(k(\eta))
  \xrightarrow{\partial^{{\rm {\etl}}}_0} {\underset{x \in \sX^{(1)}}\bigoplus}
  i_{x *} H^{j-1,q-1}_n(k(x)) \hspace*{4cm} 
  \]
  \[
 \hspace*{4cm}  \xrightarrow{\partial^{{\rm {\etl}}}_1} {\underset{x \in \sX^{(2)}}\bigoplus} i_{x *} H^{j-2,q-2}_n(k(x))
\xrightarrow{\partial^{{\rm {\etl}}}_2} 
            {\underset{x \in \sX^{(3)}}\bigoplus} i_{x *} H^{j-3,q-3}_n(k(x)) \to 0
            \]
            is a flasque resolution of $\sH^j(n,q)$ when either the fourth or
            the fifth $(\star_n)$-condition holds, or when we restrict
            $\sH^j(n,q)$ to $X$.
\end{lem}
\begin{proof}
If the fourth $(\star_n)$-condition holds, $\sH^j(n,q)$ admits
the Bloch--Ogus resolution by \cite[Thm.~1.2(5)]{Geisser-MZ}.
If the fifth  $(\star_n)$-condition
holds, then also the Bloch--Ogus resolution of $\sH^j(n,q)$ is known, and
is easily deduced by reducing to the case of smooth schemes over a field case
(which was considered by Bloch--Ogus) by a method developed by Panin
\cite[\S~5]{Panin}. The same reasoning also applies to the restriction of
$\sH^j(n,q)$ to $X$ because the latter is a regular $k$-scheme.
\end{proof}

\begin{lem}\label{lem:Unram-0}
 We have the following.
  \begin{enumerate}
  \item
    The map $H^1_\zar(\sX, \sH^j(n,2)) \to H^1_\zar(\sX, \sH^j_{\nr}(n,2))$ is
    surjective for $j \ge 0$.
    \item
      $H^i_\zar(\sX, \sH^j(n,2)) = 0$ for $j =0,1$ and $i \ge 2$ if the
      $(\star_n)$-condition holds.
      \end{enumerate}
\end{lem}
\begin{proof}
For $j \ge 0$ and $1 \le r \le 5$, we let $\sF^j_r$ be the Zariski sheaf on
  $\sX$
so that we have the short exact sequences 
  \begin{equation}\label{eqn:Unram-0-0}
  0 \to \sF^j_1 \to \sH^j(n,2) \xrightarrow{\epsilon^{\et}} \sF^j_2 \to 0; \
  0 \to  \sF^j_2 \to \sH^j_{\nr}(n,2) \to \sF^j_3 \to 0;
  \end{equation}
  \[
  0 \to \sH^j_{\nr}(n,2) \to  i_{\eta *} H^{j,2}_n(k(\eta)) \to \sF^j_4 \to 0; \
0 \to  \sF^j_4 \to  {\underset{x \in \sX^{(1)}}\bigoplus}
i_{x *} H^{j-1,1}_n(k(x)) \to \sF^j_5 \to 0.
\]

If $x \in \sX_{(2)}$, then $\sO_{\sX,x}$ is a dvr, and we can
apply \cite[Thm.~1.2(5)]{Geisser-MZ} (where we consider $\Spec(A)$
as a smooth $A$-scheme) to see that $(\sF^j_1)_x = 0$ for $j \ge 0$.
This implies by \lemref{lem:Kerz-Strunk} that $H^i_\zar(\sX, \sF^j_1) = 0$ for
$j \ge 0, i > 1$. On the other hand, we have
$H^i_\zar(\sX, \sF^j_3) = 0$ for $i > 0$ by \cite[Prop.~6]{Sakagaito} and
\lemref{lem:Kerz-Strunk}. This implies part (1) of the lemma.
If the fourth  or the fifth $(\star_n)$-condition, then
$H^i_\zar(\sX, \sH^j(n,2)) = 0$ for $j \ge 0$ and $i \ge 2$ by
\lemref{lem:BO-exact}.

Since the second  $(\star_n)$-condition is implied by the third, we are left
with showing part (2) when either the first or the second $(\star_n)$-conditions
holds. Since $ i_{\eta *} H^{j,2}_n(k(\eta))$ and
${\underset{x \in \sX^{(1)}}\bigoplus} i_{x *} H^{j-1,1}_n(k(x))$ are flasque, an
elementary argument reduces the proof to showing that
$\sF^j_5 = 0$ for $j = 0,1$. But $\sF^0_5$ is clearly zero (because
$ H^{-1,1}_n(k(x)) =0$ for every $x \in  \sX^{(1)}$), and
$\sF^1_5 = 0$ by \lemref{lem:Unram-2}. This completes the proof.
\end{proof}

\begin{remk}\label{remk:Unram-0-1}
  The reader may note that the $j =0$ case of part (2) of
  \lemref{lem:Unram-0} does not use the $(\star_n)$-condition.
\end{remk}

\begin{lem}\label{lem:Unram-9}
  Under the $(\star_n)$-condition, we have $H^i_{Y_j}(\sX, \sH^1(n,2)) =
  0 = H^i_Y(\sX, \sH^1(n,2))$ for $1 \le j \le m$ and $i \ge 3$.
\end{lem}
\begin{proof}
We let $\sX_j = \sX \setminus Y_j$ and look at the exact sequences
  \[
  H^{i-1}_\zar(\sX_j,  \sH^1(n,2)) \to H^i_{Y_j}(\sX, \sH^1(n,2)) \to  
  H^{i}_\zar(\sX,  \sH^1(n,2));
  \]
  \[
  H^{i-1}_\zar(X,  \sH^1(n,2)) \to H^i_{Y}(\sX, \sH^1(n,2)) \to  
  H^{i}_\zar(\sX,  \sH^1(n,2)).
  \]
  The Bloch--Ogus resolution for $\sH^1(n,2)$ on $X$ (cf. \lemref{lem:BO-exact})
  and \lemref{lem:Unram-0} together
 imply that the end terms of the second exact sequence vanish. This
 implies that $H^i_Y(\sX, \sH^1(n,2)) =0$.

If $Y$ has only one irreducible component, then $Y_j = Y$ and we are done.
 If $Y$ has multiple components, then $\sX_j$ satisfies the same hypothesis as
 $\sX$ does for every $1 \le j \le m$. In particular,
 $H^{i-1}_\zar(\sX_j,  \sH^1(n,2)) = 0$ by \lemref{lem:Unram-0}. The first exact
 sequence above now yields the desired result.
\end{proof}

\begin{lem}\label{lem:Unram-7}
Under the $(\star_n)$-condition, the canonical maps
\[
H^1_\zar(\sX, {{\sK}_{2, \sX}}/n) \to H^1_\zar(\sX, {{\sK}^{\nr}_{2, \sX}}/n); \ \ \
H^1_\zar(\sX, \sH^2(n,2)) \to H^1_\zar(\sX, \sH^2_{\nr}(n,2))
\]
are isomorphisms.
\end{lem}
\begin{proof}
By \lemref{lem:Unram-8}, it is enough to prove that the second map is
an isomorphism. To this end, we first compute $H^3_\et(\sX, \Lambda_n(2))$
in terms of its filtration resulting from the Leray spectral sequence
~\eqref{eqn:LSS-0}. 
Since $E^{3,1}_2(l) = 0$ by \lemref{lem:Unram-0}, we get
$\frac{F^1_L H^3_\et(\sX, \Lambda_n(2))}{F^2_L H^3_\et(\sX, \Lambda_n(2))} =
E^{1,2}_\infty(l) = E^{1,2}_2(l)$. 
Since Lemma~\ref{lem:Unram-0} also implies that
$\frac{F^i_L H^3_\et(\sX, \Lambda_n(2))}{F^{i+1}_L H^3_\et(\sX, \Lambda_n(2))} = 0$
for $i \ge 2$ so that $F^2_L H^3_\et(\sX, \Lambda_n(2)) = 0$, we
get $F^1_L H^3_\et(\sX, \Lambda_n(2)) = H^1_\zar(\sX, \sH^2(n,2))$.

On the other hand, if we apply the coniveau spectral sequence
~\eqref{eqn:Unram-2-0}
for computing $H^3_\et(\sX, \Lambda_n(2))$, then \lemref{lem:Unram-8}(2) says that
$F^1_C H^3_\et(\sX, \Lambda_n(2)) = H^1_\zar(\sX, \sH^2_{\nr}(n,2))$.
By \propref{prop:LC-SS} therefore, we get a commutative diagram
\begin{equation}\label{eqn:Unram-7-0}
  \xymatrix@C1.6pc{
    H^1_\zar(\sX, \sH^2(n,2)) \ar[r]^-{\alpha_1} \ar[d]_-{\epsilon^{\et}} &
    H^3_\et(\sX, \Lambda_n(2)) \ar[d]^-{\id} \\
    H^1_\zar(\sX, \sH^2_{\nr}(n,2)) \ar[r]^-{\alpha^{\nr}_1} &
    H^3_\et(\sX, \Lambda_n(2)),}
\end{equation}
whose horizontal arrows are injective. We now apply \lemref{lem:Unram-0} to
conclude the proof.
\end{proof}

The following result proves \thmref{thm:Main-11}.

\begin{thm}\label{thm:Unram-5}
  Under the $(\star_n)$-condition, the following hold.
  \begin{enumerate}
  \item
    The map $H^i_\zar(\sX, {{\sK}_{2,\sX}}/n) \to
    H^i_\zar(\sX, {{\sK}^{\nr}_{2,\sX}}/n)$
    is injective for $i =0$ and bijective for $i \ge 1$.
    \item
   The map $H^i_Y(\sX, {{\sK}_{2,\sX}}/n) \to H^i_Y(\sX, {{\sK}^{\nr}_{2,\sX}}/n)$
   is surjective for $i =1$ and bijective for $i \ge 2$.
 \item
$\coker\left(H^0_\zar(\sX, {{\sK}_{2,\sX}}/n) \inj
   H^0_\zar(\sX, {{\sK}^{\nr}_{2,\sX}}/n)\right)$ \hspace*{4cm}

   \[
   \hspace*{7cm} \xrightarrow{\cong}
   \Ker\left(H^1_Y(\sX, {{\sK}_{2,\sX}}/n) \surj
   H^1_Y(\sX, {{\sK}^{\nr}_{2,\sX}}/n)\right).
   \]
 \item
   There is a surjective map $H^1_\zar(\sX, {{\sK}_{2,\sX}}/n)
   \surj \CH_1(\sX)[n]$ and a bijective map $H^2_\zar(\sX, {{\sK}_{2,\sX}}/n)
   \xrightarrow{\cong} {\CH_1(\sX)}/n$.
 \item
   There is a commutative diagram
   \[
   \xymatrix@C1pc{
     H^2_Z(\sX, \sH^2(n,2)) \ar[d] &  H^2_Z(\sX, {{\sK}_{2,\sX}}/n)
     \ar[r]^-{\cong} \ar[d] \ar[l]_-{\cong} & {\CH_1(Z)}/n \ar[d] \\
 H^2_\zar(\sX, \sH^2(n,2)) &  H^2_\zar(\sX, {{\sK}_{2,\sX}}/n)
 \ar[r]^-{\cong} \ar[l]_-{\cong} & {\CH_1(\sX)}/n}
   \]
   for $Z \in \{Y_j|1 \le j \le m\} \bigcup  \{Y\}$,
   where the left and the middle vertical arrows are the forget support maps and
   the
  right vertical arrow is the push-forward map via the inclusion $Z \inj \sX$.
  \end{enumerate}
\end{thm}
\begin{proof}
  By \lemref{lem:Unram-8}, and a combination of \cite[Lem.~2]{Hoobler}
  and \lemref{lem:BO-exact}, the map
  ${{\sK}_{2,\sX}}/n \to  {{\sK}^{\nr}_{2,\sX}}/n$ is injective. Meanwhile, we showed
  in the proof of \lemref{lem:Unram-0} that all higher cohomology of its cokernel
 $\sF^2_3$ vanish. Jointly with \lemref{lem:Unram-7}, this implies item (1). 

To prove item (2), we let $\sF =  {{\sK}_{2,\sX}}/n, \
  \sF^{\nr} = {{\sK}^{\nr}_{2,\sX}}/n$ and
  look at the commutative diagram of long exact sequences
  of Zariski cohomology
  \begin{equation}\label{eqn:Unram-5-0}
    \xymatrix@C1pc{
      H^{i-1}(\sX, \sF)  \ar[r] \ar[d]_-{\alpha_{i-1}} &
      H^{i-1}(X, \sF) \ar[r] \ar[d]^-{\beta_{i-1}} & H^i_Y(\sX, \sF)
      \ar[d]^-{\gamma_i}
        \ar[r] &
        H^i(\sX, \sF) \ar[d]_-{\alpha_{i}} \ar[r] &  H^i(X, \sF)
        \ar[d]^-{\beta_{i}} \\
      H^{i-1}(\sX, \sF^{\nr})  \ar[r] & H^{i-1}(X, \sF^{\nr}) \ar[r]  &
      H^i_Y(\sX, \sF^{\nr}) \ar[r] &
      H^i(\sX, \sF^{\nr}) \ar[r] &  H^i(X, \sF^{\nr})}
    \end{equation}
  for $i \ge 1$.
  The map ${\sK}_{2,X} \to {\sK}^{\nr}_{2,X}$ is bijective since the
  Gersten resolution for ${\sK}_{2,X}$ holds on $X$ by \cite[Prop.~10]{Kerz-JAG}.
In particular, $\sF|_X \to  \sF^{\nr}|_X$ is an isomorphism. 
It follows that the maps $\beta_{i-1}$ and $\beta_i$ are isomorphisms.
The desired conclusion now follows from item (1) by a diagram chase.

The map $H^{0}(\sX, \sF) \to H^{0}(X, \sF)$ is injective either by
\cite[Lem.~2]{Hoobler} (when one of the first three of the $(\star_n)$-condition
holds) or by Lemmas~\ref{lem:Unram-8}(1) and ~\ref{lem:BO-exact}
(under fourth or fifth $(\star_n)$-condition). The same holds
for  $\sF^{\nr}$ by \thmref{thm:Unram-3}. Using this, Item (3) follows by a
diagram chase in  ~\eqref{eqn:Unram-5-0}. 
 Item (4) follows from item (1)  in combination with
\thmref{thm:Unram-3} and \corref{cor:Unram-4}.
 Lastly, item (5) follows directly from items (1) and (2) in combination with
  \thmref{thm:Unram-3} and \lemref{lem:Unram-8}.
\end{proof}

\section{Tate module of unramified cohomology}\label{sec:Unram-V}
Let $R$ be an excellent henselian discrete valuation ring  (hdvr)
with quotient field $k$ and residue field $\F$.
We assume that $\F$ is either finite or separably closed of
exponential characteristic $p > 0$.
Let $\sX$ be a connected, regular and proper $R$-scheme of Krull dimension three.
We assume that $\sX$ is faithfully flat over $R$ and its reduced closed fiber
$Y$ is an snc divisor on $\sX$. We let $X$ denote the generic fiber of
$\sX$ and assume that it is a smooth $k$-scheme (this is automatic if
$\Char(k) = 0$).
We let $\{Y_1, \ldots , Y_m\}$ be the set of irreducible components of $Y$.
We let $\ell \neq p$ be a fixed prime.
The goal of this section is to prove Theorems~\ref{thm:Main-0} and
~\ref{thm:Main-12}.

We begin by noting that, under the henselian hypothesis (which will be assumed
throughout the remainder of the paper), the $(\star_n)$-condition is satisfied
if and only if one of the following conditions is satisfied.
\begin{equation}\label{eqn:Star-condn}
  (1) \ n \ \mbox{is  odd}; \ \ (2) \ \sqrt{-1} \in \F; \ \ (3) \ \sX \in \Sm_R;
  \ \
  (4) \ R \ \mbox{is equicharacteristic}.
  \end{equation}

\subsection{The key diagram}\label{sec:Main-map}
We assume that the $(\star_\ell)$-condition holds. By virtue of
Lemmas~\ref{lem:Unram-0} and
~\ref{lem:Unram-9}, the Leray spectral sequence for the {\'e}tale cohomology
groups (with and without support) gives us for, every $n \ge 1$,
the cycle class maps
\begin{equation}\label{eqn:cyc-cl-map}
  H^i_\zar(\sX, \sH^2(\ell^n,2)) \xrightarrow{\alpha_i}
  H^{i+2}_\et(\sX, \Lambda_{\ell^n}(2)); \
 H^i_\zar(X, \sH^2(\ell^n,2))  \xrightarrow{\beta_i}
  H^{i+2}_\et(X, \Lambda_{\ell^n}(2));
\end{equation}
\[
\hspace*{.8cm} H^2_{Y_j}(\sX, \sH^2(\ell^n,2)) \xrightarrow{\gamma_{2,j}}
H^{4}_{\et, Y_j}(\sX, \Lambda_{\ell^n}(2)); \
H^2_Y(\sX, \sH^2(\ell^n,2)) \xrightarrow{\gamma_{2}}
H^{4}_{\et, Y}(\sX, \Lambda_{\ell^n}(2)).
\]
for $i = 1,2$ and $1 \le j \le m$.

We let $\sF_n$ denote the Zariski sheaf $\sH^2(\ell^n,2)$ and let
$\sG_n$ denote the {\'e}tale sheaf $\Lambda_{\ell^n}(2)$. We consider the diagram
\begin{equation}\label{eqn:Main-Diag-0}
  \xymatrix@C.8pc{
    H^1_\zar(\sX, \sF_n) \ar[r] \ar[d]_-{\alpha_1} & H^1_\zar(X, \sF_n) \ar[r]
    \ar[d]^-{\beta_1} & 
    H^2_Y(\sX, \sF_n) \ar[d]^-{\gamma_2} \ar[r] & H^2_\zar(\sX, \sF_n) \ar[r]
    \ar[d]^-{\alpha_2} & H^2_\zar(X, \sF_n) \ar[d]^-{\beta_2} \\
H^{3}_\et(\sX, \sG_n) \ar[r] & H^{3}_\et(X, \sG_n) \ar[r] &
H^{4}_{\et, Y}(\sX, \sG_n) \ar[r] &
H^{4}_\et(\sX, \sG_n) \ar[r] & H^{4}_\et(X, \sG_n).}
\end{equation}

\begin{lem}\label{lem:Main-Diag-1}
  The diagram ~\eqref{eqn:Main-Diag-0} is commutative and all groups in this
  diagram are finite.
\end{lem}
\begin{proof}
The commutativity of ~\eqref{eqn:Main-Diag-0} is an easy consequence of the fact
that each square in the diagram is induced by a morphism of spectral
sequences. It works as follows.
Let $\iota \colon Y \inj \sX$ and $j \colon X \inj \sX$ be the
inclusion maps. We let $\sG^\bullet_X = j^*(\sG^\bullet)$ for any complex of sheaves
  $\sG$ on $\sX_\et$ or $\sX_\zar$.
Let $\sG_n \to \sI^\bullet$ be an injective resolution on $\sX_{\et}$ and let
  $\phi \colon \sI^\bullet \to j_*(\sI^\bullet_X)$ denote the canonical map. 
Since $j^*$ and $j_*$ have left adjoints (namely, $j_{!}$ and $j^*$,
respectively), $\phi$ is a morphism of injective cochain complexes.

We let $\sJ^\bullet = \epsilon_*(\sI^\bullet) \cong {\bf R}\epsilon_*(\sI^\bullet)$
  (where $\epsilon \colon \sX_\et \to \sX_\zar$ is the morphism of sites) and let
  $\sJ^\bullet \to \sL^{\bullet, \bullet}$
  be a Cartan--Eilenberg resolution of $\sJ^\bullet$ on $\sX_\zar$.
 Letting $\epsilon^X$ denote the restriction of $\epsilon$ to $X$, we then get a
  commutative diagram
  \begin{equation}\label{eqn:Main-Diag-1-0}
\xymatrix@C1.6pc{
  {\bf R}\epsilon_*(\sI^\bullet) \ar[r]^-{\simeq}  \ar[dr]
  & \sJ^\bullet \ar[r]^-{\phi} \ar[d] &
  j_*(\sJ^\bullet_X) \ar[r]^-{\simeq} &
  j_* \epsilon^X_*(\sI^\bullet_X) \ar[d] \ar[r]^-{\simeq} &
  {\bf R}j_* {\bf R}\epsilon^X_*(\sI^\bullet_X) \ar[dl] \\
  & \sL^{\bullet, \bullet} \ar[rr]^-{\phi} &  & j_*(\sL^{\bullet, \bullet}_X), & }
\end{equation}
in which the left vertical arrow is the Cartan--Eilenberg
resolution map, the right vertical arrow is obtained by applying $j_*$ to this
map, and the horizontal arrows are the restriction maps.

Letting $\sL^{\bullet}$ denote the  totalization of $\sL^{\bullet, \bullet}$ and
$G = G^\star \sL^{\bullet}$ denote the filtration of $\sL^{\bullet}$ given by
$G^r \sL^{\bullet} = \sL^{\ge r}$ as in the proof of \propref{prop:LC-SS},
we get a commutative diagram
 \begin{equation}\label{eqn:Main-Diag-1-1}
\xymatrix@C1.6pc{
  {\bf R}\epsilon_*(\sI^\bullet) \ar[d] \ar[r]^-{\phi} &
  {\bf R}j_* {\bf R}\epsilon^X_*(\sI^\bullet_X) \ar[d] \\
  \sL^{\bullet} \ar[r]^-{\phi} & j_*(\sL^\bullet_X)}
 \end{equation}
 and a canonical map of filtered complexes
 $\phi \colon (\sL^{\bullet}, G) \to (j_*(\sL^\bullet_X), j_*(G))$ on $\sX_\zar$.

If we let $\sM^\bullet = C(\phi)[-1]$ denote the shifted mapping cone of
 $\phi$, then $\sM^\bullet$ is canonically equipped with a filtration $
 H = H^\star\sM^\bullet$ such that we have canonical morphisms of filtered complexes
 \begin{equation}\label{eqn:Main-Diag-1-2}
   (\sM^\bullet, H) \xrightarrow{\psi} (\sL^{\bullet}, G) \xrightarrow{\phi}
   j_*(\sL^\bullet_X, G_X) \simeq {\bf R}j_*(\sL^\bullet_X, G_X)
   \xrightarrow{\delta}  (\sM^\bullet, H)[1],
   \end{equation}
 which give rise to a distinguished triangle of graded quotients
 \begin{equation}\label{eqn:Main-Diag-1-3}
   \gr^q(\sM^\bullet) \xrightarrow{\psi} \gr^q(\sL^{\bullet}) \xrightarrow{\phi}
   {\bf R}j_*(\gr^q(\sL^\bullet_X)) \xrightarrow{\delta}  \gr^q(\sM^\bullet)[1]
   \end{equation}
 for every $q \ge 0$.

It follows from ~\eqref{eqn:Main-Diag-1-2} and ~\eqref{eqn:Main-Diag-1-3}
   (cf. proof of \propref{prop:LC-SS}) that there are morphisms of
 spectral sequences
 \begin{equation}\label{eqn:Main-Diag-1-4}
   \xymatrix@C.6pc{
     E^{i,j}_2(\sL^{\bullet}, G) \ar@2{->}[d] \ar[r] & E^{i,j}_2(\sL^{\bullet}_X, G_X)
     \ar@2{->}[d] \ar[r] & 
     E^{i+1,j}_2(\sM^{\bullet}, H) \ar@2{->}[d] \ar[r] &  E^{i+1,j}_2(\sL^{\bullet}, G)
     \ar@2{->}[d] \ar[r] & E^{i+1,j}_2(\sL^{\bullet}_X, G_X) \ar@2{->}[d] \\
H^{i+j}_\et(\sX, \sG_n) \ar[r] & H^{i+j}_\et(X, \sG_n) \ar[r] &
\H^{i+j+1}_{\et}(\sM^\bullet) \ar[r] & H^{i+j+1}_\et(\sX, \sG_n) \ar[r] &
H^{i+j+1}_\et(X, \sG_n).}
\end{equation}
As ${\bf R}\Gamma(\sM^\bullet) \to
 {\bf R}\Gamma(\sL^\bullet) \to {\bf R}\Gamma {\bf R}j_*(\sL^\bullet_X)$ is
 a distinguished triangle, we moreover see that $\H^{i+j+1}_{\et}(\sM^\bullet)$ is
 canonically isomorphic to $H^{i+j+1}_{\et,Y}(\sX, \sG_n)$.

Taking $(i,j) = (1,2)$ and using ~\eqref{eqn:Main-Diag-1-3}, we get that the
 top (resp. bottom) row of ~\eqref{eqn:Main-Diag-1-4} coincides with the
 top (resp. bottom) row of ~\eqref{eqn:Main-Diag-0}. Since the vertical arrows of
 the latter diagram are, by definition, the one induced from the spectral
 sequences of ~\eqref{eqn:Main-Diag-1-4}, we obtain the commutativity of
 ~\eqref{eqn:Main-Diag-0}, as desired.

We now prove the finiteness claim. We begin with the bottom row of
 ~\eqref{eqn:Main-Diag-0}. Since $H^i_\et(\sX, \sG_n) \cong H^i_\et(Y, \sG_n)$
 by the proper base change theorem, 
 \cite[Rem.~3.5(c)]{Jannsen-MA} says that all cohomology groups on the bottom
 row except possibly $H^{4}_{\et, Y}(\sX, \sG_n)$ are finite. It follows that
 the latter group is also finite.
 For the top row, we first note that $\alpha_1$ is
 injective as shown in the proof of \lemref{lem:Unram-7}
 (cf. ~\eqref{eqn:Unram-7-0}). The same proof also
 applies to show that $\beta_1$ is injective. This implies that the first
 two terms from left in the top row of ~\eqref{eqn:Main-Diag-0} are finite.

By \thmref{thm:Unram-5}, we have $H^2_\zar(\sX, \sF_n) \cong {\CH_1(\sX)}/n$
 and the latter group is finite by \cite[Thm.~0.6]{Saito-Sato-Ann}.
 Since ${\CH_1(\sX)}/n \surj {\CH_0(X)}/n \cong H^2_\zar(X, \sF_n)$, it follows
 that all cohomology groups on the top row except possibly
 $H^{2}_{Y}(\sX, \sF_n)$ are finite. It follows that
 the latter group is also finite. This completes the proof.
 \end{proof}

 For any scheme $Z$, we let $H^i_\et(Z, \Z_\ell(j)) =  {\varprojlim}_n
H^i_\et(Z, \Lambda_{\ell^n}(j))$ and $H^i_\et(Z, {\Q_\ell}/{\Z_\ell}(j)) =
{\varinjlim}_n H^i_\et(Z, \Lambda_{\ell^n}(j))$.

\begin{cor}\label{cor:Main-Diag-2}
  Under the $(\star_\ell)$-condition, there exists a commutative diagram of exact
  sequences
  \begin{equation}\label{eqn:Main-Diag-2-0}
  \xymatrix@C.8pc{
    {\varprojlim}_n H^1_\zar(\sX, \sF_n) \ar[r] \ar[d]_-{\alpha_1} &
    {\varprojlim}_n  H^1_\zar(X, \sF_n) \ar[r] \ar[d]^-{\beta_1} & 
    {\varprojlim}_n  H^2_Y(\sX, \sF_n) \ar[d]^-{\gamma_2} \ar[r] &
    {\varprojlim}_n  H^2_\zar(\sX, \sF_n) \ar[d]^-{\alpha_2} \\
H^{3}_{\rm {\etl}}(\sX, {\Z}_\ell(2)) \ar[r] & H^{3}_{\rm {\etl}}(X, {\Z}_\ell(2)) \ar[r] &
H^{4}_{{\rm {\etl}}, Y}(\sX, {\Z}_\ell(2)) \ar[r] & H^{4}_{\rm {\etl}}(\sX, {\Z}_\ell(2)).}
\end{equation}
\end{cor}
\begin{proof}
  Since the two rows of ~\eqref{eqn:Main-Diag-0} are exact sequences of finite
  groups and its vertical arrows are compatible with the canonical map
  ${\Z}/{\ell}^{n+1} \to {\Z}/{\ell}^{n}$ (as one easily checks),
  we can pass to the limit to get the desired conclusion.
\end{proof}

\subsection{Theorem of Esnault--Wittenberg}\label{sec:EW-thm}
Recall from \cite[\S~4]{Esnault-Wittenberg} that if $F$ is a field and $V$
is a connected smooth projective variety over $F$, then one says that
$H^2_\et(V, {\Q}_\ell(1))$ is algebraic if the canonical inclusion
$\Pic(V)^{(\ell)} {\otimes}_{\Z_\ell} {\Q}_\ell \inj H^2_\et(V, {\Q}_\ell(1))$ is an
isomorphism. Since the $\ell$-adic Tate module $T_\ell(\Br(V))$ is a torsion-free
${\Z}_\ell$-module, the Kummer sequence tells us that the algebraicity of
$H^2_\et(V, {\Q}_\ell(1))$ is equivalent to the condition that $T_\ell(\Br(V))$
vanishes. We let $\ov{V}$ denote the base change of $V$ to a separable closure
of $F$.

The following result is due to Esnault--Wittenberg
\cite[\S~3,4]{Esnault-Wittenberg}.

\begin{prop}\label{prop:EW-main}
  Assume that the $(\star_\ell)$-condition holds. Assume also that one of the
  following two conditions holds.
  \begin{enumerate}
  \item
    $\F$ is finite, $\Alb_X$ has potentially good reduction and the irreducible
    components of $Y$ satisfy the Tate conjecture.
  \item
    $\F$ is separably closed and $ H^2_{\rm {\etl}}(\ov{X}, {\Q}_\ell(1))$ is algebraic.
  \end{enumerate}

  Then the map $\gamma_2 \colon {\varprojlim}_n  H^2_Y(\sX, \sF_n) \to
  H^{4}_{{\rm {\etl}}, Y}(\sX, {\Z}_\ell(2))$ is surjective.
\end{prop}
\begin{proof}
We look at the commutative diagram (cf. ~\eqref{eqn:cyc-cl-map})
  \begin{equation}\label{eqn:EW-main-0}
    \xymatrix@C1.6pc{
      \stackrel{m}{\underset{i=1}\bigoplus} {\varprojlim}_n H^2_{Y_i}(\sX, \sF_n)
      \ar[r]^-{\bigoplus \gamma_{2,i}} \ar[d] &
      \stackrel{m}{\underset{i=1}\bigoplus} 
      H^{4}_{\et, Y_i}(\sX, {\Z}_\ell(2)) \ar[d]^-{\psi} \\
      {\varprojlim}_n H^2_{Y}(\sX, \sF_n) \ar[r]^-{\gamma_2} &
      H^{4}_{\et, Y}(\sX, {\Z}_\ell(2)),}
  \end{equation}
  where the vertical arrows are the canonical maps induced by the inclusions 
  $Y_i \inj Y \inj \sX$. Using this diagram, it suffices to show that
  each $\gamma_{2,i}$ as well as the map $\psi$ is surjective.

To show that $\gamma_{2,i}$ is surjective, we shall use the exact sequence
  \begin{equation}\label{eqn:EW-main-1}
0 \to \CH_1(Y_i)^{(\ell)} \xrightarrow{\cyc} H^{2}_{\et}(Y_i, {\Z}_\ell(1))
\to T_\ell(\Br(Y_i)) \to 0.
  \end{equation}

  If $\F$ is finite, the Tate conjecture for $Y_i$ implies that
  the cycle class map $\CH_1(Y_i)^{(\ell)} \otimes_{\Z_\ell} \Q_{\ell} \to
  H^2_\et(Y_i, \Q_{\ell}(1))$ is surjective, as shown in the proof of
  \cite[Lem.~3.8]{Esnault-Wittenberg} (cf. \cite[Prop.~4.3]{Tate-Motives}).
  This implies by ~\eqref{eqn:EW-main-1} that the first map in
  ~\eqref{eqn:EW-main-1} is an isomorphism. We deduce from \thmref{thm:Unram-5}
  and the purity theorem for {\'e}tale cohomology
  that the cycle class map $\gamma_{2,i}$ is an isomorphism.
If $\F$ is separably closed, our assumption
  and \cite[Lem.~4.4, 4.5]{Esnault-Wittenberg} together imply that the
  cycle class map in ~\eqref{eqn:EW-main-1} is again an isomorphism.
  Equivalently, $\gamma_{2,i}$ is an isomorphism.

We now show that $\psi$ is surjective. By the purity and the duality
  theorems (cf. \cite[Prop.~2.5, 2.6]{Esnault-Wittenberg}), this is equivalent to
  showing that the restriction map
  \begin{equation}\label{eqn:EW-main-2}
H^j_\et(Y, {\Q_\ell}/{\Z_\ell}(1)) \to  \ \stackrel{m}{\underset{i=1}\bigoplus}
H^j_\et(Y_i, {\Q_\ell}/{\Z_\ell}(1))
  \end{equation}
  is injective, where $j = 3$ if $\F$ is finite and $j = 2$ is $\F$ is
  separably closed.
  But this follows by Lemma~3.8 and (3.10) of op. cit. when $\F$ is finite, and
  from the argument of \S~4.2 of loc. cit. when $\F$ is separably closed.
\end{proof}

\subsection{Theorem of Saito--Sato}\label{sec:Main-0-prf-0}
In this subsection, we shall use the main results of \cite{Saito-Sato-Ann} to
prove some lemmas about the cycle class maps $\alpha_i$ for $i =1,2$ and
$\beta_1$. We shall use these to prove Theorems~\ref{thm:Main-0} and
~\ref{thm:Main-12}.

\begin{lem}\label{lem:SS-main-0}
  Assume that the $(\star_\ell)$-condition holds. Then the map
  $\alpha_1 \colon H^1_\zar(\sX, \sF_n) \to H^{3}_{\rm {\etl}}(\sX, {\Z}/{\ell^n}(2))$
  is surjective and the map
  $\alpha_2 \colon H^2_\zar(\sX, \sF_n) \to H^{4}_{\rm {\etl}}(\sX, {\Z}/{\ell^n}(2))$
  is injective.
\end{lem}
\begin{proof}
We can use \lemref{lem:Unram-7} and
the diagram~\ref{eqn:Unram-7-0} to reduce the problem of showing the surjectivity
of $\alpha_1$ to showing that 
$\alpha^{\nr}_1 \colon H^1_\zar(\sX, \sH^2_{\nr}(\ell^n,2)) \to
  H^{3}_\et(\sX, {\Z}/{\ell^n}(2))$ (cf. ~\eqref{eqn:Unram-7-0}) is surjective.

By the coniveau spectral sequence ~\eqref{eqn:Unram-2-0} whose abutment is
  $H^{3}_\et(\sX, \Lambda_{\ell^n}(2))$, we get an exact sequence
  \begin{equation}\label{eqn:SS-main-0-0}
    0 \to E^{1,2}_2(c) \to H^{3}_\et(\sX, \Lambda_{\ell^n}(2)) \to
    KH_3(\sX, \Lambda_{\ell^n}),
  \end{equation}
  where the last term is the Kato homology whose definition is given in
  \cite[p.~1597]{Saito-Sato-Ann}.
Using \lemref{lem:Unram-8}, ~\eqref{eqn:SS-main-0-0}
  coincides with the exact sequence
  \begin{equation}\label{eqn:SS-main-0-1*}
    0 \to  H^1_\zar(\sX, \sH^2_{\nr}(\ell^n,2)) \to
    H^{3}_\et(\sX, \Lambda_{\ell^n}(2)) \to
    KH_3(\sX, \Lambda_{\ell^n}).
  \end{equation}
If $\F$ is separably closed, then $KH_3(\sX, \Lambda_{\ell^n})$ vanishes
for every $n \ge 1$ by \cite[Thm.~2.13]{Saito-Sato-Ann}
(see its proof on p.~1630). We derive the same conclusion using
  Lemma~8.8 and Theorem~2.13 of op. cit. if $\F$ is finite.

To prove the injectivity of $\alpha_2$, we let $n \ge 1$ and use the spectral
  sequences ~\eqref{eqn:Unram-2-0} and ~\eqref{eqn:LSS-0} whose common abutment is
  $H^{4}_\et(\sX, \Lambda_{\ell^n}(2))$. It follows by \propref{prop:LC-SS} and
  \lemref{lem:Unram-8} that we have a commutative diagram
  \begin{equation}\label{eqn:SS-main-0-2}
 \xymatrix@C1pc{
   E^{2,2}_2(l) \ar[r]^-{\cong} \ar[d] &  H^2_\zar(\sX, \sF_n)
   \ar[d]^{\epsilon^{\et}} \ar[r]^-{\alpha_2} & 
        H^{4}_\et(\sX, \Lambda_{\ell^n}(2)) \ar[d]^-{\id} \\
        E^{2,2}_2(c) \ar[r]^-{\cong} &  H^2_\zar(\sX, \sH^2_{\nr}(\ell^n,2))
        \ar[r]^-{\alpha^{\nr}_2} & H^{4}_\et(\sX, \Lambda_{\ell^n}(2)).}
  \end{equation}

Since the middle vertical arrow is an isomorphism by
  \lemref{lem:Unram-8} and \thmref{thm:Unram-5}, it suffices to show that
  $\alpha^{\nr}_2$ is injective. But this follows
  from \cite[Thm.~0.6, Defn.~1.12]{Saito-Sato-Ann}. This completes the proof.
  \end{proof}

Recall from \S~\ref{sec:VRC} that $H^i_{\nr}(X, {\Z}/{\ell^n}) :=
H^0_\zar(X, \sH^i({\Z}/{\ell^n}(i-1))), \
H^i_{\nr}(X, {\Q_\ell}/{\Z_\ell}) = {\varinjlim}_n H^i_{\nr}(X, {\Z}/{\ell^n})$
and $H^i_{\nr}(X, {\Z}_\ell) = {\varprojlim}_n H^i_{\nr}(X, {\Z}/{\ell^n})$.
The groups $H^i_{\nr}(\sX, {\Q_\ell}/{\Z_\ell})$ and
$H^i_{\nr}(\sX, {\Z}_\ell)$ are defined similarly.
The group $H^i_{\nr}(X, {\Q_\ell}/{\Z_\ell})$ is called the $i$-th
  unramified cohomology of $X$.

\begin{lem}\label{lem:EW-main-3}
  Under the hypotheses of \thmref{thm:Main-0} or \thmref{thm:Main-12},
  there exists an exact sequence
  \begin{equation}\label{eqn:EW-main-3-0}
    0 \to {\varprojlim}_n H^1_\zar(X, \sF_n) \xrightarrow{\beta_1}
    H^{3}_{\rm {\etl}}(X, \Z_{\ell}(2)) \to H^3_{\nr}(X, \Z_\ell) \to 0.
  \end{equation}
\end{lem}
\begin{proof}
For any $n \ge 1$, the Bloch--Ogus resolution for $\sH^i({\ell^n},2)$ on $X$
(cf. \lemref{lem:BO-exact})
and the Leray spectral sequence ~\eqref{eqn:LSS-0} together yield an exact
sequence
  \begin{equation}\label{eqn:EW-main-3-1}
  0 \to H^1_\zar(X, \sF_n) \xrightarrow{\beta_1} H^{3}_\et(X, \Lambda_{\ell^n}(2))
  \to H^3_{\nr}(X, \Lambda_{\ell^n}) \to H^2_\zar(X, \sF_n) 
  \xrightarrow{\beta_2} H^{4}_\et(X, \Lambda_{\ell^n}(2)).
  \end{equation}

Using \propref{prop:LC-SS} and diagram~\eqref{eqn:SS-main-0-2} for $X$,
 we get that ${\beta_2}$ coincides with the cycle class map
 $\rho_X \colon {\CH_0(X)}/{\ell^n} \to H^{4}_\et(X, \Lambda_{\ell^n}(2))$
 given in \cite[Defn.~1.12]{Saito-Sato-Ann}.
We can thus rewrite ~\eqref{eqn:EW-main-3-1} as
 \begin{equation}\label{eqn:EW-main-3-2}
 0 \to H^1_\zar(X, \sF_n) \xrightarrow{\beta_1} H^{3}_\et(X, \Lambda_{\ell^n}(2))
  \to H^3_{\nr}(X, \Lambda_{\ell^n}) \to  {\CH_0(X)}/{\ell^n}
  \xrightarrow{\rho_X} H^{4}_\et(X, \Lambda_{\ell^n}(2)). 
 \end{equation}

 The groups
 $H^3_\et(X, \Lambda_{\ell^n}(2))$ and $H^4_\et(X, \Lambda_{\ell^n}(2))$ are finite
 by \lemref{lem:Main-Diag-1}. We also showed that $ {\CH_0(X)}/{\ell^n}$ is finite
(see the end of the proof of \lemref{lem:Main-Diag-1}).
It follows that all terms of ~\eqref{eqn:EW-main-3-2} are finite.
We can therefore pass to the limit to
  get an an exact sequence
  \begin{equation}\label{eqn:EW-main-3-3}
    0 \to {\varprojlim}_n H^1_\zar(X, \sF_n) \xrightarrow{\beta_1}
    H^{3}_\et(X, \Z_{\ell}(2)) \to H^3_{\nr}(X, \Z_{\ell}) \hspace*{4cm}
  \end{equation}
  \[
  \hspace*{8cm} \to \CH_0(X)^{(\ell)}
  \xrightarrow{\rho_X} H^{4}_\et(X, \Z_{\ell}(2)). 
 \]
The exactness of ~\eqref{eqn:EW-main-3-0} now follows by 
\cite[Thm.~3.1, 4.1]{Esnault-Wittenberg} which say that $\rho_X$ is injective
under our assumptions. This completes the proof.
   \end{proof}

For $n \ge 1$, we let $\CH^i(X, j; {\Z}/{\ell^n})$ denote the $j$-th homology of
Bloch's cycle complex $Z^i(X, \bullet) \otimes {\Z}/{\ell^n}$.
For any field $F$ and $Z \in \Sch_F$, we let $H^i(Z, {\Z}/{\ell^n}(j))$
denote Voevodsky's motivic cohomology of $Z$ (cf. Definition~\ref{defn:MC-defn}). 

\begin{lem}\label{lem:Chow-finite}
  We have the following.
  \begin{enumerate}
  \item
  There are canonical isomorphisms
  \[
  H^1_\zar(X, {\sK_{2,X}}/{\ell^n}) \xrightarrow{\cong}
  H^1_\zar(X, \sH^2(\ell^n,2)) \xrightarrow{\cong} \CH^2(X, 1; {\Z}/{\ell^n})
  \xrightarrow{\cong} H^3(X, {\Z}/{\ell^n}(2)).
  \]
  \item
    The cycle class map $H^3(X, {\Z}/{\ell^n}(2)) \to H^3_{\rm {\etl}}(X, {\Z}/{\ell^n}(2))$
    is injective. 
  \item
  There is a canonical isomorphism
  $H^i_{\nr}(X, \Z_\ell) \cong T_\ell(H^i_{\nr}(X, {\Q_\ell}/{\Z_\ell}))$ for every
  $i \ge 1$.
  \end{enumerate}
\end{lem}
\begin{proof}
The first and the third isomorphisms in item (1) are well known (cf.
  \lemref{lem:Unram-8}, \cite{Voevodsky-IMRN}). The integral version of the
  second isomorphism was shown by Landsburg \cite[Thm.~2.5]{Landsburg} and his
  proof also works with finite coefficients. Alternatively, one can use the
    commutative diagram
    \begin{equation}\label{eqn:Chow-finite-0}
      \xymatrix@C1pc{
        0 \ar[r] &  {H^1_\zar(X, {\sK_{2,X}})}/{\ell^n} \ar[r] \ar[d] &
        H^1_\zar(X, {\sK_{2,X}}/{\ell^n}) \ar[r] \ar[d] &
        {H^2_\zar(X, {\sK_{2,X}})}[\ell^n]  \ar[r] \ar[d] & 0 \\
        0 \ar[r] & {\CH^2(X, 1)}/{\ell^n} \ar[r] & \CH^2(X, 1; {\Z}/{\ell^n})
        \ar[r] &  \CH^2(X)[\ell^n] \ar[r] & 0.}
    \end{equation}

The bottom row is clearly exact while the exactness of the top row is
    due to Suslin \cite[(4.4)]{Suslin-KT}). It also follows from a version of
    \corref{cor:Unram-4} for $X$ because ${\sK_{2,X}} \cong {\sK^{\nr}_{2,X}}$.
    The left vertical arrow is bijective by \cite[Thm.~2.5]{Landsburg} while the
  same for the right vertical arrow is the well-known Bloch's formula. The
    desired isomorphism follows.

Item (2) follows from item (1) and ~\eqref{eqn:EW-main-3-1} which shows that
$\beta_1$ is injective without the hypotheses of \thmref{thm:Main-0} or
\thmref{thm:Main-12}, and it
    is well-known that $\beta_1$ is the {\'e}tale realization map
    (equivalently, the cycle class map) under the isomorphism of item (1).

To prove item (3), we look at the exact sequence of Zariski sheaves
\[
\sH^{i-1}({\Q_\ell}/{\Z_\ell}(i-1)) \xrightarrow{\ell^n}
\sH^{i-1}({\Q_\ell}/{\Z_\ell}(i-1))
\to \sH^i({\Z}/{\ell^n}(i-1)) \to {\sH^i({\Q_\ell}/{\Z_\ell}(i-1))}[\ell^n] \to 0
\]
on $X$.
Using the Norm-residue isomorphism $\sK_{i-1, X} \otimes {\Q_l}/{\Z_\ell}
\xrightarrow{\cong} \sH^{i-1}({\Q_\ell}/{\Z_\ell}(i-1))$, we see that the first map
in the above exact sequence is surjective. 
Equivalently, the map
$\sH^i({\Z}/{\ell^n}(i-1)) \to {\sH^i({\Q_\ell}/{\Z_\ell}(i-1))}[\ell^n]$ is
bijective. Taking the global sections, we obtain $H^i_{\nr}(X, {\Z}/{\ell^n})
\xrightarrow{\cong} {H^i_{\nr}(X, {\Q_\ell}/{\Z_\ell})}[\ell^n]$ for every $n \ge 1$.
Passing to the limit, we get the isomorphism of item (3).
\end{proof}

\subsection{The main result}\label{sec:Main-0-prf}
We can now prove the main result of \S~\ref{sec:Unram-V} which will prove
Theorems~\ref{thm:Main-0} and ~\ref{thm:Main-12}.
Let the notation and hypotheses be as given at the outset of \S~\ref{sec:Unram-V}.
If $\F$ is finite, we let $\F' = \F(\sqrt{-1})$. 

Recall the $(\star_\ell)$-condition from ~\eqref{eqn:Star-condn} and
recall from \S~\ref{sec:VRC} that the $(\star\star_\ell)$-condition
holds for $\sX$ if one of the following conditions is satisfied.
  \begin{enumerate}
  \item
  $\F$ is finite, the $(\star_\ell)$-condition holds,
  $\Alb_X$ has potentially good reduction and the irreducible
    components of $Y$ satisfy the Tate conjecture.
  \item
    $\F$ is finite, $\Alb_X$ has potentially good reduction and the irreducible
    components of $Y_{\F'}$ satisfy the Tate conjecture.
  \end{enumerate}

 \begin{thm}\label{thm:Main-0-0}
  Assume one of the following.
  \begin{enumerate}
  \item  
    The  $(\star\star_\ell)$-condition holds.
  \item
    $\F$ is separably closed and $H^2_{\rm {\etl}}(\ov{X}, \Q_\ell(1))$ is algebraic.
  \end{enumerate}
  Then $H^3_{\nr}(X, \Z_\ell) = 0$. Equivalently, the cycle class map
  $H^3(X, \Z_\ell(2)) \to H^3_{\rm {\etl}}(X, \Z_\ell(2))$ is an isomorphism.
\end{thm}
\begin{proof}
  Suppose first that part (1) of the $(\star\star_\ell)$-condition holds.
  We look at the
  diagram~\eqref{eqn:Main-Diag-2-0}. The map $\alpha_1$ in this diagram is
  surjective and $\alpha_2$ is injective by \lemref{lem:SS-main-0} while
  $\gamma_2$ is surjective by \propref{prop:EW-main}. A diagram chase shows
  that $\beta_1$ is surjective. We now apply \lemref{lem:EW-main-3} to
  conclude the proof.

Suppose next that part (2) of the $(\star\star_\ell)$-condition holds.
Using the push-forward of cycles, it is then easy to see that the Tate
conjecture holds
for all irreducible components of $Y$ too (cf. \cite[\S~2, p.~580]{Totaro}).
If the $(\star_\ell)$-condition holds, this allows us to refer to condition (1) to
finish the proof.
We assume now that the $(\star_\ell)$-condition does not hold. In particular,
$\ell = 2$ and $p$ is an odd prime.

Let $R \to R'$ be the unique finite {\'e}tale extension of rings corresponding
to the field extension ${\F'}/{\F}$. We let $k'$ be the quotient field of the
excellent hdvr $R'$.
We let $\sX' = \sX \times_{\Spec(R)} \Spec(R')$ and
let $X'$ be the generic fiber of $\sX'$.
Then the $(\star_\ell)$-condition holds for $\sX'$.
Furthermore, it is easily seen that $\Alb_{X'} \cong (\Alb_X)_{k'}$ has potentially
good reduction. We conclude from the previous case that
$H^3_{\nr}(X', \Z_\ell) = 0$.

We now let $f \colon X' \to X$ denote the projection map and look at the
diagram
\begin{equation}\label{eqn:Main-0-0-1}
  \xymatrix@C1pc{
    H^3(X', \Z_2(2)) \ar[r] \ar[d]_-{f_*} & H^3_\et(X', \Z_2(2)) \ar[r]
    \ar[d]^-{f_*} & H^3_{\nr}(X', \Z_2) \ar[r] \ar[d] & 0 \\
  H^3(X, \Z_2(2)) \ar[r] & H^3_\et(X, \Z_2(2)) \ar[r] \ar[d] &
  H^3_{\nr}(X, \Z_2) \ar[r]  & 0 \\
  & \coker(f_*) \ar[d] & \\
  & 0 & }
\end{equation}
in which the horizontal arrows in the left square are the {\'e}tale realization
maps (which we call the cycle class maps)
from Voevodsky's motivic cohomology to {\'e}tale motivic cohomology.
The rows are exact by \lemref{lem:EW-main-3}.

It is well-known that the above square is commutative
(cf. \cite[Thm.~4.3.2]{Deglise-Tata}).
Since $f_* \circ f^* \colon H^3_\et(X, \Z_2(2)) \to H^3_\et(X, \Z_2(2))$ is
multiplication by two, it follows that
$\coker(f_*)$ is a 2-torsion group. Since $H^3_{\nr}(X', \Z_2) = 0$,
we get that $\coker(f_*) \surj H^3_{\nr}(X, \Z_2)$. In particular,
$H^3_{\nr}(X, \Z_2)$ is a 2-torsion group. But this forces $H^3_{\nr}(X, \Z_2)$
to be trivial since \lemref{lem:Chow-finite} implies that this group is
torsion-free.

Finally, suppose that the hypothesis (2) of the theorem is satisfied.
  If $\ell$ is odd, then the $(\star_\ell)$-condition holds. If $\ell = 2$, then
  $\ell \in \F^\times$ by our assumption. In particular, $\F(\sqrt{-1})$ is a
  separable extension of $\F$. But this implies that $\sqrt{-1} \in \F$ as
  $\F$ is separably closed. As $R$ is henselian, it follows that the
  $(\star_\ell)$-condition holds in this case too. Now, we proceed exactly as
  we argued under the assumption that part (1) of the
  $(\star\star_\ell)$-condition holds.
That the vanishing of $H^3_{\nr}(X, \Z_\ell)$ is equivalent to the bijectivity
of the cycle class map, follows from Lemmas~\ref{lem:EW-main-3} and
~\ref{lem:Chow-finite}. This completes the proof.
\end{proof}

We end this section with the following corollary of \lemref{lem:SS-main-0}
which proves \thmref{thm:Main-0-0} for $\sX$ under much weaker conditions.

\begin{thm}\label{thm:Main-0-2}
  Under the $(\star_\ell)$-condition, we have $H^3_{\nr}(\sX, {\Z}/{\ell^n}) = 0$.
\end{thm}
\begin{proof}
Using Lemmas~\ref{lem:Unram-8} and ~\ref{lem:Unram-0}, \thmref{thm:Unram-5}(4)
  and the Leray spectral sequence ~\eqref{eqn:LSS-0}, we get an exact sequence
  \[
  0 \to H^1_\zar(\sX, {{\sK}_{2, \sX}}/{\ell^n}) \xrightarrow{\alpha_1}
  H^3_\et(\sX, \sG_n) \to H^3_{\nr}(\sX, {\Z}/{\ell^n}) \to
   {\CH_1(\sX)}/{\ell^n} \xrightarrow{\alpha_2} H^4_\et(\sX, \sG_n).
\]
Lemma~\ref{lem:SS-main-0} implies that $\alpha_1$ is surjective, and we
showed in its proof that $\alpha_2$ is an isomorphism. The desired vanishing
follows.
\end{proof}

\section{Suslin homology of schemes}\label{sec:Suslin-review}
In this section, we recall motivic cohomology with compact support and the
properties that we will use. We then recall Suslin homology and establish the
main result of the section, relating Suslin homology to motivic cohomology with
compact support with finite coefficients.

\subsection{Review of motivic cohomology}\label{sec:MC-review}
We fix a field $k$ of exponential characteristic $p \ge 1$ and fix an integer
$n$ such that $(p,n) = 1$. 
We let $\Lambda$ be a commutative ring which we take to be $\Z$ if $p = 1$ or
any $\Z[\tfrac{1}{p}]$-algebra if $p \ge 2$. 
We let $\Nsch_{k}$ denote the category of Noetherian $k$-schemes of finite Krull
dimension. Let $R \in \{\Lambda, \Lambda_n\}$, where recall that $\Lambda_n =
{\Z}/n$.

By \cite[Thm.~11.4.5]{CD-Springer}, every $X \in \Nsch_{k}$ has associated to it
a monoidal triangulated category of integral mixed motives $\dm(X, R)$ which
coincides with the original construction of Voevodsky \cite{Voe00} when $X$ is the
spectrum of a perfect field and of Suslin \cite{Suslin-AKT} when $X$ is the
spectrum of an arbitrary field. The assignment
$X \mapsto \dm(X, R)$ on $\Nsch_{k}$
satisfies several axioms of Grothendieck's six functor formalism
(cf. \cite[Thm.~1, p.~xviii]{CD-Springer}).  
The constant  Nisnevich sheaf with transfer associated to the ring $R$ is the
identity object for the monoidal structure of $\dm(X, R)$.
We shall denote this object by $R_X$. If $X = \Spec(A)$, we shall use
$R_A$ as another notation for $R_X$.

For $X \in \Sch_k$ with the structure map $f \colon X \to \Spec(k)$, we let
$M_k(X) = f_\sharp R_X \cong f_{!} f^{!} R_k$ and
$M^c_k(X) = f_* f^! R_k$ as objects of $\dm(k, R)$.
If $X$ is complete, then $f_{!} = f_*$ and hence $M^c_k(X) = M_k(X)$.
Recall that  $R_k(1) =
{R_{\rm tr}(\P^1_k)}/{R_{\rm tr}(\Spec(k))}$ as a presheaf with transfer
via the identification of $\Spec(k)$ with the residue field of
$\infty \in \P^1_k$.
In particular, $R_k(1) \cong M_k({\P^1_k}/{\Spec(k)}) \in  \dm(k, R)$.
For $j \ge 0$, we let $R_k(j)$ denote the $j$-fold self-product of $R_k(1)$
under the monoidal structure of $\dm(k, R)$.
We shall drop the field $k$ from the notation of motives of objects in $\Sch_k$
except when we need to vary the base field.

\begin{defn}\label{defn:MC-defn}
  For $i,j \ge 0$, we let
  \[
  H_i(X, R(j)) = \Hom_{\dm(k, R)}( R_{k}(j)[i], M(X));
  \]
\[
H^i(X, R(j)) = \Hom_{\dm(k, R)}(M(X), R_{k}(j)[i]);
\]
\[
\hspace*{2cm} {\cong}^{1} \ \Hom_{\dm(k, R)}(R_k, f_* R_{X}(j)[i]);
\]
\[
  H^i_c(X, R(j)) = \Hom_{\dm(k, R)}(M^c(X),  R_{k}(j)[i])
  \]
  \[
  \hspace*{2cm} {\cong}^2 \ \Hom_{\dm(k, R)}(R_k,  f_! R_{X}(j)[i]),
  \]
\end{defn}
where the isomorphisms ${\cong}^1$ and ${\cong}^2$
are deduced using the adjointness relations among the six functors
(cf. \cite[Lem.~5.5]{GKR}).

By \cite[Thm.~5.1]{CD-Doc} and \cite[Thm.~2.14]{Krishna-Pelaez},
there are functorial isomorphisms
\begin{equation}\label{eqn:MV-sing-0} 
\begin{array}{lll}
  H^p(X, \Lambda(q))   &
  \xrightarrow{\,\,\cong\,\,} &
  \Hom_{\sS\sH(X)}(\Lambda_X,  \Sigma^{p,q} (H\Lambda_X)) \\
  & \xrightarrow{\,\,\cong\,\,} &
  \Hom_{\sS\sH_{\rm cdh}(k)}(\Sigma^\infty_T X_+, \Sigma^{p,q}H\Lambda),
  \end{array}
\end{equation}
where $\sS\sH(X)= \sS\sH(X, \Lambda)$ is the monoidal stable homotopy
category of smooth schemes over $X$ (with $\Lambda$-coefficients and unit
object $\Lambda_X$) and $\sS\sH_{\rm cdh}(k)$ is the stable homotopy category
of $\Sch_k$ with respect to the cdh topology. $H\Lambda$ is the motivic
Eilenberg-MacLane spectrum and $H\Lambda_X = {\bf L}f^*(H\Lambda)$, where
$f \colon X \to \Spec(k)$ is the structure map.
We refer to, e.g., \cite[\S~2,3]{Krishna-Pelaez}
for the definitions of the suspension operators $\Sigma^\infty_T$ and
$\Sigma^{p,q}$.

The following properties of motivic cohomology will be used repeatedly in this
paper.
If $\iota \colon Y \inj X$ is  a closed immersion in $\Sch_k$ and
$j \colon U \inj X$ is the inclusion of the complement of $Y$, then one
has a natural long exact sequence (cf. \cite[Lem.~5.4]{GKR})
\begin{equation}\label{eqn:Mot-loc}
  \cdots  \xrightarrow{j_{*}} H^i_c(X, R(j)) \xrightarrow{\iota^*}
  H^i_c(Y, R(j))  \xrightarrow{\partial} H^{i+1}_c(U, R(j))
    \xrightarrow{j_{*}} H^{i+1}_c(X, R(j)) \xrightarrow{\iota^*} \cdots .
    \end{equation}

\begin{lem}\label{lem:Vanishing-coh}
    Let $X \in \Sch_k$ have dimension $d$. Then we have the following.
    \begin{enumerate}
    \item
      $H^{2a-b}(X, R(a)) = 0 = H^{2a-b}_c(X, R(a))$ if $a > d+b$.
    \item
      If $X \inj X'$ is a closed immersion with complement $U$, the canonical
      map \\
      $H^{2a-b+1}_c(U, R(a)) \to H^{2a-b+1}_c(X', R(a))$ is an isomorphism if
      $a > d+b$.
    \item
      The canonical maps ${H^{2a-b-1}(X, R(a))}/n \to H^{2a-b-1}(X, \Lambda_n(a))$
      and \\
      ${H^{2a-b-1}_c(X, R(a))}/n \to H^{2a-b-1}_c(X, \Lambda_n(a))$
      are isomorphisms if $a > d+b$.
      \end{enumerate}
  \end{lem}
  \begin{proof}
    We let $k^\hs$ be a perfect closure of $k$ and let $Z^\hs$ denote the base
    change of a $k$-scheme $Z$ to $k^\hs$. By \cite[Prop.~8.1]{CD-Doc} (see also
    \cite[Cor.~2.19]{Elmanto-Khan}), we can replace $X$ by
    $X^\hs$ to prove the vanishing of $H^{2a-b}(X, R(a))$. This allows us to assume
    that $k$ is perfect in which case the claim follows by
    \cite[Thms.~2.14, 5.1]{Krishna-Pelaez}. To prove the vanishing of
    $H^{2a-b}_c(X, R(a))$,
 we choose a compactification $\ov{X}$ of $X$ of dimension $d$
 (cf. \cite{Nagata}).
    We let $Y = \ov{X} \setminus X$ with the reduced
    closed subscheme structure.

    In the localization exact sequence (cf. ~\eqref{eqn:Mot-loc})
    \[
    H^{2a-b-1}(Y, R(a)) \to H^{2a-b}_c(X, R(a)) \to H^{2a-b}(\ov{X}, R(a)) \to
    H^{2a-b}(Y, R(a)),
    \]
    all cohomology groups, except possibly $H^{2a-b}_c(X, R(a))$, vanish by what
    we just proved. It follows that $H^{2a-b}_c(X, R(a))$ also vanishes.

    Item (2) of the lemma follows from item (1) and the localization
    sequence ~\eqref{eqn:Mot-loc} while item (3) follows from  item (1) and the
    exact sequence
    \[
    0 \to {H^{2a-b-1}(X, R(a))}/n \to H^{2a-b-1}(X, \Lambda_n(a)) \to
    H^{2a-b}(X, R(a))[n] \to 0
    \]
    (and a similar exact sequence for cohomology with compact support)
    given by the universal coefficient theorem.
\end{proof}

Suppose now that $R = \Lambda_n$. For $X \in \Sch_k$, we let $\dm_\et(X, R)$ denote the
  monoidal triangulated category of {\'e}tale sheaves of $R$-modules with transfer on
  $\Sm_X$. Let $\dr_\et(X, R)$ denote the unbounded derived category of sheaves of
  $R$-modules on $X_\et$. Let $\epsilon_X \colon X_\et \to X_\zar$ be the
  morphism of sites. By  \cite[\S~1,2]{CD-Comp}
  and \cite[\S~9]{CD-Doc}, the assignment $X \mapsto \dm_\et(X, R)$ is a presheaf
  of monoidal  triangulated categories on $\Sch_k$ which is endowed with the
  six functor formalism. Furthermore, there are monoidal functors
\begin{equation}\label{eqn:MC-01}
  \dm(X, R) \xrightarrow{\epsilon^*_X} \dm_\et(X, R) \xrightarrow{\tau_X} \dr_\et(X, R),
\end{equation}
where $\epsilon^*_X$ is the {\'e}tale sheafification functor and $\tau_X$ 
is obtained by forgetting the transfer structure in $\dm_\et(X, R)$. The
latter is an equivalence of monoidal triangulated categories
(cf. \cite[Thm.~4.5.2]{CD-Comp}).
These functors commute with the six functors when restricted to the subcategories of
constructible objects (cf. \cite[Rem.~9.6]{CD-Doc}). We let
$M_\et(X) = \epsilon^*_k(M(X))$ for $X \in \Sch_k$. One similarly defines
$M_{\et,c}(X)$.
  
For $f \colon X \to \Spec(k)$ in $\Sch_k$, we let
\begin{equation}\label{eqn:RE-1}
H^i_{\et, \sM}(X, R(j)) =  \Hom_{\dm_\et(k, R)}(R_k, f_*R_X(j)[i]);
  \end{equation}
\[
  H^i_{\et, \sM, c}(X, R(j)) = \Hom_{\dm_\et(k, R)}(R_k, f_!R_X(j)[i]).
\]

For $j \in \Z$, we let $R'(j)$ be the sheaf on $X_\et$ defined so that
$R'(0)= R$, $R'(j) = \mu_m^{\otimes j}$ when $j >0$,
$R'(j) = \un{\Hom}(R'(-j), R)$ when $j<0$, where the latter is the
internal hom in the category of sheaves of abelian groups on $X_\et$ 
(cf. \cite[\S~V.1, p.~163]{Milne-etale}). 
Then the functor $\tau_X$ induces an isomorphism
$H^i_{\et, \sM}(X, R(j)) \xrightarrow{\cong} H^i_{\et}(X, R'(j))$
for every $X \in \Sch_k$ (cf. \cite[Lem.~6.2]{GKR}).
We shall use the common notation  $H^i_{\et}(X, R(j))$ for these isomorphic groups.
One similarly has an isomorphism
$H^i_{\et, \sM, c}(X, R(j)) \xrightarrow{\cong}  H^i_{\et,c}(X, R'(j))$
for every $X \in \Sch_k$, and we shall write these isomorphic groups as
$H^i_{\et,c}(X, R(j))$.

Let $X \in \Sch_k$ and let $\iota \colon Y \inj X$ be a closed
immersion. Let $j \colon U \inj X$ be the inclusion of the complement of $Y$.
In $\dm(k, \Lambda_n)$, we have a commutative diagram of distinguished triangles
(e.g, see the proof of of  \cite[Lem.~5.5]{GKR})
\begin{equation}\label{eqn:MC-02}
  \xymatrix@C1.7pc{
    M^c(Y) \ar[r] \ar[d] & M^c(X) \ar[r] \ar[d] & M^c(U) \ar[d] \\
    (\epsilon_k)_*(M^c_\et(Y)) \ar[r] &  (\epsilon_k)_*(M^c_\et(X)) \ar[r] &
    (\epsilon_k)_*(M^c_\et(U)),}
\end{equation}
where the vertical arrows are the unit of adjunction maps.

This gives rise to a commutative diagram of the long exact cohomology sequences 
\begin{equation}\label{eqn:MC-03}
  \xymatrix@C.5pc{
    \cdots  \ar[r] & H^{i}_c(U, \Lambda_n(j)) \ar[r]^-{j_{*}} \ar[d] &
    H^i_c(X, \Lambda_n(j)) \ar[r]^-{\iota^*} \ar[d] & H^i_c(Y, \Lambda_n(j))
    \ar[r]^-{\partial_\zar} \ar[d] &  H^{i+1}_c(U, \Lambda_n(j))
    \ar[r]^-{j_{*}} \ar[d] & \cdots \\
    \cdots \ar[r] & H^{i}_{\et,c}(U, \Lambda_n(j)) \ar[r]^-{j_{*}} &
    H^i_{\et,c}(X, \Lambda_n(j)) \ar[r]^-{\iota^*} & H^i_{\et,c}(Y, \Lambda_n(j))
    \ar[r]^-{\partial_\et}  &  H^{i+1}_{\et,c}(U, \Lambda_n(j))
    \ar[r]^-{j_{*}}  & \cdots,}
\end{equation}
where the vertical arrows are the {\'e}tale realization (cycle class) maps.

In the above diagram, we
can replace the motivic (resp. {\'e}tale) cohomology with compact support for
$X$ and $Y$ by their usual motivic (resp. {\'e}tale) cohomology if $X$ is proper
over $k$. Furthermore, this diagram is compatible with the canonical
maps $H^i_c(-, \Lambda_n(j))  \to H^i_c(-, \Lambda_m(j))$
and $H^i_{\et,c}(-, \Lambda_n(j))  \to H^i_{\et,c}(-, \Lambda_m(j))$
induced by the inclusion $\Lambda_n \inj \Lambda_m$ given by $a \mapsto a (m/n)$,
whenever $n|m$. It is also compatible with the canonical
maps $H^i_c(-, \Lambda_m(j))  \to H^i_c(-, \Lambda_n(j))$
and $H^i_{\et,c}(-, \Lambda_m(j))  \to H^i_{\et,c}(-, \Lambda_n(j))$
induced by the quotient map $\Lambda_m \surj \Lambda_n$ whenever $n|m$.
We let $H^i_c(-, \Z_\ell(j)) = {\varprojlim}_n H^i_c(-, {\Z}/{\ell^n}(j))$ and
$H^i_c(-, {\Q_\ell}/{\Z_\ell}(j)) = {\varinjlim}_{n} H^i_c(-, {\Z}/{\ell^n}(j))$.
We shall use similar notations for the corresponding {\'e}tale cohomology groups.

We shall need the following known result a few times.
  
\begin{lem}\label{lem:Etale-fin}
  Assume $k$ is either a finite or a henselian discrete valuation field (hdvf)
  with finite or separably closed residue field. Then $H^{i}_{{\rm {\etl}},c}(X, \Lambda_n(j))$
  is finite for all $i$ and $j$.
\end{lem}
\begin{proof}
We choose a compactification $X \inj \ov{X}$ (cf. \cite{Nagata}) and let
$Z = \ov{X} \setminus X$. Using ~\eqref{eqn:MC-03}, we can then assume that $X$ is
proper over $k$ in which case the result is well-known
(e.g., see \cite[Rem.~3.5(2)]{Jannsen-MA}).
\end{proof}

\subsection{The Suslin homology}\label{sec:SHom}
  Let the notation be as in the previous subsection. We let $X$ be
  a connected, regular $k$-scheme of dimension $d$.
  Recall that the Suslin-Voevodsky singular homology $H^S_0(X)$ of $X$
  (also called the
Suslin homology in the literature) is deﬁned as the $0$-th homology of a certain
explicit complex of algebraic cycles, introduced by Suslin--Voevodsky
\cite{Suslin-Voevodsky}.
Instead of recalling this whole complex, we recall the following equivalent
definition of $H^S_0(X)$. This equivalence was shown by Schmidt
\cite[Thm.~5.1]{Schmidt-ANT} (see also \cite[Defn.~1.1]{Yamazaki-Nagoya}).

For $C \in \sC(X)$, we let $\ov{C}$ denote the
unique compactification of $C$ and let $\nu \colon \ov{C}_n \to \ov{C}$ denote
the normalization map. We let
$\Delta(C) = \ov{C}_n \setminus C_n = \ov{C}_n \setminus \nu^{-1}(C)$ and
$U(C) = \Ker(\sO^\times_{\ov{C}_n, \Delta(C)} \surj \sO^\times_{\Delta(C)})$. 
Then $H^S_0(X)$ is the quotient of $Z_0(X)$ by the subgroup generated by
$\{\divf(f)| C \in \sC(X), f \in U(C)\}$.

The Suslin homology $H^S_0(-)$ has the following properties which will be useful.

\begin{prop}\label{prop:Suslin-hom}
  The assignment $X \mapsto H^S_0(X)$ satisfies the following.
  \begin {enumerate}
 \item
    It is covariantly functorial in $X$.
  \item
    If $f \colon X' \to X$ is a finite dominant
    map of non-singular varieties over $k$, then
    the pull-back map $f^* \colon Z_0(X) \to Z_0(X')$ induces
    $f^* \colon H^S_0(X) \to H^S_0(X')$ such that $f_* \circ f^*$ is
    multiplication by $\deg(f)$.
  \item
    If $j \colon U \inj X$ is an open immersion, then the map
    $H^S_0(U) \to H^S_0(X)$
    is surjective.
  \item
    If $\Char(k) = 0$, then there is a canonical isomorphism
    $H^S_0(X) \xrightarrow{\cong} H^{2d}_{c}(X, \Z(d))$.
  \item
    If $X$ is proper over $k$, then the identity map of $Z_0(X)$ induces an
    isomorphism $H^S_0(X) \xrightarrow{\cong} \CH_0({X})$.
  \item
    If $X \in \Sm_k$ and ${k'}/{k}$ is a purely inseparable extension, then there
    exists a pullback map ${H^S_0(X)}/n \to{H^S_0(X_{k'})}/n$ which is an
    isomorphism.  
    \item
    If $X \in \Sm_k$, there is a natural isomorphism
    ${H^S_0(X)}/n \xrightarrow{\cong} H_0(X, \Lambda_n(0))$.
\end{enumerate}
  \end{prop}
 \begin{proof}
For items (1), (2) and (3), see \cite[\S~2, p.~191, Prop.~5.6]{Schmidt-ANT}.
  Item (4) follows from \cite[Prop.~14.18]{MVW} and the isomorphism between the
  motivic homology and motivic cohomology with compact support, given at the
  end of \cite{Friedlander-Voevodsky}. Item (5) is clear from the definitions
  of $H^S_0(X)$ and $\CH_0(X)$. The proof of item (6) is identical to
  (in fact, slightly easier than) that of \cite[Cor.~3.17]{GKR}, we only have to
  replace $K^M_i(A)$ by $K^M_{i-1}(A)$ everywhere in the proof for all local
  rings $A$. To prove item (7), we can assume $k$ to be perfect by item (6) and
  \cite[Prop.~8.1]{CD-Doc}. The latter case follows from \cite[Prop.~14.18]{MVW}.
 \end{proof}

Recall from \cite[Lem.~5.9, 5.10]{GKR} that for a closed point $x \in X$, there
 exists a natural Gysin homomorphism
 $\iota^x_* \colon H^i(k(x), \Lambda_n(j)) \to H^{2d+i}(X, \Lambda_n(d+j))$. 
 We also have the Chern class map $\Psi_{k(x)} \colon K_0(k(x)) \xrightarrow{\cong}
 H^0(k(x), \Z(0))$ which is an isomorphism.
For an Artinian local ring $A$, we let $\ell(\Spec(A))$ denote the length of $A$.

\begin{prop}\label{prop:Suslin-hom-0}
If $X \in \Sm_k$, there are canonical isomorphisms 
\[
\Psi_X \colon {H^S_0(X)}/n \xrightarrow{\cong} {H^{2d}_c(X, \Lambda(d))}/n
\xrightarrow{\cong} H^{2d}_c(X, \Lambda_n(d))
    \]
    such that the diagram
    \begin{equation}\label{eqn:Tame-MCCS-0}
      \xymatrix@C1pc{
        {K_0(k(x))}/n \ar[r]^-{\Psi_{k(x)}} \ar[d]_-{\iota^x_*} &
        H^0(k(x), \Lambda_n(0)) \ar[d]^-{\iota^x_*} \\
        {H^S_0(X)}/n \ar[r]^-{\Psi_X} & H^{2d}_c(X, \Lambda_n(d))}
    \end{equation}
    is commutative for every $x \in X_{(0)}$, where the left vertical arrow is the
    canonical map.
  \end{prop}
\begin{proof}
By \lemref{lem:Vanishing-coh}(3), we only have to show that
  $\Psi_X \colon  {H^S_0(X)}/n \to H^{2d}_c(X, \Lambda_n(d))$ is an isomorphism.
In view of \propref{prop:Suslin-hom}, we can assume that $p \ge 2$.
We first assume that $k$ is perfect. In this case, $x \in \Sm_k$ for all
$x \in X_{(0)}$, and the claim follows from item (7) of \propref{prop:Suslin-hom} and
\cite[Thm.~5.3.18]{Kelly}.

Suppose now that $k$ is not perfect and let $k^{\hs}$ be a chosen perfect closure of
$k$. Let $X^\hs = X_{k^{\hs}}$ and let $f \colon X^\hs \to X$ be the projection map.
For $x \in X_{(0)}$ and the unique point $y \in X^\hs$ lying over $x$, we let
$f^*_x \colon H^0(k(x), \Lambda_n(0)) \to H^0(k(y), \Lambda_n(0))$ be given by
$f^*_x(a) = \ell(S_x) \tau^*_x(a)$, where $S_x = \Spec(k(x)) \times_X X^\hs$ and
$\tau_x \colon \Spec(k(y)) \to \Spec(k(x))$ is the projection map.
We define $f^*_x \colon {K_0(k(x))}/n \to {K_0(k(y))}/n$ in the similar manner.
We let $\wt{f}^* = {\underset{x \in X_{(0)}}\bigoplus} f^*_x$. We let
$\alpha = {\underset{x \in X_{(0)}}\bigoplus} \Psi_{k(x)}$ and define
$\alpha^\hs$ in the similar manner.

We now consider the diagram
\begin{equation}\label{eqn:Tame-MCCS-3}
  \xymatrix@C1pc{
    {\underset{x \in X_{(0)}}\bigoplus} {K_0(k(x))}/n \ar[rr]^-{\alpha}
    \ar@{->>}[dr]  \ar[dd]_-{\wt{f}^*} & &
        {\underset{x \in X_{(0)}}\bigoplus} H^0(k(x), \Lambda_n(0))
        \ar[dd]^>>>>>>>>{\wt{f}^*}
    \ar[dr] & \\
& {H^S_0(X)}/n \ar[dd]_>>>>>>>>{f^*} \ar@{.>}[rr] & & H^{2d}_c(X, \Lambda_n(d))
\ar[dd]^-{f^*} \\
{\underset{x \in X^\hs_{(0)}}\bigoplus} {K_0(k(x))}/n \ar[rr]^->>>>>>>>>{\alpha^\hs}
\ar@{->>}[dr]  & &
{\underset{x \in X^\hs_{(0)}}\bigoplus} H^0(k(x), \Lambda_n(0)) \ar[dr] & \\
& {H^S_0(X^\hs)}/n \ar[rr]^-{\Psi_{X^\hs}} & &  H^{2d}_c(X^\hs, \Lambda_n(d)).}
\end{equation}

We showed above that $\Psi_{X^\hs}$ is an isomorphism. The right vertical arrow of
the front face is an isomorphism by \cite[Lem.~5.9]{GKR} and
\cite[Cor.~2.1.7]{Elmanto-Khan}. Its left vertical arrow is an
isomorphism by \propref{prop:Suslin-hom}(6). It follows that there is a unique
isomorphism $\Psi_X \colon {H^S_0(X)}/n \xrightarrow{\cong} H^{2d}_c(X, \Lambda_n(d))$
such that the front face of ~\eqref{eqn:Tame-MCCS-3} commutes. To complete the proof
of the proposition, it remains to show that its top face commutes.
Since the right vertical arrow of the front face is an isomorphism, it suffices to
show that all other faces of ~\eqref{eqn:Tame-MCCS-3} commute.

Now, the left face of ~\eqref{eqn:Tame-MCCS-3} commutes by the construction of the
pull-back map $f^*$ between the Suslin homology groups (cf. \cite[\S~3.3]{GKR}).
The right face commutes by \cite[Cor.~5.13]{GKR}.
The back face commutes directly by the definition of all its arrows. The bottom face
commutes by using ~\eqref{eqn:Tame-MCCS-0} for $X^{\hs}$. We have already shown
that the front face commutes. This completes the proof.
\end{proof}

\section{Left kernel of Brauer--Manin pairing}\label{sec:BM-left}
The goal of this section is to construct the Brauer--Manin pairing for open surfaces
and prove Theorems~\ref{thm:Main-4} and ~\ref{thm:Main-5}.
We begin with the definition of the class group.

\subsection{The class group of open varieties}\label{sec:CG-defn}
We now fix a prime number $p$ and let $R$ be an hdvr with finite residue field
$\F$ of characteristic $p$.
We let $k$ denote the quotient field of $R$. We let
$X$ be a connected, smooth $k$-scheme of dimension $d$.
For $i \ge 1$, we let $H^i(X) :=  H^i_\et(X, {\Q}/{\Z}(i-1))$.
For a local ring $A$, we let $\wh{A}$ denote the completion of $A$ with respect to
its maximal ideal.
We begin by recalling the definition of the class group of $X$ from
\cite[\S~9]{KRS}.

For any $C \in \sC(X)$, we let $k(C)$ be the function field of $C$ and let
$k(C)_x$ denote the quotient field of $\wh{\sO_{\ov{C}_n,x}}$ for
$x \in (\ov{C}_n)_{(0)}$.
We let $k(C)^\times_\infty = {\underset{x \in \Delta(C)}\bigoplus} k(C)^\times_x$.
We then have the diagonal inclusion $\delta_C \colon k(C)^\times \inj
k(C)^\times_\infty$.
We let $\partial_C \colon k(C)^\times \to Z_0(X)$ be the boundary map.
The class group $C_0(X)$ of $X$ is defined to be the cokernel of the map
\begin{equation}\label{eqn:Class-grp}
{\underset{C \in \sC(X)}\sum} (\partial_C,\delta_C) \colon
{\underset{C \in \sC(X)}\bigoplus}  k(C)^\times \to Z_0(X) \bigoplus
\left({\underset{C \in \sC(X)}\bigoplus} k(C)^\times_\infty \right).
\end{equation}
It is easy to check that $X \mapsto C_0(X)$ is a covariant functor in $X \in \Sm_k$
(cf. \cite[Prop.~3.14]{GKR}).

We next define the tame Brauer group of $X$. 
If $F$ is an hdvf and $i \ge 1$, let
$\Fil_{\bullet} H^i_\et(F)$ denote Kato's filtration of
$H^i_\et(F)$ (cf. \cite[Defn.~2.1]{Kato-Swan-cond}).

\begin{defn}\label{defn:Suslin-Br}
If $\ov{X}$ is a non-singular projective curve over $k$ and
$D = \{x_1, \ldots , x_m\} \subset \ov{X}$ is a reduced divisor with
$X = \ov{X} \setminus D$, we let
$\Br^{\tm}(X)$ be the subgroup of $\Br(X)$ consisting of elements whose image under
the map $\Br(X) \to \Br(k(X)_{x_i}) \cong H^2_{\rm {\etl}}(k(X)_{x_i})$
lies in $\Fil_{1}H^2_{\rm {\etl}}(k(X)_{x_i})$ for every $i$.
If $\dim(X) \ge 2$, we let $\Br^{\tm}(X)$ denote the subgroup of $\Br(X)$ consisting
of elements $\chi \in \Br(X)$ such that for every $C \in \sC(X)$, the Brauer
class $\nu^*(\chi) \in \Br(C_n)$ lies in $\Br^{\tm}(C_n)$, where
$\nu \colon C_n \to C \inj X$ is the map induced by the normalization of $C$.
We shall refer to $\Br^{\tm}(X)$ as the `tame Brauer group' of $X$.
\end{defn}

For $C \in \sC(X)$, we let $U_1K_1(k(C)_{\infty})  =
{\underset{x \in \Delta(C)}\bigoplus} \Ker\left((\wh{\sO_{\ov{C}_n,x}})^\times \to
  k(x)^\times\right)$.

  \begin{prop}\label{prop:Tame-part}
    There is a canonical exact sequence
   \begin{equation}\label{eqn:Tame-part-0}
{\underset{C \in \sC(X)}\bigoplus} U_1 K_1(k(C)_{\infty})
        \xrightarrow{\vartheta_X} C_0(X) \xrightarrow{\kappa_X} H^S_0(X) \to 0.
      \end{equation} 
 The canonical maps
 $\kappa_X \colon {C_0(X)}/n \to {H^S_0(X)}/n$ and
 $\kappa'_X \colon \Br^{\tm}(X)[n] \to \Br(X)[n]$
are isomorphisms if $n \in k^\times$.
\end{prop}
\begin{proof}
  The proof of the exact sequence is completely identical to that of
  \cite[Lem.~3.21]{GKR}. It follows from \cite[Lem.~3.23]{GKR} that the first term
  of ~\eqref{eqn:Tame-part-0} is $n$-divisible, which proves the first isomorphism
  of the second part of the proposition.

It is easy to see that there is an exact sequence
   \begin{equation}\label{eqn:Tame-part-1}
     0 \to \Br^{\tm}(X) \xrightarrow{\kappa'_X} \Br(X) \to
     {\underset{C \in \sC(X)}\prod}
       \frac{\bigoplus_{x \in \Delta(C)} H^2_\et(k(C)_{x})}
            {\bigoplus_{x \in \Delta(C)} \Fil_{1} H^2_\et(k(C)_{x})}.
   \end{equation}
Suppose now that $\chi \in \Br(X)[n]$, and look at the commutative diagram
   \begin{equation}\label{eqn:Tame-part-2}
     \xymatrix@C1pc{
       & H^2_\et(X, {\Z}/n(1)) \ar[d] \ar[dr]^-{\psi} & \\
       \Br^{\tm}(X) \ar[r] & \Br(X) \ar[r] & {\underset{C \in \sC(X)}\prod}
       \frac{\bigoplus_{x \in \Delta(C)} H^2_\et(k(C)_{x})}
            {\bigoplus_{x \in \Delta(C)} \Fil_{1} H^2_\et(k(C)_{x})},}
   \end{equation}
   where the vertical arrow is the Kummer map and $\psi$
   is the product of pull-back maps between {\'e}tale cohomology.

 We can find $\chi' \in H^2_\et(X, {\Z}/n(1))$ which maps to $\chi$.
   Meanwhile, for any $C \in \sC(X)$ and $x \in \Delta(C)$,
   \cite[Cor.~2.5]{Kato-Swan-cond} says that the pull-back 
   $\nu^* \colon H^2_\et(X, {\Z}/n(1)) \to \frac{ H^2_\et(k(C)_{x})}
   {\Fil_{1} H^2_\et(k(C)_{x})}$ is the zero map.
   A diagram chase then implies that $\chi$ dies in this quotient. We conclude that
   $\chi \in  \Br^{\tm}(X)$.
\end{proof}

In the sequel, we shall denote the composite isomorphism
${C_0(X)}/n \xrightarrow{\kappa_X} {H^S_0(X)}/n \xrightarrow{\Psi_X}
H^4_c(X, \Lambda_n(2))$ by $\vartheta_X$.

\subsection{Brauer--Manin pairing for open varieties}\label{sec:BM-defn}
To define the pairing between the class group and Brauer group of $X$,
we let $x \in X_{(0)}$ with the inclusion $\iota_x \colon \Spec(k(x)) \inj X$. We let
$\rho_{X,x} \colon \Br(X) \xrightarrow{\iota^*_x} \Br(k(x))
\xrightarrow{\Cores} \Br(k) \xrightarrow{\inv_{k}} {\Q}/{\Z}$ denote the
composite map, where $\inv_k$ is the classical invariant map which is a bijection,
and $\Cores \colon H^i_\et(k(x)) \to H^i_\et(k)$ is the corestriction map.
Given $\chi \in \Br(X)$, we let $\<x, \chi\> = \rho_{X,x}(\chi)$. Extending it
linearly over $Z_0(X)$, we get a pairing
$\<-, -\> \colon Z_0(X) \times \Br(X) \to {\Q}/{\Z}$.

\begin{prop}\label{prop:BM-pair-def}
  The pairing $\<-, -\> \colon  Z_0(X) \times \Br(X) \to {\Q}/{\Z}$ induces
  a commutative diagram of pairings between abelian groups
  \begin{equation}\label{eqn:BM-pair-def-0}
    \xymatrix@C1.6pc{
C_0(X) \times \Br(X) \ar@<-6ex>[d]_{\kappa_X} \ar[r]^-{\<-, -\>} &
      {\Q}/{\Z} \ar[d]^-{\id} \\
      H^S_0(X) \times \Br^\tm(X)  \ar@<-6ex>[u]_-{\kappa'_X} \ar[r]^-{\<-, -\>} &
         {\Q}/{\Z},}
    \end{equation}
\end{prop}
\begin{proof}
The top pairing is shown in \cite[Prop.~9.4]{KRS}. To show that the bottom pairing
is defined, we let $C \in \sC(X)$ and look at the diagram
  \begin{equation}\label{eqn:BM-pair-def-1}
    \xymatrix@C1.7pc{
      Z_0(C_n) \times \Br^\tm(C_n) \ar@<-6ex>[d]_{\nu_*} \ar[r]^-{\<-, -\>} &
      {\Q}/{\Z} \ar[d]^-{\id} \\
      \sZ_0(X) \times \Br^\tm(X)  \ar@<-6ex>[u]_-{\nu^*} \ar[r]^-{\<-, -\>} &
         {\Q}/{\Z},}
    \end{equation}
  where $\nu \colon C_n \to X$ denotes the canonical map.

If $x \in C_n$ is a closed point with $y = \nu(x)$, 
  we know that $\nu_*([x]) = [k(x): k(y)] [y] \in Z_0(X)$. We also know
  that the map $\Br(k(y)) \to \Br(k(x))$ is multiplication by $[k(x): k(y)]$
  when we identify each of these groups with ${\Q}/{\Z}$. It easily follows from
  these facts and the construction of the pairing $\<-, -\>$ that the above
  diagram is commutative. It suffices therefore to prove our claim when $d =1$.
But this follows from \cite[Prop.~9.4]{KRS} once we observe that
the group $\Br(\ov{C}_n|D)$ used in loc. cit. coincides with $\Br^\tm(C_n)$ if
we let $D = \Delta(C)$ be the reduced divisor. Similarly,
$\CH_0(\ov{C}_n|\Delta(C))$ coincides with $H^S_0(C_n)$. This completes the proof.
\end{proof}

\begin{remk}\label{remk:BM-pair-def-2}
  When $\Char(k) = 0$, the last term of ~\eqref{eqn:Tame-part-1} vanishes
  (cf. \cite[Cor.~2.5]{Kato-Swan-cond}) and the first term of
  ~\eqref{eqn:Tame-part-0} is divisible. It follows that
  the top pairing in ~\eqref{eqn:BM-pair-def-0} descends to a pairing
  $H^S_0(X) \times \Br(X) \to {\Q}/{\Z}$ if $\Char(k) = 0$. We thus recover the
  pairing constructed by Yamazaki \cite{Yamazaki-Nagoya}.
However, using the exact sequence ~\eqref{eqn:Tame-part-0}, it is not hard to see
  that the map $C_0(X) \to \Br(X)^\star$ does not in general factor through
  $H^S_0(X)$ if $\Char(k) >0$. In particular, Yamazaki's pairing
  does not extend to positive characteristic in general without altering the
  underlying groups.
\end{remk}

  By the Brauer--Manin pairing in this paper, we shall mean the top
  pairing of ~\eqref{eqn:BM-pair-def-0}. When $X$ is projective, we have 
  $C_0(X) \cong H^S_0(X) \cong \CH_0(X)$ and $\Br^\tm(X) \cong \Br(X)$, as one easily
  checks. In this case,  both pairings of  ~\eqref{eqn:BM-pair-def-0} therefore
  coincide with the classical Brauer--Manin pairing.

The following is immediate from the definition of the Brauer--Manin pairing.

\begin{lem}\label{lem:BM-pair-functor}
  Given a closed immersion $\iota \colon Y \inj X$ in $\Sm_k$, there is
  a commutative diagram
   \begin{equation}\label{eqn:BM-pair-functor-0}
    \xymatrix@C1.7pc{
C_0(Y) \times \Br(Y) \ar@<-6ex>[d]_{\iota_*} \ar[r]^-{\<-, -\>} &
      {\Q}/{\Z} \ar[d]^-{\id} \\
      C_0(X) \times \Br(X)  \ar@<-6ex>[u]_-{\iota^*} \ar[r]^-{\<-, -\>} &
         {\Q}/{\Z},}
    \end{equation}
\end{lem}

\vskip.2cm

From now onward, we fix a prime $\ell \neq p$. For the Saito--Tate duality pairing
in the following result, we refer to \thmref{thm:SD-1}.

\begin{lem}\label{lem:BM-pair-def-3}
 We have the following.
  \begin{enumerate}
  \item
    The map $C_0(X)^{(\ell)} \to H^S_0(X)^{(\ell)}$ is an isomorphism and the
    Brauer--Manin pairing for $X$ induces a pairing
    \[
    H^S_0(X)^{(\ell)} \times \Br(X)\{\ell\} \to {\Q_\ell}/{\Z_\ell}.
    \]
    \item
    There is a commutative diagram
    \begin{equation}\label{eqn:BM-pair-def-3-0}
    \xymatrix@C1.7pc{
H^S_0(X)^{(\ell)} \times \Br(X)\{\ell\}  \ar@<-6ex>[d]_-{\cong} \ar[r]^-{\<-, -\>} &
{\Q_\ell}/{\Z_\ell} \ar[dd]^-{\id} \\
 H^{2d}_{c}(X, \Z_\ell(d)) \ar@<-6ex>[d]_-{\epsilon^*_X} \hspace*{2cm} & \\
 H^{2d}_{{\rm {\etl}},c}(X, {\Z_\ell}(d)) \times H^2_{\rm {\etl}}(X, {\Q_\ell}/{\Z_\ell}(1))
 \ar@<-6ex>[uu] \ar[r]^-{\<-, -\>} & {\Q_\ell}/{\Z_\ell},}
    \end{equation}
    where the bottom row is the Saito--Tate duality pairing.
    \end{enumerate}
  \end{lem}
  \begin{proof}
The first item is clear from \propref{prop:Tame-part}. 
In ~\eqref{eqn:BM-pair-def-3-0}, the marked isomorphism is a consequence of
\propref{prop:Suslin-hom-0}, $\epsilon^*_X$ is the {\'e}tale realization map
and the vertical arrow going up is the one induced by the Kummer sequence.
To show that  ~\eqref{eqn:BM-pair-def-3-0} is commutative, one can use
~\eqref{eqn:Tame-MCCS-0}, \lemref{lem:BM-pair-functor} and
\cite[Lem.~6.3]{GKR} to reduce it to the case when $X = \Spec(k')$ for a finite
field extension ${k'}/k$. In the latter case, the claim is a routine check and we
omit the details.
\end{proof}

\subsection{Proof of Theorems~\ref{thm:Main-4} and ~\ref{thm:Main-5}}
  \label{sec:Main-4-prf}
  Let $R$ be an excellent hdvr with quotient field $k$
  and residue field $\F$. We assume that $\F$ is either finite or separably closed
  of exponential characteristic $p > 0$.
Let $\sX$ be a connected, regular and proper $R$-scheme of Krull dimension three.
We assume that $\sX$ is faithfully flat over $R$ and its reduced closed fiber $Y$
is an snc divisor on $\sX$. We let $X$ denote the generic fiber of
$\sX$ and assume that it is a smooth $k$-scheme.  Let $j \colon U \inj X$ be the
inclusion of a nonempty open subscheme and $D = X \setminus U$ with the reduced
closed subscheme structure. 
Note that our assumptions imply that $X$ is a projective $k$-scheme and each
irreducible component of $Y$ is a projective $\F$-scheme.

We need a few lemmas before we prove the main results.

\begin{lem}\label{lem:BM-fin}
  Assume that $\F$ is finite. Then the Brauer--Manin pairing for $U$ induces a
  pairing between projective and injective systems of finite groups
  \[
  \{{C_0(U)}/{\ell^n}\}_{n\ge 1} \times \{\Br(U)[\ell^n]\}_{n \ge 1} \to
  \{{\Z}/{\ell^n}\}_{n \ge 1}.
    \]
\end{lem}
\begin{proof}
  We only need to show that ${C_0(U)}/{\ell^n}$ and $\Br(U)[\ell^n]$ are finite.
  Using the surjection $H^2_\et(U, {\Z}/{\ell^n}(1)) \surj \Br(U)[\ell^n]$,
  the finiteness of $\Br(U)[\ell^n]$ follows directly from
  \lemref{lem:Etale-fin} and \cite[Thm.~6.8]{GKR}.

To prove the same for ${C_0(U)}/{\ell^n}$, we can use
Propositions~\ref{prop:Suslin-hom-0} and ~\ref{prop:Tame-part} to reduce to
showing that $H^4_c(U, \Lambda_{\ell^n}(2))$ is finite. It suffices now to show that 
the first and the last groups in the exact sequence (cf. ~\eqref{eqn:Mot-loc})
  \[
  H^3(D, \Lambda_{\ell^n}(2)) \to H^4_c(U, \Lambda_{\ell^n}(2)) \to
  H^4(X, \Lambda_{\ell^n}(2))
  \]
  are finite. The last group is finite because it coincides with
  ${\CH_0(X)}/{\ell^n}$ whose finiteness was shown in the last part of the proof of
  \lemref{lem:Main-Diag-1},
  and the first group is finite by
  \cite[Cor.~7.3]{GKR} and \lemref{lem:Etale-fin}.
 \end{proof}

We endow $C_0(U)^{(\ell)} \cong H^S_0(U)^{(\ell)}$ with the profinite topology and
$\Br(U)\{\ell\}$ with the discrete topology using \lemref{lem:BM-fin}.
We then obtain a continuous pairing of topological abelian groups
\begin{equation}\label{eqn:BM-fin-0}
  C_0(U)^{(\ell)} \times \Br(U)\{\ell\} \to {{\Q}_{\ell}}/{\Z_\ell}.
\end{equation}

Another finiteness result we need is the following. This was proven in
\cite[Prop.~4.2]{Gazaki-Rathore} when $k$ is a $p$-adic field.

\begin{lem}\label{lem:Torsion-Sus-hom}
  $H^S_0(U)[n]$ is a finite group if $n \in k^\times$.
\end{lem}
\begin{proof}
  If $\Char(k) = 0$, the lemma is equivalent by \propref{prop:Suslin-hom}(4) to
  showing that $H^4_c(U, \Z(2))[n]$ is finite. Using the surjection
  $H^3_c(U, {\Z}/n(2)) \surj H^4_c(U, \Z(2))[n]$, it suffices to show that
  $H^3_c(U, {\Z}/n(2))$ is finite. Using \cite[Cor.~7.3]{GKR}, this reduces to
  showing that $H^3_{c. \et}(U, {\Z}/n(2))$ is finite. But this follows from
  \lemref{lem:Etale-fin}.
  
Suppose now that $\Char(k) = p > 0$. 
We let $U^\hs$ be the base change of $U$ to $k^\hs$. \propref{prop:Suslin-hom}(2)
then says that the kernel of the pull-back map $H^S_0(U) \to H^S_0(U^\hs)$ is a
  $p$-primary torsion group. In particular,
  $H^S_0(U)[n] \inj H^S_0(U^\hs)[n]$.
On the other hand, we have a canonical isomorphism
$H^S_0(U^\hs)[\tfrac{1}{p}] \xrightarrow{\cong} H_0(U^\hs, \Z[\tfrac{1}{p}](0))$ by
\cite[Prop.~14.18]{MVW}, and
a canonical duality isomorphism $H_0(U^\hs, \Z[\tfrac{1}{p}](0))
\xrightarrow{\cong} H^4_c(U^\hs, \Z[\tfrac{1}{p}](2))$ by \cite[Thm.~5.3.18]{Kelly}.

 We thus get
  \[
  H^S_0(U^\hs)[n] \xrightarrow{\cong}  (H^S_0(U^\hs)[\tfrac{1}{p}])[n]
  \xrightarrow{\cong} 
 H^4_c(U^\hs, \Z[\tfrac{1}{p}](2))[n].
 \]
 It suffices therefore to show that $H^4_c(U^\hs, \Z[\tfrac{1}{p}](2))[n]$ is
 finite. But this
 follows from exactly the same argument as in the characteristic zero case above.
\end{proof}

For an abelian group $A$, we let $A_{\ell\divf} = \Ker(A \to A^{(\ell)}) =
{\underset{n \ge 1} \bigcap} \ell^nA$. It is easy to check that
$A_{\ell\divf}$ is the maximal $\ell$-divisible subgroup of $A$ if and only if it is
$\ell$-divisible. Moreover, $A_{\ell\divf}$ is $\ell$-divisible if $A[\ell^n]$ is
finite for every $n \ge 1$.

\begin{lem}\label{lem:Divisible-class-grp}
  $C_0(U)_{\ell\divf}$ and $H^S_0(U)_{\ell\divf}$ are the maximal $\ell$-divisible
  subgroups of $C_0(U)$ and $H^S_0(U)$, respectively.
\end{lem}
\begin{proof}
  The result for $H^S_0(U)_{\ell\divf}$ follows directly from
  \lemref{lem:Torsion-Sus-hom}
  and the above fact. To prove the same for $C_0(U)$, we only have to show that
  $C_0(U)_{\ell\divf}$ is $\ell$-divisible.
To that end, we use the commutative diagram
  \begin{equation}\label{eqn:Divisible-class-grp-0}
    \xymatrix@C1pc{
      0 \ar[r] & C_0(U)_{\ell\divf} \ar[r] \ar[d] & C_0(U) \ar[r]
      \ar@{->>}[d]^-{\kappa_U} &
      C_0(U)^{(\ell)} \ar[d] \\
      0 \ar[r] & H^S_0(U)_{\ell\divf} \ar[r]  & H^S_0(U) \ar[r] &  H^S_0(U)^{(\ell)}.}
  \end{equation}

The right vertical arrow is an isomorphism and the kernel of $\kappa_U$
  is $\ell$-divisible by \propref{prop:Tame-part}.
   This implies that the left vertical arrow is surjective
  whose kernel is $\ell$-divisible. We conclude that $C_0(U)_{\ell\divf}$ is
  $\ell$-divisible.
\end{proof}

We now complete the proof of Theorem~\ref{thm:Main-4}.
The first part is already shown above. We restate the more non-trivial
second part to make the underlying assumptions explicit.

\begin{thm}\label{thm:BM-main-pf}
 Assume that the $(\star\star_\ell)$-condition holds. Then the cycle class map
  \[
  \cyc_U \colon H^S_0(U)^{(\ell)} \to H^4_{{\rm {\etl}},c}(U, \Z_\ell(2))
  \]
  is injective. Equivalently, the Brauer--Manin pairing for $U$ induces an
  injective map
  $C_0(U)^{(\ell)} \to \Hom(\Br(U)\{\ell\}, {\Q}/{\Z})$.
\end{thm}
  \begin{proof}
We first show the equivalence of the two statements. Suppose that $\cyc_U$ is
    injective.
   It follows from \lemref{lem:BM-pair-def-3} that one has a commutative diagram
  \begin{equation}\label{eqn:Main-0-1-0}
    \xymatrix@C1pc{
      H^S_0(U)^{(\ell)} \ar[d]_-{\cong} \ar[rr] \ar[dr]^-{\cyc_U} & &
      \Hom(\Br(U)\{\ell\}, {\Q}/{\Z}) \ar@{^{(}->}[d] \\
      H^4_{c}(U, {\Z}_\ell(2)) \ar[r]^-{\epsilon^*_U}  &  H^4_{\et,c}(U, {\Z}_\ell(2))
      \ar[r]^-{\cong} & \Hom(H^2_\et(U, {\Q_\ell}/{\Z_\ell}(1)), {\Q}/{\Z}),}
    \end{equation}
  where the top row is induced by the Brauer--Manin pairing for $U$.

The right horizontal arrow on the lower level is bijective by the Saito--Tate
  duality \cite{Saito-Duality} (cf. \thmref{thm:SD-1}).
  A diagram chase shows that the top horizontal arrow is injective.
  We now apply \propref{prop:Tame-part} and
  \lemref{lem:Divisible-class-grp} to conclude that
  the kernel of the map $C_0(U) \to \Br(U)\{\ell\}^\star$ coincides with
  $C_0(U)_{\ell\divf}$. The converse is clear from \lemref{lem:BM-pair-def-3}
  and the diagram ~\eqref{eqn:Main-0-1-0}.

 To prove the theorem, it remains to show that $\cyc_U$
  (equivalently, $\epsilon^*_U$) is injective. To that end, we proceed as follows.

By ~\eqref{eqn:MC-03}, we have a commutative diagram
  \begin{equation}\label{eqn:Main-0-1-1}
    \xymatrix@C1pc{
      H^3(X, {\Z}/{\ell^n}(2)) \ar[r] \ar[d] & H^3(D, {\Z}/{\ell^n}(2)) \ar[r]
      \ar[d] & H^4_c(U, {\Z}/{\ell^n}(2)) \ar[r] \ar[d] &
      H^4(X, {\Z}/{\ell^n}(2)) \ar[d] \\
      H^3_\et(X, {\Z}/{\ell^n}(2)) \ar[r] & H^3_\et(D, {\Z}/{\ell^n}(2)) \ar[r] &
      H^4_{\et,c}(U, {\Z}/{\ell^n}(2)) \ar[r] & H^4_\et(X, {\Z}/{\ell^n}(2))}
    \end{equation}
  of exact sequences for every $n \ge 1$
  whose vertical arrows are the cycle class maps.

All groups on the lower level of ~\eqref{eqn:Main-0-1-1}
are finite by \lemref{lem:Etale-fin}.
The first and the second vertical arrows from left are injective by
\cite[Cor.~7.3]{GKR}. It follows that the first two groups on the
upper level are finite. We showed in the end of the proof of 
\lemref{lem:Main-Diag-1} that $H^4(X, {\Z}/{\ell^n}(2)) \cong {\CH_0(X)}/{\ell^n}$
is finite. It follows that $H^4_c(U, {\Z}/{\ell^n}(2))$ is finite.
  We have thus shown that all groups in ~\eqref{eqn:Main-0-1-1} are finite.
  In particular, its rows remain exact after taking inverse limit as $n \to \infty$.

We thus get a commutative diagram of exact sequences
  \begin{equation}\label{eqn:Main-0-1-2}
    \xymatrix@C1pc{
      H^3(X, {\Z}_{\ell}(2)) \ar[r] \ar[d]_-{\epsilon^*_X} & H^3(D, {\Z}_{\ell}(2))
      \ar[r] \ar[d]^-{\epsilon^*_D} &
      H^4_c(U, {\Z}_{\ell}(2)) \ar[r] \ar[d]^-{\epsilon^*_U} &
      H^4(X, {\Z}_{\ell}(2)) \ar[d]^-{\epsilon^*_X} \\
      H^3_\et(X, {\Z}_{\ell}(2)) \ar[r] & H^3_\et(D, {\Z}_{\ell}(2)) \ar[r] &
      H^4_{\et,c}(U, {\Z}_{\ell}(2)) \ar[r] & H^4_\et(X, {\Z}_{\ell}(2)).}
  \end{equation}

The first vertical arrow in this diagram is bijective by \thmref{thm:Main-0-0}, the
  second vertical arrow is injective by \cite[Cor.~7.3]{GKR} and the fourth vertical
  arrow is injective by \cite[Thm.~3.1]{Esnault-Wittenberg}. A diagram chase now
  shows that ${\epsilon^*_U}$ is injective. This completes the proof.
\end{proof}

  \begin{cor}\label{cor:BM-main-pf-0}
    Assume that the $(\star\star_\ell)$-condition holds. Then the Brauer--Manin
    pairing for $U$ induces an injective map
  $C_0(U) \to \Hom(\Br(U)\{\ell\}, {\Q}/{\Z})$ modulo the maximal $\ell$-divisible
  subgroup of $C_0(U)$.
  \end{cor}
  \begin{proof}
    This is immediate from \thmref{thm:BM-main-pf} because 
    $C_0(U) \to \Hom(\Br(U)\{\ell\}, {\Q}/{\Z})$ factors through
    $C_0(U) \to C_0(U)^{(\ell)} \to \Hom(\Br(U)\{\ell\}, {\Q}/{\Z})$,
    and the kernel of the first map is the maximal $\ell$-divisible
  subgroup of $C_0(U)$ by \lemref{lem:Divisible-class-grp}.
\end{proof}

\vskip.2cm

The following result proves Theorem~\ref{thm:Main-5}.

\begin{thm}\label{thm:BM-main-p5-f}
  Assume that $\F$ is separably closed and $H^2_{\rm {\etl}}(\ov{X}, {\Q}_\ell(1))$ is
  algebraic. Then the cycle class map
  \[
  \cyc_U \colon H^S_0(U)^{(\ell)} \to H^4_{{\rm {\etl}},c}(U, \Z_\ell(2))
  \]
  is injective.
\end{thm}
\begin{proof}
  The proof is completely identical to that of \thmref{thm:BM-main-pf} except that
  one uses \cite[Thm.~4.1]{Esnault-Wittenberg} in place of
  \cite[Thm.~3.1]{Esnault-Wittenberg}.
\end{proof}

\section{Saito--Tate duality for regular schemes}\label{sec:ST-duality}
In order to prove \thmref{thm:Main-10}, we need some results concerning the
functoriality of the {\'e}tale realization map from the motivic cohomology of
schemes over a local field to their {\'e}tale cohomology. We shall need these
results for schemes which may not necessarily be smooth or projective over the
ground field.
We shall derive these results by proving an extension of the Saito--Tate duality
\cite{Saito-Duality} for {\'e}tale cohomology from smooth to regular schemes over a
local field and showing its compatibility with the localization sequence for
{\'e}tale cohomology. We begin by setting up the notation for this section.

We let $k$ be an hdvf with finite residue field.
We let $p > 0$ denote the exponential characteristic of $k$.
We let $\Reg_k$ denote the category of equidimensional and 
regular quasi-projective $k$-schemes.
We fix an integer $n$ which is invertible in $k$ and let $\Lambda = {\Z}/n$
(this notation for ${\Z}/n$ will be used only in this section and we shall revert
back to our earlier notation $\Lambda_n$ from the next section onward).
For a Noetherian $k$-scheme $X$, we let
$D_\et(X, \Lambda) \subset \dr_\et(X, \Lambda)$ denote the bounded derived category
of constructible sheaves of $\Lambda$-modules on $X_\et$.  For
$A, B \in D_\et(X, \Lambda)$, we shall let $\un{\Hom}_X(A,B) \in D_\et(X, \Lambda)$
denote the derived internal hom object from $A$ to $B$. We shall write
$\Hom_{D_\et(X, \Lambda)}(A,B)$ as $[A,B]_X$.

Unless mentioned otherwise, all functors between 
derived categories in this section will be assumed to be derived. The tensor product
of two objects in $D_\et(X, \Lambda)$ will be the derived tensor product. 
We shall let $\Lambda_X$ denote the constant sheaf $\Lambda$ on $X_\et$.
We shall denote the object $\Lambda_X(i)[j] = \mu^{\otimes i}_{n,X}[j] \in
D_\et(X, \Lambda)$ by $\Lambda_X(i,j)$. If $X$ has pure dimension $d$, we shall let
$L_X = \Lambda_X(d,2d)$. If $\iota \colon Y \inj X$ is an immersion and
$K \in   D_\et(X, \Lambda)$, we shall often denote $\iota^*(K)$ by $K_Y$.
We shall denote $K \otimes \Lambda_X(i,j)$ by $K(i,j)$. We let $K(i) = K(i,0)$
and $K[j] = K(0,j)$.

Recall that the presheaf of monoidal triangulated categories
$X \mapsto D_\et(X, \Lambda)$ on $\Sch_k$ is endowed with Grothendieck's six
functors formalism (cf. \cite{CD-Comp}). We shall use the following
elementary result throughout this section without explicitly referring to it.

\begin{lem}\label{lem:Twist}
  Let $f \colon X' \to X$ be a morphism of $k$-schemes.
  Let $K \in  D_{\rm {\etl}}(X, \Lambda)$
  and $K' \in  D_{\rm {\etl}}(X', \Lambda)$. Then there are canonical isomorphisms
  $f^{!}(K(i,j)) \cong (f^!K)(i,j)$ and $f_!(K'(i,j)) \cong (f_!K')(i,j)$.
\end{lem}
\begin{proof}
This is a formal consequence of the six functor formalism (cf.
  \cite[Defn.~A.1.10(5)]{CD-Comp}). The second isomorphism is a special
  case of the more general isomorphism $f_!K' \otimes L \xrightarrow{\cong}
  f_!(K' \otimes f^*L)$ for any $L \in D_\et(X, \Lambda)$.
 For the
  first isomorphism, one notes that
  \[
  \begin{array}{lllll}
  f^!(K(i,j)) & = & f^!(\Lambda_X(i,j) \otimes K) & \xrightarrow{\cong} &
  f^!\left(\un{\Hom}_X(\Lambda_X(-i,-j), K)\right) \\
  & \xrightarrow{\cong} & \un{\Hom}_{X'}\left(\Lambda_{X'}(-i,-j), f^!K\right)
  & \xrightarrow{\cong}& \Lambda_{X'}(i,j) \otimes f^!K =  (f^!K)(i,j).
  \end{array}
  \]
  This gives the desired isomorphism.
 \end{proof}

\subsection{Trace map for regular schemes and its functoriality}\label{sec:Trace*}
Recall from \cite[Thm.~A]{Khan-Sigma} (this was originally proven in
\cite[Exp.~XVIII]{SGA4} for smooth morphisms) that for a morphism
$f \colon X' \to X$ in $\Reg_k$ of relative dimension $s$, there is a
natural trace map $\tr_f \colon f_!\Lambda_{X'}(s,2s) \to \Lambda_X$
whose adjoint $\tr^\star_f \colon \Lambda_{X'}(s,2s) \to f^!\Lambda_X$
is an isomorphism. Furthermore, $\tr^\star_f$ coincides with Gabber's
purity isomorphism (cf. \cite{Fujiwara}) if $f$ is a closed immersion.
For $K \in D_\et(X, \Lambda)$, we let $\theta_X(K)$
   denote the composite map
  \begin{equation}\label{eqn:SD-1-2-0}
  f_{*} \un{\Hom}_X(K, L_X) \to
  \un{\Hom}_k(f_{!}K, f_{!} L_{X}) 
  \xrightarrow{\circ \tr_f} \un{\Hom}_k(f_{!}K, \Lambda_{k}),
\end{equation}
  $f \colon X \to \Spec(k)$ is the structure map and the first map is obtained by
  applying the $f_{!}$ functor.

We now fix a closed immersion $\iota \colon Y \inj X$ of pure codimension $e$
in $\Reg_k$ and let $j \colon U \inj X$ be the inclusion of the complement
of $Y$. We let $d = \dim(X)$ and $d' = \dim(Y) = d-e$.
For $i,j \in \Z$, we shall write $i' = 2d+2-i$ and $j' = d+1-j$.
We let $\pi \colon X \to \Spec(k)$ denote the structure map. We let
$f = \pi \circ j$ and $g = \pi \circ \iota$. 

The trace map has the naturality property in that the diagrams
\begin{equation}\label{eqn:SD-1-2}
     \xymatrix@C1pc{
       g_{!}\Lambda_{Y}(d',2d')  \ar[r]^-{\cong} \ar[d]_-{\tr_g} &
       \pi_{!} \iota_{!}\left(\Lambda_{Y}(d',2d')\right) \ar[r]^-{\cong}   &
       \pi_{!}\left((\iota_{!}\Lambda_{Y}(-e, -2e))(d,2d)\right)
       \ar[d]^-{\pi_{!}(\tr_{\iota})} \\
       \Lambda_{k} & & \pi_{!}\Lambda_{X}(d,2d) \ar[ll]_-{\tr_\pi};}
\end{equation}

\begin{equation}\label{eqn:SD-1-2-1}
     \xymatrix@C1.6pc{
       \Lambda_{Y}(d',2d') \ar[r]^-{\tr^\star_\iota} \ar[d]_-{\tr^\star_g}
       & \iota^!\Lambda_X(d,2d)
       \ar[d]^-{\iota^!(\tr^\star_\pi)} \\
       g^!\Lambda_k & \iota^!\pi^!\Lambda_k \ar[l]_-{\cong}}
     \end{equation}
     are commutative. 

\subsection{Poincar{\'e} duality for regular schemes}\label{sec:PD-reg}
We fix a pair of integers $i, j \in \Z$ and let $K = \Lambda_X(i,j)
\in D_\et(X, \Lambda)$. The exact triangle
$j_!K_U \xrightarrow{j_*} K \xrightarrow{\iota^*} \iota_!K_Y \xrightarrow{+1}$
yields a diagram
\begin{equation}\label{eqn:SD-1-3}
  \xymatrix@C.5pc{
\pi_* \un{\Hom}_X(j_!K_U[1], L_X) \ar[d]_-{\theta_X(j_!K_U[1])} \ar[r] &
\pi_* \un{\Hom}_X(\iota_!K_Y, L_X) \ar[r] \ar[d]_-{\theta_X(\iota_!K_Y)} &
\pi_* \un{\Hom}_X(K, L_X) \ar[r] \ar[d]^-{\theta_X(K)} &
\pi_* \un{\Hom}_X(j_!K_U, L_X) \ar[d]^-{\theta_X(j_!K_U)} \\
\un{\Hom}_k(\pi_!j_!K_U[1], \Lambda_{k}) \ar[r] &
\un{\Hom}_k(\pi_!\iota_!K_Y, \Lambda_{k}) \ar[r] &
\un{\Hom}_k(\pi_!K, \Lambda_{k}) \ar[r] &
\un{\Hom}_k(\pi_!j_!K_U, \Lambda_{k})}
\end{equation}
in which any three consecutive columns form a commutative diagram of
exact triangles in $D_\et(k, \Lambda)$.
We let $\partial \colon \iota_!K_Y\to j_!K[1]_U$ denote the
boundary map of the above exact triangle.

\begin{lem}\label{lem:SD-2}
  There is a commutative diagram
\begin{equation}\label{eqn:SD-2-0}
  \xymatrix@C1pc{  
    g_* \un{\Hom}_Y(K_Y, L_Y) \ar[r]^-{\tr^\star_\iota}  \ar[d]_-{\theta_Y(K_Y)} &
       \pi_* \un{\Hom}_X(\iota_!K_Y, L_X) \ar[d]^-{\theta_X(\iota_!K_Y)} \\
\un{\Hom}_k(g_!K_Y, \Lambda_{k}) &
\un{\Hom}_k(\pi_!\iota_!K_Y, \Lambda_{k}),  \ar[l]}
\end{equation}
whose all arrows are isomorphisms.
\end{lem}
\begin{proof}
The bottom horizontal arrow is induced by the identity $g_! = \pi_! \iota_!$ and
  hence is an isomorphism. The top horizontal arrow is the composition of
  isomorphisms
  \[
  \begin{array}{lll}
g_* \un{\Hom}_Y(K_Y, L_{Y}) & \cong & \pi_* \iota_*  \un{\Hom}_Y(K_Y, L_Y) \\
& \xrightarrow{\tr^\star_\iota} &
\pi_* \iota_*  \un{\Hom}_Y(K_Y, (\iota^{!}\Lambda_{X}(e,2e))(d',2d'))
    \\
    & \cong & \pi_* \iota_* \un{\Hom}_Y(K_Y, \iota^{!}L_X) \\    
    & \cong & \pi_* \un{\Hom}_X(\iota_!K_Y, L_X).
    \end{array}
  \]

We now look at the diagram
  \begin{equation}\label{eqn:SD-2-1}
  \xymatrix@C1.6pc{  
g_* \un{\Hom}_Y(K_Y, L_Y) \ar[r]^-{\tr^\star_\iota \circ (-)} \ar[ddd]_-{g_!(-)} &
g_*  \un{\Hom}_Y\left(K_Y, (\iota^{!}\Lambda_{X}(e,2e))(d',2d')\right)
\ar[d]^-{\cong} \\
& \pi_* \iota_*  \un{\Hom}_Y(K_Y, \iota^{!}L_X) \ar[d]^-{\cong} \\
    & \pi_* \un{\Hom}_X(\iota_!K_Y, L_X) \ar[d]^-{\pi_!(-)} \\
\un{\Hom}_k(g_!K_Y, g_! L_Y) \ar[d]_-{\tr_g \circ (-)}
          \ar[r]^-{\pi_!(\tr_\iota)\circ (-)} &
\un{\Hom}_k(g_!K_Y, \pi_! L_X) \ar[dl]^-{\tr_\pi \circ (-)} \\
\un{\Hom}_k(g_!K_Y, \Lambda_{k}). &}
  \end{equation}

  The top square commutes because each of
  $\left(\pi_!(\tr_\iota) \circ (-)\right) \circ \left(g_!(-)\right)$ and
  $\left(\pi_!(-)\right) \circ \left(\tr^\star_\iota\circ (-)\right)$
  is the one induced by applying $\pi_!$ to the composite map
  $\iota_!K_Y \to \iota_!L_Y \to L_X$, where the last map is induced by the counit of
  adjunction $\iota_!\iota^! \to \id$ in $D_\et(X, \Lambda)$. The bottom triangle
  commutes by ~\eqref{eqn:SD-1-2}. It follows that ~\eqref{eqn:SD-2-0} is
  commutative.

To prove that the vertical arrows of  ~\eqref{eqn:SD-2-0} are isomorphisms, it
  suffices to show that  $\theta_X(K)$ is an isomorphism. To that end,
  we use the quasi-projectivity of $X$ to get a closed embedding
  $\phi \colon X \inj X'$, where $X'$ is a smooth and quasi-projective $k$-scheme.
  Letting $K' = \Lambda_{X'}(j)$, the previous paragraph then says that
  the diagram
  \begin{equation}\label{eqn:SD-2-2}
  \xymatrix@C1pc{  
    \pi_* \un{\Hom}_X(K, L_X) \ar[r]^-{\tr^\star_\phi}  \ar[d]_-{\theta_X(K)} &
       \pi'_* \un{\Hom}_{X'}(\phi_!K, L_{X'}) \ar[d]^-{\theta_{X'}(\phi_!K)} \\
\un{\Hom}_k(\pi_!K, \Lambda_{k}) &
\un{\Hom}_k(\pi'_!\phi_!K, \Lambda_{k}),  \ar[l]}
\end{equation}
is commutative, where $\pi' \colon X' \to \Spec(k)$ is the structure map. 
We are now done because $\phi_!K \in D_\et(X', \Lambda)$ whence
$\theta_{X'}(\phi_!K)$ is an isomorphism by the Poincar{\'e} duality theorem for
smooth $k$-schemes (cf. \cite[Exp.~XVIII, Thm.~3.2.5]{SGA4}).
This completes the proof.
\end{proof}

\begin{lem}\label{lem:SD-3}
  There are commutative diagrams
\begin{equation}\label{eqn:SD-3-0}
  \xymatrix@C.4pc{  
f_* \un{\Hom}_U(K_U, L_U) \ar[r] \ar[d]_-{\theta_U(K_U)} &
\pi_* \un{\Hom}_X(j_!K_U, L_X) \ar[d]^-{\theta_X(j_!K_U)} & 
f_* \un{\Hom}_U(K_U[1], L_U) \ar[r] \ar[d]_-{\theta_U(K_U[1])} &
\pi_* \un{\Hom}_X(j_!K_U[1], L_X) \ar[d]^-{\theta_X(j_!K_U[1])} \\
\un{\Hom}_k(f_!K_U, \Lambda_{k}) & \un{\Hom}_k(\pi_!j_!K_U, \Lambda_{k})  \ar[l] & 
\un{\Hom}_k(f_!K_U[1], \Lambda_{k}) & \un{\Hom}_k(\pi_!j_!K_U[1], \Lambda_{k})
\ar[l],}
\end{equation}
whose all arrows are isomorphisms.
\end{lem}
\begin{proof}
The asserted claim for the diagram on the right follows from the same for the
  one on the left by replacing $K$ by $K[1]$ (cf. \lemref{lem:Twist}) everywhere
  in the above discussion. The claim for the left diagram is identical to that of
  \lemref{lem:SD-2} once we replace $\iota_*$ by $j_*$ and use that $j^! \cong j^*$.
\end{proof}

Combining ~\eqref{eqn:SD-1-3} with Lemmas~\ref{lem:SD-2} and ~\ref{lem:SD-3}, we
get the following Poincar{\'e} duality for regular schemes.
\begin{prop}\label{prop:SD-4}
 There is a commutative diagram 
  \begin{equation}\label{eqn:SD-4-0}
    \xymatrix@C1pc{
f_* \un{\Hom}_U(K[1]_U, L_U) \ar[r] \ar[d]_-{\theta_U(K[1]_U)} \ar[r] &
g_* \un{\Hom}_Y(K_Y, L_Y) \ar[r] \ar[d]_-{\theta_Y(K_Y)} &
\pi_* \un{\Hom}_X(K, L_X) \ar[r] \ar[d]^-{\theta_X(K)} &
f_* \un{\Hom}_U(K_U, L_U) \ar[d]^-{\theta_U(K_U)} \\
\un{\Hom}_k(f_!K[1]_U, \Lambda_{k}) \ar[r] &
\un{\Hom}_k(g_!K_Y, \Lambda_{k}) \ar[r] &
\un{\Hom}_k(\pi_!K, \Lambda_{k}) \ar[r] &
\un{\Hom}_k(f_!K_U, \Lambda_{k})}
\end{equation}
in which any three consecutive columns form a commutative diagram of
exact triangles in $D_{\rm {\etl}}(k, \Lambda)$.
\end{prop}

We shall use the following local duality theorem of Tate (cf. \cite{Tate}).

\begin{thm}\label{thm:Tate-Duality-0}
  For any $K \in D_{\rm {\etl}}(k, \Lambda)$, we have the following.
  \begin{enumerate}
  \item
    $H^i_{\rm {\etl}}(k, K)$ is finite for all $i \in \Z$. If $H^i(K) = 0$ for all
    $i > m$, then $\H^i_{\rm {\etl}}(k, K) =0$ for all $i > m+2$.
  \item
    The trace map for $\Spec(k)$ induces an isomorphism $H^2_{\rm {\etl}}(k, \Lambda_k(1))
    \xrightarrow{\cong} \Lambda$.
  \item
    The composition of the Ext groups
    \[
    H^i_{\rm {\etl}}(k, K) \times H^{2-i}_{\rm {\etl}}(k, \un{\Hom}_k(K, \Lambda_k(1)))
    \to H^2_{\rm {\etl}}(k, \Lambda(1)) \cong \Lambda
    \]
    is a perfect pairing of finite groups which is natural in $K$.
  \end{enumerate}
\end{thm}

\subsection{The duality theorem}\label{sec:Duality-main}
We keep the notation of \S~\ref{sec:Trace*}. We begin with the
following.

\begin{lem}\label{lem:SD-5}
  There is a canonical isomorphism $\Tr_X \colon H^{2d+2}_{{\rm {\etl}},c}(X, \Lambda(d+1))
  \xrightarrow{\cong} H^2_{{\rm {\etl}}}(k, \Lambda(1)) \cong \Lambda$ and there is a perfect
  pairing of finite groups
  \begin{equation}\label{eqn:SD-5-0}
    H^i_{{\rm {\etl}},c}(X, \Lambda(j)) \times H^{i'}_{\rm {\etl}}(X, \Lambda(j')) \to
    H^2_{{\rm {\etl}}}(k, \Lambda(1)).
  \end{equation}
\end{lem}
\begin{proof}
We let $K = \pi_!\Lambda_X(j)$ and apply \thmref{thm:Tate-Duality-0} to get
  a perfect pairing of finite groups
  \[
    [\Lambda_k, \pi_!\Lambda_X(j,i)]_k \times
    [\Lambda_k,  \un{\Hom}_k(\pi_!\Lambda_X(j,i), \Lambda_k(1,2))]_k \to
  [\Lambda_k, \pi_!\Lambda_X(d+1,2d+2)]_k
  \xrightarrow{\Tr_X} \Lambda.
  \]

By definition, we have
  $H^i_\et(k, \pi_!\Lambda_X(j)) \cong H^i_{\et,c}(X, \Lambda_X(j))$.
  Meanwhile, we have
  \begin{equation}\label{eqn:SD-5-1}
    \begin{array}{lll}
      [\Lambda_k,  \un{\Hom}_k(\pi_!\Lambda_X(j,i), \Lambda_k(1,2))]_k & \cong &
      [\Lambda_k, \pi_* \un{\Hom}_X(\Lambda_X(j,i), \pi^!\Lambda_k(1,2))]_k \\
      & \cong &
      [\Lambda_k, \pi_* \un{\Hom}_X(\Lambda_X(j,i), \Lambda_X(d+1,2d+2))]_k \\
      & \cong & [\Lambda_X, \un{\Hom}_X(\Lambda_X(j,i), \Lambda_X(d+1,2d+2))]_X \\
      & \cong & [\Lambda_X, \Lambda_X(d+1-j,2d+2-i)]_X \\
       & \cong &  H^{i'}_\et(X, \Lambda(j')),
  \end{array}
  \end{equation}
  where the first isomorphism is by the adjointness of $(\pi_!, \pi^!)$ and the
  second isomorphism is induced by $(\tr^\star_\pi)^{-1}$.
This yields the perfect pairing of finite abelian groups as displayed
  in ~\eqref{eqn:SD-5-0}.

Letting $K = \pi_!\Lambda_X(d+1)$ and $i = 2d+2$ in this pairing, we get
 \[
  H^{2d+2}_{\et,c}(X, \Lambda(d+1)) \xrightarrow{\cong}
  \Hom(H^0_\et(X, \Lambda(0)), H^2_\et(k, \Lambda(1))) \xrightarrow{\cong}
  H^2_\et(k, \Lambda(1)) \xrightarrow{\Tr_k} \Lambda.
  \]
 This completes the proof. 
\end{proof}

We now state and prove the main result of  \S~\ref{sec:ST-duality}.
This extends the Saito--Tate duality for smooth $k$-schemes to regular $k$-schemes,
and moreover, proves its compatibility with the localization sequences.
We let $H^i_{\et,c}(S, \Lambda(j))^\star =
\Hom(H^i_{\et,c}(S, \Lambda(j)), H^2_\et(k, \Lambda_k(1)))$ for any
$S \in \{X,Y,U\}$. Note that this
notation does not contradict our original definition of
$A^\star$ for an abelian group $A$.

\begin{thm}\label{thm:SD-1}
 For every $i, j \ge 0$, we have the following.
 \begin{enumerate}
 \item
   There exists a canonical isomorphism
   \[
   \Tr_X \colon H^{2d+2}_{{\rm {\etl}},c}(X, \Lambda(d+1)) \xrightarrow{\cong}
   H^2_{\rm {\etl}}(k, \Lambda(1)).
   \]
 \item
  There exists a perfect pairing of finite groups
  \[
  H^i_{{\rm {\etl}},c}(X, \Lambda(j)) \times H^{i'}_{{\rm {\etl}}}(X, \Lambda(j'))
  \to H^2_{\rm {\etl}}(k, \Lambda(1)).
  \]
\item
  The exact triangles $\iota_* \Lambda_Y(j-e, i-2e) \xrightarrow{\iota_*}
  \Lambda_X(j) \xrightarrow{j^*} j_* \Lambda_U(j) \xrightarrow{\partial'}$ and
  $j_!\Lambda_U(j') \xrightarrow{j_*} \Lambda_X(j') \xrightarrow{\iota^*}
  \iota_* \Lambda_Y(j')  \xrightarrow{\partial}$ induce
a commutative diagram of exact sequences
  \begin{equation}\label{eqn:SD-1-0}
  \xymatrix@C.6pc{
  H^i_{\rm {\etl}}(X, \Lambda(j)) \ar[r]^-{j^*} \ar[d] & H^i_{\rm {\etl}}(U, \Lambda(j))
  \ar[r]^-{\partial'} \ar[d] & H^{i+1-2e}_{{\rm {\etl}}}(Y, \Lambda(j-e)) \ar[r]^-{\iota_*}
  \ar[d] & H^{i+1}_{\rm {\etl}}(X, \Lambda(j)) \ar[d] \\
  H^{i'}_{{\rm {\etl}}, c}(X, \Lambda(j'))^\star \ar[r]^-{(j_*)^\star} &
  H^{i'}_{{\rm {\etl}},c}(U, \Lambda(j'))^\star \ar[r]^-{\partial^\star} &
  H^{i'-1}_{{\rm {\etl}},c}(Y, \Lambda(j'))^\star \ar[r]^-{(\iota^*)^\star} &
      H^{i'-1}_{{\rm {\etl}},c}(X, \Lambda(j'))^\star,}
  \end{equation}
  whose vertical arrows are isomorphisms.
 \end{enumerate}
 \end{thm}
\begin{proof}
We only have to show item (3) as the other items follow directly from
  \lemref{lem:SD-5}. For proving item (3), we let
  \[
  K = \Lambda_X(j',i'), \ K' =  \Lambda_X(j, i), \ M = \Lambda_X(j', i'-1),
  \]
  \[
  M' = \Lambda_X(j-e, i+1-2e), \ E = \Lambda_X(d+1, 2d+2), \
  E' = \Lambda_X(d'+1, 2d'+2).
  \]
 
To prove the commutativity of the left square in ~\eqref{eqn:SD-1-0},
  we first note that the diagram of pairings
  \begin{equation}\label{eqn:SD-1-1}
    \xymatrix@C1pc{
      [\Lambda_k, \pi_!K]_k \times
      [\Lambda_k,  \un{\Hom}_k(\pi_!K, \Lambda_k(1,2))]_k
      \ar[r]^-{\alpha_U} \ar@<7ex>[d]^-{(j_*)^\star} &
      H^2_\et(k, \Lambda(1)) \ar[d]^-{\id} \\
   [\Lambda_k, f_! K_U]_k \times
   [\Lambda_k,  \un{\Hom}_k(f_!K_U, \Lambda_k(1,2))]_k
   \ar[r]^-{\alpha_X}  \ar@<16ex>[u]^-{j_*} &
   H^2_\et(k, \Lambda(1))}
  \end{equation}
  is clearly commutative, where  for $S \in \{X,Y,U\}$, we let $\alpha_S$ denote
  the composition of morphisms (in $D_\et(k, \Lambda)$) on its left hand side
  followed by $\Tr_S$.

It remains therefore to show that the diagram
  \begin{equation}\label{eqn:SD-1-5}
    \xymatrix@C1.8pc{
      [\Lambda_X, K']_X \ar[d]_-{j^*} \ar[r]^-{\cong} &
      [\Lambda_k, \pi_*\un{\Hom}_X(K, E)]_k
      \ar[r]^-{(-) \circ \theta_X(K)} \ar[d]^-{j^*} & 
[\Lambda_k,  \un{\Hom}_k(\pi_!K, \Lambda_k(1,2))]_k 
      \ar[d]^-{(j_*)^\star} \\
         [\Lambda_X, j_*K'_U]_X \ar[r]^-{\cong} &
         [\Lambda_k, f_*\un{\Hom}_U(K_U, E_U)]_k
         \ar[r]^-{(-) \circ \theta_U(K_U)} &
         [\Lambda_k,  \un{\Hom}_k(f_!K_U, \Lambda_k(1,2))]_k}
    \end{equation}
  is commutative.
 But the left square in this diagram is clearly commutative, since its horizontal
  isomorphisms arise directly from the adjunction. The commutativity of the right
  square follows by applying \propref{prop:SD-4} to $K(-1,-2)$, in particular,
 using the commutativity of the right square in ~\eqref{eqn:SD-4-0}.

To prove the commutativity of the middle square in ~\eqref{eqn:SD-1-0},
  we note that the diagram of pairings
  \begin{equation}\label{eqn:SD-1-7}
    \xymatrix@C1pc{
      [\Lambda_k, f_!K_U]_k \times
      [\Lambda_k,  \un{\Hom}_k(f_!K_U, \Lambda_k(1,2))]_k
      \ar[r]^-{\alpha_Y} \ar@<7ex>[d]^-{\partial^\star} &
      H^2_\et(k, \Lambda(1)) \ar[d]^-{\id} \\
   [\Lambda_k, g_! M_Y]_k \times
   [\Lambda_k,  \un{\Hom}_k(g_!M_Y, \Lambda_k(1,2))]_k
   \ar[r]^-{\alpha_U}  \ar@<16ex>[u]^-{\partial} &
   H^2_\et(k, \Lambda(1))}
  \end{equation}
  is clearly commutative.

 It remains therefore to show that the diagram
 \begin{equation}\label{eqn:SD-1-4}
    \xymatrix@C1.8pc{
      [\Lambda_X, j_*K'_U]_X \ar[r]^-{\cong} \ar[d]_-{\partial'} &
      [\Lambda_k, f_*\un{\Hom}_U(K_U, E_U)]_k
\ar[r]^-{(-) \circ \theta_U(K_U)} \ar[d]^-{\partial'} &
[\Lambda_k,  \un{\Hom}_k(f_!K_U, \Lambda_k(1,2))]_k \ar[d]^-{\partial^\star} \\
[\Lambda_X, \iota_*M'_Y]_X \ar[r]^-{\cong} &
[\Lambda_k, g_*\un{\Hom}_Y(M_Y, E'_Y)]_k
\ar[r]^-{(-) \circ \theta_Y(M_Y)} &
   [\Lambda_k, \un{\Hom}_k(g_!M_Y, \Lambda_k(1,2))]_k}
 \end{equation}
 is commutative. But this diagram commutes by the same reasoning that
 establishes the commutativity of ~\eqref{eqn:SD-1-5} except that we now use the
 commutativity of the left square in ~\eqref{eqn:SD-4-0}.

Finally, to show the commutativity of the right square in ~\eqref{eqn:SD-1-0},
  we note that the diagram of pairings
  \begin{equation}\label{eqn:SD-1-6}
    \xymatrix@C1pc{
[\Lambda_k, g_! M_Y]_k \times
   [\Lambda_k,  \un{\Hom}_k(g_!M_Y, \Lambda_k(1,2))]_k
   \ar[r]^-{\alpha_Y} \ar@<7ex>[d]^-{(\iota^*)^\star} &
   H^2_\et(k, \Lambda(1)) \ar[d]^-{\id} \\
  [\Lambda_k, \pi_!M]_k \times
      [\Lambda_k,  \un{\Hom}_k(\pi_!M, \Lambda_k(1,2))]_k
      \ar[r]^-{\alpha_X} \ar@<16ex>[u]^-{\iota^*} &
      H^2_\et(k, \Lambda(1))}
    \end{equation}
      is clearly commutative.
      We now use  \propref{prop:SD-4} (more precisely, the commutativity of the
      middle square in ~\eqref{eqn:SD-4-0}) and repeat the previous arguments to
      conclude the proof of item (3), and hence of the theorem.
\end{proof}

\section{Right kernel of Brauer--Manin pairing}\label{sec:R-ker}
The goal of the next two sections is to study the right kernel of the
Brauer--Manin pairing for open varieties over local fields
and complete the proofs of the remaining theorems stated in \S~\ref{sec:Intro}.
In this section, we prove Theorem~\ref{thm:Main-10}.
We begin with an application of \thmref{thm:SD-1} which will be used in the
proof of Theorem~\ref{thm:Main-10}.

\subsection{General reciprocity map}\label{sec:Rec-map}
Let $k$ be an hdvf with finite residue field.
Let $p \ge 1$ denote the exponential characteristic of $k$.
Let $X \in \Reg_k$ be a connected scheme of dimension $d$ and let $i,j \in \Z$ be
two integers such that $i \le j+1$. We write $i' = 2d+2-i$ and $j' = d+1-j$.
For any abelian group $M$ and a prime $\ell$, we let $M^{(\ell')} :=
{\varprojlim}_{\ell \nmid n} M/n$ denote the prime-to-$\ell$ completion of $M$.
We let $\Lambda = \Z[\tfrac{1}{p}], {{\Q}/{\Z}}' = ({\Q}/{\Z})\{p'\}$ and
$\Lambda_n = {\Z}/n$ if $n$ is prime to $p$.

By the universal coefficient theorem, one has an inclusion
${H^i_c(X, \Lambda(j))}/n \inj H^i_c(X, \Lambda_n(j))$, where $H^i_c(X, \Lambda(j))$
is the motivic cohomology with compact support (cf. Definition~\ref{defn:MC-defn}).
On the other hand, the condition $i \le j+1$ implies by \cite[Cor.~7.3]{GKR}
that the {\'e}tale realization map $\epsilon^{i,j}_X \colon
H^i_c(X, \Lambda_n(j)) \to H^i_{\et,c}(X, \Lambda_n(j))$ is injective.
It follows that ${H^i_c(X, \Lambda(j))}/n$ is a finite group.

We endow each finite group with the discrete topology and endow
$H^i_c(X, \Lambda(j))$ with the subgroup topology generated by the neighborhood
system of the identity given by $\{nH^i_c(X, \Lambda(j))\}_{n \ge 1}$.
We consider the canonical maps
\begin{equation}\label{eqn:Main-map}
  \xymatrix@C1pc{
    H^i_c(X, \Lambda(j)) \ar[r]^-{\theta_X} \ar[d]_-{\rho_X} &
    {\varprojlim}_{n \in k^\times} {H^i_c(X, \Lambda(j))}/n \ar@{^{(}->}[r] &
  {\varprojlim}_{n \in k^\times}  H^i_c(X, \Lambda_n(j)) \ar@{^{(}->}[d] \\
  {\varprojlim}_{n \in k^\times}  H^{i'}_{\et}(X, \Lambda_n(j'))^\star & &
  {\varprojlim}_{n \in k^\times}  H^i_{\et,c}(X, \Lambda_n(j)) \ar[ll]_-{\cong},}
\end{equation}
where the bottom horizontal arrow is the duality isomorphism given by
\thmref{thm:SD-1}. 

Since $H^{i-1}_c(X, \Lambda_n(j)) \surj  H^i_c(X, \Lambda(j))[n]$ and the former
group is finite by the same reason as above, it follows that $\Ker(\theta_X)$
is the maximal divisible subgroup of $H^i_c(X, \Lambda(j))$. Note that
$H^i_c(X, \Lambda(j))\{p\} = 0$. Since all cohomology groups with finite
coefficients appearing in the above diagram
are finite, all inverse limits are canonically endowed with the profinite
topology and all maps are continuous homomorphisms.

We let
\[
\rho'_X \colon
\left({\varprojlim}_{n \in k^\times}  H^{i'}_{\et}(X, \Lambda_n(j'))^\star\right)^\vee
\to H^{i}_c(X, \Lambda(j))^\vee
\]
denote the map induced by $\rho_X$ upon taking
the continuous duals of its source and target.
Since $\left({\varprojlim}_{n \in k^\times}
H^{i'}_{\et}(X, \Lambda_n(j'))^\star\right)^\vee
\cong {\varinjlim}_{n \in k^\times}  \left(H^{i'}_{\et}(X, \Lambda_n(j'))^\star\right)^\vee
\cong {\varinjlim}_{n \in k^\times} H^{i'}_{\et}(X, \Lambda_n(j'))$,
we get that the continuous dual of $\rho_X$ is
the map
\begin{equation}\label{eqn:Main-map*}
\rho'_X \colon {\varinjlim}_{n \in k^\times} H^{i'}_{\et}(X, \Lambda_n(j'))
\xrightarrow{\cong}   H^{i'}_{\et}(X, {{\Q}/{\Z}}'(j'))  \to
H^{i}_c(X, \Lambda(j))^\vee.
\end{equation}

We shall usually refer to $\rho_X$ as the reciprocity map and $\rho'_X$ as the
dual reciprocity map for $H^{i}_c(X, \Lambda(j))$. 

\subsection{Class field theory for regular curves}\label{sec:Et-realn}
We now let $d =1$ in \S~\ref{sec:Rec-map} so that $X$ is
a regular (but not necessarily smooth) and connected curve over $k$.
We let $\ov{X}$ be a connected regular compactification of $X$ with the inclusion
$j \colon X \inj \ov{X}$. We let $S = \ov{X} \setminus X$ and let
$\iota \colon S \inj \ov{X}$ be the inclusion. Let $\pi \colon \ov{X} \to \Spec(k)$
be the structure map and let $f = \pi \circ j, \ g = \pi \circ \iota$.
Letting $C_1(X) :=  H^3_c(X, \Lambda(2))$, we get from ~\eqref{eqn:Main-map}
a continuous group homomorphism
 \begin{equation}\label{eqn:Main-map-0}
   \rho_X \colon C_1(X) \to \pi^{\ab}_1(X)^{(p')}.
   \end{equation}
One can think of $\rho_X$ as the reciprocity map for the  prime-to-$p$ class field
theory of $X$. We let $\rho'_X \colon H^1_\et(X, {{\Q}/{\Z}}') \to C_1(X)^\vee$ be
the continuous dual map of ~\eqref{eqn:Main-map*}.
Our main result concerning $\rho_X$ is the following extension of
the class field theory for smooth curves due to Saito \cite{Saito-JNT} and
Hiranouchi \cite{Hiranouchi-2}. This also extends the main result of
\cite{Gupta-Rathore}.

\begin{thm}\label{thm:CFT-reg}
  The kernel of $\rho_X$ is the maximal divisible subgroup of $C_1(X)$. Its
  topological cokernel is canonically isomorphic to the
  topological cokernel of $\rho_{\ov{X}}$, and both are isomorphic to
  ${\underset{\ell \neq p}\prod} \Z_\ell^{r_\ell}$, where $0 \le r_\ell \le r$ for
  some integer $r$ not depending on $\ell$.
\end{thm}
\begin{proof}
The first part is already shown in  \S~\ref{sec:Rec-map}. 
To prove the second part of the lemma, it is equivalent to show that
the cokernel of the composite continuous homomorphism
between profinite abelian groups
\[
H^3_{c}(X, \Lambda(2))^{(p')} \to {\varprojlim}_{n \in k^\times} H^3_{c}(X, \Lambda_n(2))
\xrightarrow{\epsilon^{3,2}_X} {\varprojlim}_{n \in k^\times}
H^3_{\et,c}(X, \Lambda_n(2))
\xrightarrow{\cong} H^1_\et(X, {\Q}/{\Z}')^\vee
\]
has the desired form.
Since the first arrow is an isomorphism by \lemref{lem:Vanishing-coh}
and the last arrow is an isomorphism by \thmref{thm:SD-1}, it suffices to
prove this claim for the {\'e}tale realization map $\epsilon^{3,2}_X \colon
{\varprojlim}_{n \in k^\times} H^3_{c}(X, \Lambda_n(2)) \to {\varprojlim}_{n \in k^\times}
H^3_{\et,c}(X, \Lambda_n(2))$. Ditto for $\ov{X}$.

To that end, we consider the diagram
  \begin{equation}\label{eqn:CFT-reg-0}
    \xymatrix@C1.7pc{
      H^2(S, \Lambda_n(2)) \ar[r]^-{\partial} \ar[d] &  H^3_{c}({X}, \Lambda_n(2))
      \ar[r]^-{j_*} \ar[d]^-{\epsilon^{3,2}_{X,n}} &
      H^3(\ov{X}, \Lambda_n(2)) \ar[r] \ar[d]^-{\epsilon^{3,2}_{\ov{X},n}} & 0 \\
      H^2_\et(S, \Lambda_n(2)) \ar[r]^-{\partial_\et} & H^3_{\et,c}({X}, \Lambda_n(2))
      \ar[r]^-{j_*} & H^3_\et(\ov{X}, \Lambda_n(2)) \ar[r] & 0,}
    \end{equation}
   whose vertical arrows are the  {\'e}tale realization maps.  

This is a commutative diagram (of finite groups) because the {\'e}tale realization
maps are compatible with the localization sequences. The maps $j_*$ in the two
rows are surjective because the next terms in the top and the bottom rows would
otherwise be $H^3(S, \Lambda_n(2))$ and $H^3_\et(S, \Lambda_n(2))$, respectively.
But these groups are zero by \thmref{thm:Tate-Duality-0} and \cite[Cor.~7.3]{GKR}.

As the left vertical arrow in ~\eqref{eqn:CFT-reg-0} is an isomorphism by
  \cite[Cor.~7.3]{GKR}, we get that the map
  $j_* \colon \coker(\epsilon^{3,2}_{X,n}) \to  \coker(\epsilon^{3,2}_{\ov{X},n})$ is an
  isomorphism. Since the canonical map
  $\coker(\epsilon^{3,2}_{X}) \to {\varprojlim}_{n \in k^\times}
  \coker(\epsilon^{3,2}_{X,n})$ is an isomorphism, we conclude that
 $j_* \colon \coker(\epsilon^{3,2}_{X}) \to \coker(\epsilon^{3,2}_{\ov{X}})$
  is an isomorphism of profinite groups.
It suffices therefore to prove the claim when $X$ is projective
  over $k$.
  
We assume for the rest of the proof that $X$ is projective. We can then find a
  commutative diagram (a trivial case of Temkin's $p$-alteration)
\begin{equation}\label{eqn:CFT-reg-1}
  \xymatrix@C1.7pc{
    X' \ar[r]^-{f} \ar[d]_-{\pi'} & X \ar[d]^-{\pi} \\
    \Spec(k') \ar[r] & \Spec(k),}
\end{equation}
where ${k'}/k$ is a finite purely inseparable extension (in particular, the trivial
extension if $\Char(k) = 0$), $X'$ is a connected smooth projective curve over
$k'$ and $f$ is a finite flat map of projective $k$-schemes of degree $p^m$
for some $m \ge 0$.

We let $H^i(X, \Z^{(p')}(j)) :=  {\varprojlim}_{n \in k^\times} H^i(X, \Lambda_n(j))$
and let $H^i_\et(X, \Z^{(p')}(j))$ be defined similarly.
We now look at the diagram
\begin{equation}\label{eqn:CFT-reg-2}
  \xymatrix@C1.7pc{
    H^3({X}, \Z^{(p')}(2)) \ar[r]^-{f^*} \ar[d]_-{\epsilon^{3,2}_X} &
     H^3({X'}, \Z^{(p')}(2)) \ar[r]^-{f_*} \ar[d]^-{\epsilon^{3,2}_{X'}} &
     H^3({X}, \Z^{(p')}(2)) \ar[d]^-{\epsilon^{3,2}_X} \\
   H^3_{\et}({X}, \Z^{(p')}(2)) \ar[r]^-{f^*}  &
     H^3_\et({X'}, \Z^{(p')}(2)) \ar[r]^-{f_*} &
     H^3_\et({X}, \Z^{(p')}(2)),}
  \end{equation}
where the horizontal arrows in the right square are the transfer maps
(cf. \cite[Prop.~6.3]{CD-Doc}).
All maps in this diagram are continuous $\Z^{(p')}$-linear maps
between profinite $\Z^{(p')}$-modules.
In particular, the cokernels of the vertical arrows coincide with the
respective topological cokernels and they are profinite $\Z^{(p')}$-modules.

The left square in ~\eqref{eqn:CFT-reg-2} commutes because the {\'e}tale realization
map is functorial and the right square commutes because the {\'e}tale realization
map preserves the transfer maps (cf. \cite[Prop.~9.4, Rem.~9.6]{CD-Doc}).
By \cite[Prop.~6.3]{CD-Doc}, the composite horizontal map on the
upper level of ~\eqref{eqn:CFT-reg-2} is multiplication by $p^m$.
The same is also true for the composite horizontal map on the lower level
by \cite[Prop.~6.1.8(4)]{CD-Comp}.

It follows that there are continuous homomorphisms of
profinite $\Z^{(p')}$-modules
\[
\coker(\epsilon^{3,2}_X) \xrightarrow{f^*} \coker(\epsilon^{3,2}_{X'})
 \xrightarrow{f_*} \coker(\epsilon^{3,2}_X)
 \]
 whose composition is multiplication by $p^m$.
 It follows that $\coker(\epsilon^{3,2}_X)$ is a closed $\Z^{(p')}$-submodule of
 $\coker(\epsilon^{3,2}_{X'})$. 
 Meanwhile, it follows from \cite[Prop.~6.11]{GKR}, \cite[Thm.~2.6]{Saito-JNT}
 and \cite[Lem.~5.4]{JS-Doc} (the latter result allows one to pass from $k$
 to its completion) that $\coker(\epsilon^{3,2}_{X'}) \cong (\Z^{(p')})^{r'}$ for some
 $r' \ge 0$. One deduces from \cite[Thm.~4.3.3]{Profin-book}
 that as a profinite abelian group,
 $\coker(\epsilon^{3,2}_X)$ has the desired form. This proves the claim, and
 completes the proof of the theorem.
\end{proof}

\subsection{The key lemma}\label{sec:Key-diag}
For the rest of \S~\ref{sec:R-ker}, we shall work under the following setting.
We let $R$ be an excellent hdvr with quotient field
$k$ and finite residue field $\F$ of characteristic $p$.
Let $\sX$ be a connected, regular and proper $R$-scheme of Krull dimension three.
We assume that $\sX$ is faithfully flat over $R$ and its reduced closed fiber $Y$
is an snc divisor on $\sX$. We let $X$ denote the generic fiber of
$\sX$ and assume that it is a smooth $k$-scheme. 

Let $j \colon U \inj X$ be the inclusion of a nonempty open subscheme and
$D = X \setminus U$. We let $j \colon U \inj X$ and $\iota \colon D \inj X$ be the
inclusions. We let $D^o \subset D$ be a dense open subscheme
which is regular of pure dimension one. We let $S = D \setminus D^o$ and 
$X^o = X \setminus S$. We let $u \colon U \inj X^o, \ v \colon D^o \inj X^o$ and
$w \colon X^o \inj X$ be the inclusions.

Recall from \S~\ref{sec:BM-left} that there is a pairing
$C_0(U) \times \Br(U) \to {\Q}/{\Z}$ which yields a continuous pairing of
topological abelian groups $C_0(U)^{(p')} \times \Br(U)\{p'\} \to {\Q}/{\Z}$,
where $C_0(U)^{(p')}$ has the profinite topology and $\Br(U)\{p'\}$ has the
discrete topology (cf. \lemref{lem:BM-fin}). We let $\Phi_U \colon \Br(U) \to
C_0(U)^\star$ denote the induced map. The goal of this section is a description
of the prime-to-$p$ part of $\Ker(\Phi_U)$. We shall do this by analyzing
$\Ker(\Phi_U)\{\ell\}$ separately for each prime $\ell \neq p$.
From now on therefore, we fix a prime $\ell \neq p$.
We let $\Lambda = \Z[\tfrac{1}{p}]$ and $\Lambda_n = {\Z}/n$.
We begin with some elementary lemmas.

\begin{lem}\label{lem:Elem-0}
  Let $M$ be an abelian group such that $M/{\ell^n}$ is finite for each $n \ge 1$
  and let $M^{(\ell)}$ be endowed with the profinite topology.
  Then there is a natural isomorphism of abelian groups
  $\gamma_M \colon M^\star\{\ell\} \xrightarrow{\cong} (M^{(\ell)})^\vee$.
\end{lem}
\begin{proof}
We have
  \[
M^\star\{\ell\}  \xrightarrow{\cong}  {\varinjlim}_n \Hom(\Lambda_{\ell^n}, M^\star)
     \xrightarrow{\cong}  {\varinjlim}_n \Hom(M/{\ell^n}, {\Q}/{\Z}) 
     \xrightarrow{\cong}  {\varinjlim}_n (M/{\ell^n})^\vee  \xrightarrow{\cong} 
    (M^{(\ell)})^\vee,
  \]
  where the third isomorphism follows by the finiteness of each $M/{\ell^n}$.
\end{proof}

\begin{cor}\label{cor:Elem-1}
  The isomorphism of topological groups
  $\vartheta_U \colon {C_0(U)}^{(\ell)} \xrightarrow{\cong}
  {\varprojlim}_n H^4_c(U, \Lambda_{\ell^n}(2))$ induces an isomorphism
    $\vartheta'_U \colon H^4_c(U, \Lambda(2))^\star\{\ell\} \xrightarrow{\cong}
    C_0(U)^\star\{\ell\}$ upon taking the continuous duals.
\end{cor}
\begin{proof}
By Propositions~\ref{prop:Suslin-hom-0} and ~\ref{prop:Tame-part}, there are
  canonical isomorphisms of profinite groups
  \begin{equation}\label{eqn:Elem-1-0}
    {C_0(U)}^{(\ell)} \xrightarrow{\cong} {H^S_0(U)}^{(\ell)} \xrightarrow{\cong}
    H^4_c(U, \Lambda(2))^{(\ell)} \xrightarrow{\cong}
    {\varprojlim}_n H^4_c(U, \Lambda_{\ell^n}(2)).
    \end{equation}
 Meanwhile, \lemref{lem:Elem-0} says that there are isomorphisms
    \begin{equation}\label{eqn:Elem-1-1}
      H^4_c(U, \Lambda(2))^\star\{\ell\}  \xrightarrow{\cong} 
      (H^4_c(U, \Lambda(2))^{(\ell)})^\vee \xrightarrow{(\vartheta_U)^\vee} 
      ({C_0(U)}^{(\ell)})^\vee \xleftarrow{\cong} C_0(U)^\star\{\ell\}.
      \end{equation}
Letting $M =  H^4_c(U, \Lambda(2)), N = C_0(U)$, the map
    $(\gamma_N)^{-1} \circ (\vartheta_U)^\vee \circ \gamma_M$ is the desired
    isomorphism $\vartheta'_U$.
    \end{proof}

Let $\psi_Z \colon H^2_\et(Z, {\Q_\ell}/{\Z_\ell}(1)) \to \Br(Z)\{\ell\}$ denote the
natural surjective map induced by the Kummer sequence for $Z \in \{X, U\}$.
      
\begin{lem}\label{lem:Elem-2}
  There exists a commutative diagram
  \begin{equation}\label{eqn:Elem-2-1}
    \xymatrix@C1pc{
      H^2_{\rm {\etl}}(X, {\Q_\ell}/{\Z_\ell}(1)) \ar[dd]_-{\rho'_X} \ar@{->>}[dr]^-{\psi_X}
      \ar[rr] &  &
H^2_{\rm {\etl}}(U, {\Q_\ell}/{\Z_\ell}(1)) \ar[dd]^->>>>>{\rho'_U} \ar@{->>}[dr]^-{\psi_U} & \\
& \Br(X)\{\ell\} \ar[rr] \ar[dd]^->>>>>{\Phi_X} & & \Br(U)\{\ell\}
\ar[dd]^-{\Phi_U} \\
H^4(X, \Lambda(2))^\star\{\ell\} \ar[rr] \ar[dr]_-{\vartheta'_X}^-{\cong} & &
H^4_c(U, \Lambda(2))^\star\{\ell\} \ar[dr]^-{\vartheta'_U}_-{\cong} & \\
& \CH_0(X)^\star\{\ell\} \ar[rr] & & C_0(U)^\star\{\ell\},}
  \end{equation}
  whose all horizontal arrows are the pull-back maps induced by $j \colon U \inj X$.
  \end{lem}
\begin{proof}
The top face commutes by the naturality of the Kummer sequence, the front face
  commutes by the naturality of the Brauer--Manin pairing, the bottom face commutes
  by the naturality of $\vartheta_U$ and the back face commutes by
  \thmref{thm:SD-1}.
  The commutativity of the left and the right faces follows from the
  commutative diagram~\ref{eqn:Main-0-1-0} and the definition of $\vartheta'_X$
  and $\vartheta'_U$. Finally, $\vartheta'_X$ and $\vartheta'_U$ are isomorphisms by
  \corref{cor:Elem-1}.
\end{proof}

\begin{lem}\label{lem:Elem-3}
  Let $A \in \{\Lambda, \Lambda_{\ell^n}| n \ge 1\}$. Then $M/{\ell^n}$ (resp. $M$) is
  finite if $A = \Lambda$ (resp. $A = \Lambda_{\ell^n}$) whenever
  $M \in \{H^4_c(X^o, A(2)),  H^4_c(U, A(2)), H^3_c(D^o, A(2)), H^3_c(X^o, A(2))\}$.
  \end{lem}
\begin{proof}
By the universal coefficient theorem and \lemref{lem:Vanishing-coh}, it suffices to
  show that $M$ is finite if
  $M \in \Sigma = \{H^4(X, \Lambda_{\ell^n}(2)),  H^4_c(U, \Lambda_{\ell^n}(2)),
  H^3_c(D^o, \Lambda_{\ell^n}(2)), H^3_c(X^o, \Lambda_{\ell^n}(2))\}$.
  Now, the last two groups in $\Sigma$ are finite by \cite[Cor.~7.3]{GKR} and
  \thmref{thm:SD-1}. For the other two groups, we consider the exact sequence
  \[
  H^3(D, \Lambda_{\ell^n}(2)) \to H^4_c(U, \Lambda_{\ell^n}(2)) \to 
  H^4(X, \Lambda_{\ell^n}(2)).
  \]
  The first group in this sequence is finite by \cite[Cor.~7.3]{GKR} and
  \lemref{lem:Etale-fin}. It remains therefore to show that
  $H^4(X, \Lambda_{\ell^n}(2))$ is finite. But this follows from the
  isomorphism ${\CH_0(X)}/{\ell^n} \xrightarrow{\cong}  H^4(X, \Lambda_{\ell^n}(2))$
  due to Voevodsky \cite{Voevodsky-IMRN}, and the finiteness of
  ${\CH_0(X)}/{\ell^n}$, already shown earlier.
\end{proof}

\begin{lem}\label{lem:Elem-4}
  There exists a commutative diagram 
   \begin{equation}\label{eqn:Elem-4-0}
    \xymatrix@C1pc{
H^2_{\rm {\etl}}(X, {\Q_\ell}/{\Z_\ell}(1)) \ar[r]^-{j^*} \ar[d]_-{\beta_X} & 
H^2_{\rm {\etl}}(U, {\Q_\ell}/{\Z_\ell}(1)) \ar[r]^-{\partial'} \ar[d]^-{\beta_U} & 
H^1_{\rm {\etl}}(D^o, {\Q_\ell}/{\Z_\ell}) \ar[r]^-{v_*} \ar[d]^-{\beta_{D^o}} &
H^3_{\rm {\etl}}(X^o, {\Q_\ell}/{\Z_\ell}(1)) \ar[d]^-{\beta_{X^o}} \\
H^4(X, \Z_\ell(2))^\vee \ar[r]^-{(j_*)^\vee} \ar[d]_{\alpha_X} & H^4_c(U, \Z_\ell(2))^\vee
\ar[r]^-{\partial^\vee} \ar[d]^-{\alpha_{U}} & H^3_c(D^o, \Z_\ell(2))^\vee
\ar[r]^-{(v^*)^\vee}  \ar[d]^-{\alpha_{D^o}} &
H^3_c(X^o, \Z_\ell(2))^\vee \ar[d] \\
H^4(X, \Lambda(2))^\star\{\ell\} \ar[r]^-{(j_*)^\star} & 
H^4_c(U, \Lambda(2))^\star\{\ell\} \ar[r]^-{\partial^\star} &
H^3_c(D^o, \Lambda(2))^\star\{\ell\} \ar[r]^-{(v^*)^\star} &
H^3_c(X^o, \Lambda(2))^\star\{\ell\},}
   \end{equation}
   in which the top and the middle rows are exact, the bottom row is exact
   at its first two places, $(j_*)^\vee$ and $(j_*)^\star$
   are injective and $\alpha_Z$ is bijective for $Z \in \{X, U, D^o\}$. 
   \end{lem}
\begin{proof}
We first note that all cohomology groups with coefficient in
  $\Lambda_{\ell^n}(i)$ for $i \le 2$ appearing in the given diagram are finite by
  \thmref{thm:SD-1} and \lemref{lem:Elem-3}.
  We next observe that the vertical arrows in the upper level of
  ~\eqref{eqn:Elem-4-0}
  are  obtained by composing the duality isomorphisms of
\thmref{thm:SD-1} with the duals of the étale realization maps, and then passing
to the limit. Hence, every square in the upper level commutes by
\thmref{thm:SD-1}.

The vertical arrows in the lower level are obtained by taking the continuous duals
of the canonical inclusions
${H^i_c(Z, \Lambda(j))}/{\ell^n} \inj H^i_c(Z, \Lambda_{\ell^n}(j))$ for
  $Z \in \{X, U, D^o, X^o\}$, and precomposing them with the inverses of the
isomorphisms of \lemref{lem:Elem-0} for the corresponding cohomology groups.
It follows that every square in the lower level of \eqref{eqn:Elem-4-0} also
commutes.

Using the localization sequence and purity theorem for {\'e}tale cohomology,
  one checks that the map
  $H^2_\et(X, {\Q_\ell}/{\Z_\ell}(1)) \to H^2_\et(X^o, {\Q_\ell}/{\Z_\ell}(1))$ is
  an isomorphism. On the other hand, the map
  $H^4_c(X^o, \Lambda_{\ell^n}(2)) \to H^4(X, \Lambda_{\ell^n}(2))$ is an isomorphism
  for every $n \ge 1$ by \lemref{lem:Vanishing-coh}. The inclusion
  ${H^4_c(Z, \Lambda(2))}/{\ell^n} \inj H^4_c(Z, \Lambda_{\ell^n}(2))$
  is an isomorphism for $Z \in \{X, U, X^o\}$ and
  ${H^3_c(D^o, \Lambda(2))}/{\ell^n} \inj H^3_c(D^o, \Lambda_{\ell^n}(2))$
  is an isomorphism by the universal coefficient theorem and
  \lemref{lem:Vanishing-coh}. Together, these isomorphisms show that $\alpha_Z$ is
  bijective for $Z \in \{X, U, D^o\}$ and that the map
  ${H^4_c(X^o, \Lambda(2))}/{\ell^n} \to  {H^4(X, \Lambda(2))}/{\ell^n}$ is bijective
  for every $n \ge 1$. It follows from \lemref{lem:Elem-0} that all
  cohomology groups in  the first column of ~\eqref{eqn:Elem-4-0} remain unchanged
  if we replace $X$ by $X^o$.

Since the map $H^4_c(U, \Lambda_{\ell^n}(2)) \to H^4_c(X^o, \Lambda_{\ell^n}(2))$
  is a surjective map of finite groups by ~\eqref{eqn:Mot-loc} and
  \lemref{lem:Vanishing-coh}, it follows that
  $j_* \colon H^4_c(U, \Z_\ell(2)) \to H^4_c(X^o, \Z_\ell(2))$ is a surjective
  homomorphism of profinite groups. One deduces that $(j_*)^\vee$ is injective.
Equivalently, $(j_*)^\star$ is injective.
  The exactness of the top row is clear, and the same for the middle row
  follows by applying the localization sequence for the motivic cohomology
  with coefficient in $\Lambda_{\ell^n}$, passing to the limit and then using the
  Pontryagin duality for the profinite groups.
  This also implies that the bottom row is exact at the first and
  the second places. This completes the proof.
  \end{proof}

The key lemma we want to prove is the following.

\begin{lem}\label{lem:Br-Chow-diag}
  There exists a commutative diagram of complexes of abelian groups
  \begin{equation}\label{eqn:Br-Chow-diag-0}
    \xymatrix@C1pc{
    0 \ar[r] & \Br(X)\{\ell\} \ar[r]^-{j^*} \ar[d]_-{\beta'_X} &   
    \Br(U)\{\ell\} \ar[r]^-{\wt{\partial}'} \ar[d]^-{\beta'_U} &
    H^1_{\rm {\etl}}(D^o, {\Q_\ell}/{\Z_\ell})
    \ar[r]^-{v_*} \ar[d]^-{\beta_{D^o}} & H^3_{\rm {\etl}}(X^o, {\Q_\ell}/{\Z_\ell}(1))
    \ar[d]^-{\beta_{X^o}} \\
    0 \ar[r] & H^4(X, \Z_\ell(2))^\vee \ar[r]^-{(j_*)^\vee} \ar[d]_{\alpha'_X}^-{\cong} &
    H^4_c(U, \Z_\ell(2))^\vee
\ar[r]^-{\partial^\vee} \ar[d]^-{{\alpha}'_{U}}_-{\cong} & H^3_c(D^o, \Z_\ell(2))^\vee
\ar[r]^-{(v^*)^\vee}  \ar[d]^-{\alpha_{D^o}}_-{\cong} &
H^3_c(X^o, \Z_\ell(2))^\vee \ar@{->>}[d] \\
0 \ar[r] & \CH_0(X)^\star\{\ell\} \ar[r]^-{(j_*)^\star} & C_0(U)^\star\{\ell\}
    \ar[r]^-{\wt{\partial}^\star} &
    H^3_c(D^o, \Lambda(2))^\star\{\ell\}
        \ar[r]^-{(v^*)^\star} & H^3_c(X^o, \Lambda(2))^\star\{\ell\},}
  \end{equation}
  in which the top and the middle rows are exact and the bottom row is exact
  at the first two places. Moreover, we have $\alpha'_X \circ \beta'_X =
  \Phi_X, \ \alpha'_U \circ \beta'_U = \Phi_U$ and
  $\alpha_{D^o} \circ \beta_{D^o} = \rho'_{D^o}$.
\end{lem}
\begin{proof}
First of all, one observes using the Kummer sequence that the pull-back map
  $\coker\left(H^2_\et(X, {\Q_\ell}/{\Z_\ell}(1)) \xrightarrow{j^*}
H^2_\et(U, {\Q_\ell}/{\Z_\ell}(1))\right) \to
\coker\left(\Br(X)\{\ell\} \xrightarrow{j^*} \Br(U)\{\ell\}\right)$
is an isomorphism. It follows that there exists a unique homomorphism
$\wt{\partial}'$
as in the top row of ~\eqref{eqn:Br-Chow-diag-0} such that the diagram
\begin{equation}\label{eqn:Br-Chow-diag-1}
    \xymatrix@C1pc{
H^2_\et(X, {\Q_\ell}/{\Z_\ell}(1)) \ar[r]^-{j^*} \ar@{->>}[d]_-{\psi_X} & 
H^2_\et(U, {\Q_\ell}/{\Z_\ell}(1)) \ar[r]^-{\partial'} \ar@{->>}[d]^-{\psi_U} & 
H^1_\et(D^o, {\Q_\ell}/{\Z_\ell}) \ar[r]^-{v_*} \ar[d]^-{\id} &
H^3_\et(X^o, {\Q_\ell}/{\Z_\ell}(1)) \ar[d]^-{\id} \\
\Br(X)\{\ell\} \ar[r]^-{j^*} &  \Br(U)\{\ell\} \ar[r]^-{\wt{\partial}'}  &
H^1_\et(D^o, {\Q_\ell}/{\Z_\ell}) \ar[r]^-{v_*} &
H^3_\et(X^o, {\Q_\ell}/{\Z_\ell}(1))}
\end{equation}
commutes and the bottom sequence is exact. Moreover, $j^*$ in the bottom row is
injective (cf. \cite[Thm.~3.5.7]{CTS}).

The identity $\alpha_{D^o} \circ \beta_{D^o} = \rho'_{D^o}$ follows immediately from
the definition of $\rho'_{D^o}$ in \S~\ref{sec:Rec-map}.
To prove the remaining assertions, we define
\[
\alpha'_X = \vartheta'_X \circ \alpha_X, \ \ \alpha'_U =
\vartheta'_U \circ \alpha_U, \
\ \beta'_X = (\alpha_X)^{-1} \circ (\vartheta'_X)^{-1} \circ \Phi_X,
\]
\[
\beta'_U = (\alpha_U)^{-1} \circ (\vartheta'_U)^{-1} \circ \Phi_U, \ \
\wt{\partial}^\star = \partial^\star \circ (\vartheta'_U)^{-1}.
\]
Then, with the possible exception of the commutativity of the upper and lower
middle squares in ~\eqref{eqn:Br-Chow-diag-0}, all the remaining assertions of the
lemma follow immediately from a combination of Lemmas~\ref{lem:Elem-2} and
~\ref{lem:Elem-4}.

The commutativity of the upper middle square in ~\eqref{eqn:Br-Chow-diag-0}
follows from that of the composite
middle square. To prove the commutativity of the latter, it suffices, by the
surjectivity of $\psi_U$ and the definition of $\wt{\partial}'$,
to show that
$\wt{\partial}^\star \circ \alpha'_U \circ \beta'_U \circ \psi_U =
\alpha_{D^o} \circ \beta_{D^o} \circ \partial'$.
To that end, we compute
\[
\begin{array}{lll}
  \wt{\partial}^\star \circ \alpha'_U \circ \beta'_U \circ \psi_U & = &
  \wt{\partial}^\star \circ \vartheta'_U \circ \alpha_U \circ (\alpha_U)^{-1} \circ
  (\vartheta'_U)^{-1} \circ \Phi_U \circ \psi_U = 
  \wt{\partial}^\star \circ  \Phi_U \circ \psi_U \\
  & = & \partial \circ  (\vartheta'_U)^{-1} \circ \Phi_U \circ \psi_U  \ \ {=}^{1} 
  \ \ \partial \circ \rho'_U \\
  & {=}^{2} & \partial \circ \alpha_U \circ \beta_U \ \ {=}^{3} \ \ 
  \alpha_{D^o} \circ \beta_{D^o} \circ \partial'.
\end{array}
\]

Here, equality ${=}^1$ follows from \lemref{lem:Elem-2}, equality
${=}^{2}$ follows from the identity $\rho'_U =  \alpha_U \circ \beta_U$ which is the
definition of $\rho'_U$, and  equality ${=}^{3}$ from \lemref{lem:Elem-4}.
Finally, the lower middle square commutes because
$\wt{\partial}^\star \circ \alpha'_U = \partial^\star \circ  (\vartheta'_U)^{-1} \circ
\vartheta'_U \circ \alpha_U = \partial^\star \circ  \alpha_U = \alpha_{D^o} \circ
\partial^\vee$, where the last equality follows from  \lemref{lem:Elem-4}.
This completes the proof.
\end{proof}

\subsection{Proof of \thmref{thm:Main-10}}\label{sec:Main-10-prf}
We shall continue to work under the setting of \S~\ref{sec:Key-diag}.
We need a couple of more results to prove \thmref{thm:Main-10}.
The first is an application of a theorem of Esnault--Wittenberg
\cite{Esnault-Wittenberg}.

\begin{lem}\label{lem:EW-surjection}
Assume that $\Alb_X$ has potentially good reduction and the irreducible
  components of $Y$ satisfy the Tate conjecture. 
Then the map $\Phi_X \colon \Br(X)\{\ell\} \to \CH_0(X)^\star\{\ell\}$ is surjective.
\end{lem}
\begin{proof}
It is easy to check that the diagram
  \begin{equation}\label{eqn:EW-surjection-0}
    \xymatrix@C1pc{
      H^2_\et(X, {\Q_\ell}/{\Z_\ell}(1)) \ar[r]^-{\psi_X} \ar[dr]_-{\vartheta'_X \circ
        \rho'_X} & \Br(X)\{\ell\}
      \ar[d]^-{\Phi_X} \\
      & \CH_0(X)^\star\{\ell\}}
  \end{equation}
  is commutative.
 Therefore, it suffices to show that $\vartheta'_X \circ \rho'_X$
  is surjective. By Lemmas~\ref{lem:Elem-0} and ~\ref{lem:Elem-3}, this is
  equivalent to showing that the homomorphism of discrete torsion groups
  $\rho'_X \colon H^2_\et(X, {\Q_\ell}/{\Z_\ell}(1)) \to (\CH_0(X)^{(\ell)})^\vee$
  is surjective. By \thmref{thm:SD-1} and Pontryagin duality, this in turn is
  equivalent to showing that the continuous homomorphism of profinite groups 
  $\rho_X = \cyc_X \colon \CH_0(X)^{(\ell)} \to H^4_\et(X, \Z_\ell(2))$ is injective.
  But this is precisely \cite[Thm.~3.1]{Esnault-Wittenberg}.
\end{proof}

The second result we need is the following. Recall the
$(\star \star_\ell)$-condition from \S~\ref{sec:Main-0-prf}.

\begin{lem}\label{lem:EW-injection}
  Assume that the $(\star \star_\ell)$-condition holds for $X$. Then the
  map
  \[
  \beta_{X^o} \colon H^3_{\rm {\etl}}(X^o, {\Q_\ell}/{\Z_\ell}(1)) \to H^3_c(X^o, \Z_\ell(2))^\vee
  \]
  is an isomorphism.
\end{lem}
\begin{proof}
By \thmref{thm:SD-1} and Pontryagin duality for profinite groups,
  the lemma is equivalent to showing that the {\'e}tale realization map
  $\epsilon^{3,2}_{X^o} \colon H^3_c(X^o, \Z_\ell(2)) \to  H^3_{\et,c}(X^o, \Z_\ell(2))$ 
  is an isomorphism. To prove the latter assertion, we proceed as follows.

We let $n \ge 1$ and consider the commutative diagram of
  exact sequences (cf. ~\eqref{eqn:MC-03})
  \begin{equation}\label{eqn:EW-injection-0}
    \xymatrix@C0.8pc{
      H^2(X, \Lambda_{\ell^n}(2)) \ar[r] \ar[d] & H^2(S, \Lambda_{\ell^n}(2)) \ar[r]
      \ar[d]
      &   H^3_c(X^o, \Lambda_{\ell^n}(2)) \ar[r]^-{w_*} \ar[d] &
      H^3(X, \Lambda_{\ell^n}(2)) \ar[r] \ar[d] & 0 \\
    H^2_\et(X, \Lambda_{\ell^n}(2)) \ar[r] & H^2_\et(S, \Lambda_{\ell^n}(2)) \ar[r] 
      &   H^3_{\et,c}(X^o, \Lambda_{\ell^n}(2)) \ar[r]^-{w_*} &
    H^3_\et(X, \Lambda_{\ell^n}(2)), &}
  \end{equation}
  whose vertical arrows are the {\'e}tale realization maps.
Recall here that $S = X \setminus X^o$ and $w \colon X^o \inj X$ is the
  inclusion. The map $w_*$ in the top row is surjective because the next term in
  the localization sequence is $H^3(S, \Lambda_{\ell^n}(2))$ which is zero by the
  Tate duality (cf. \thmref{thm:SD-1}).

The first two vertical arrows in ~\eqref{eqn:EW-injection-0} are isomorphisms
  and the last two vertical arrows are injective by
  \cite[Cor.~7.3]{GKR}. As all groups in the bottom row are finite by
  \thmref{thm:SD-1}, it follows that  ~\eqref{eqn:EW-injection-0} is a commutative
  diagram of finite groups. As this is clearly compatible with the change in
  the values of $n \ge 1$, we can pass to the limit and still get a
  commutative diagram
  \begin{equation}\label{eqn:EW-injection-1}
    \xymatrix@C0.8pc{
      H^2(X, \Z_\ell(2)) \ar[r] \ar[d] & H^2(S, \Z_\ell(2)) \ar[r] \ar[d]
      &   H^3_c(X^o, \Z_\ell(2)) \ar[r]^-{w_*} \ar[d]^-{\epsilon^{3,2}_{X^o}} &
      H^3(X, \Z_\ell(2)) \ar[r] \ar[d]^-{\epsilon^{3,2}_{X}} & 0 \\
    H^2_\et(X, \Z_\ell(2)) \ar[r] & H^2_\et(S, \Z_\ell(2)) \ar[r]
      &   H^3_{\et,c}(X^o, \Z_\ell(2)) \ar[r]^-{w_*}  &
    H^3_\et(X, \Z_\ell(2)), &}
  \end{equation}
  whose rows remain exact and the first two vertical arrows are isomorphisms.
We are now done because $\epsilon^{3,2}_{X}$ is an isomorphism by
  \thmref{thm:Main-0-0}. In particular, $w_*$ is surjective also in the bottom row. 
\end{proof}

We are now ready to prove \thmref{thm:Main-10}. We restate it here to make the
assumptions and notation explicit. We let $\Br(U)^{\perp} = \Ker(\Phi_U)$ and
$H^1_\et(D^o)^{\perp} = \Ker(\rho'_{D^o}) \bigcap \Ker(v_*)$.

\begin{thm}\label{thm:Main-10*}
  Assume that $\sX$ is a projective $R$-scheme and
  the $(\star \star_\ell)$-condition holds. Then we have the
  following.
  \begin{enumerate}
    \item
      $H^1_{\rm {\etl}}(D^o)^{\perp}\{p'\} = \Ker(\rho'_{D^o})\{p'\} \cong
      {\underset{\ell \neq p}\bigoplus} ({\Q_\ell}/{\Z_\ell})^{r_\ell}$, where
      $0 \le r_\ell \le r$ for
  some integer $r$ not depending on $\ell$. 
    \item
    There exists an exact sequence
  \[
  0 \to \Br(\sX)\{p'\} \to  \Br(U)^{\perp}\{p'\} \to
  {\underset{\ell \neq p}\bigoplus} ({\Q_\ell}/{\Z_\ell})^{r_\ell} \to 0.
  \]
  \end{enumerate}
\end{thm}
\begin{proof}
It suffices to prove the statements of the theorem for
  the $\ell$-primary torsion subgroups for each individual $\ell \neq p$.
  We now fix $\ell \neq p$.
Item (1) then follows from \thmref{thm:CFT-reg}, and Lemmas~\ref{lem:Br-Chow-diag}
  and ~\ref{lem:EW-injection}. For item (2), we note that a combination of
  Lemmas~\ref{lem:Br-Chow-diag}, ~\ref{lem:EW-surjection} and
  ~\ref{lem:EW-injection} implies by a diagram chase in
  ~\eqref{eqn:Br-Chow-diag-0} that there is an exact sequence
  \begin{equation}\label{eqn:Main-10*-0}
    0 \to \Ker(\Phi_X)\{\ell\} \to \Ker(\Phi_U)\{\ell\} \to
    \Ker(\rho'_{D^o})\{\ell\} \to 0.
  \end{equation}
 We now apply \cite[Thm.~1.6]{KM-0} (see \cite[Thm.~1.1.3]{Saito-Sato-ENS}
  when $k$ is complete and $\Char(k) = 0$), which says that the
  canonical map $\Br(\sX) \to \Ker(\Phi_X)$ is an isomorphism.
  This completes the proof.
\end{proof}

\begin{cor}\label{cor:Main-10*-1}
  Assume in \thmref{thm:Main-10*} that each irreducible component of $D$ is a
  smooth $k$-scheme and its Jacobian has good reduction. Then the map
  $\Br(\sX)\{p'\} \to  \Br(U)^{\perp}\{p'\}$ is an isomorphism.
\end{cor}
\begin{proof}
  It is known in this case that $r = 0$ (cf. \cite[\S~6]{Saito-JNT}).
  \end{proof}

\section{The remaining proofs}\label{sec:Remarks}
The goal of this section is to prove the remaining results stated in
\S~\ref{sec:Intro}.
We let the setting be of \S~\ref{sec:Key-diag} and let $\sD \subset \sX$
be the closure of $D$ with the reduced closed subscheme structure and let
$\sU = \sX \setminus \sD$ so that $\sU$ is the largest open subscheme of $\sX$
whose generic fiber is $U$.

\subsection{Proof of \thmref{thm:Main-8}}\label{sec:Pf-main-8**}
The following two examples will
prove \thmref{thm:Main-8}.
\subsection{Example~1:}\label{sec:Exm-1}
We let $R = \Z_2$ and $\sX = \P^2_R$ so that $X = \P^2_k$ and $Y = \P^2_{\F}$. We let
$D \subset X$ be the elliptic curve given by the homogeneous equation
$f(x,y,z) = y^2z +x^3 + x^2z - 2z^3$. Then all hypotheses of \thmref{thm:Main-10*}
are satisfied. Furthermore, it is classically known that $r_\ell =1$ for every
$\ell \neq 2$ so that we get
$\Br(\sX) \cong \Br(R) = 0$ (cf. \cite[Cor.~6.1.4]{CTS}) while
$\Br(U)^\perp\{\ell\} \cong {\Q_\ell}/{\Z_\ell}$ for every $\ell \neq 2$.

\subsection{Example~2:}\label{sec:Exm-2}
We let $k$ be any local field with the ring of integers $R$ and
let $\sX = \P^1_R \times_R \P^1_R$ and
$\sU = \G_{m,R} \times_R \P^1_R$ so that $X = \P^1_k \times_k \P^1_k, \
U = \G_{m,k} \times \P^1_k$ and $D = \{0, \infty\} \times \P^1_k$.
Since all  hypotheses of \thmref{thm:Main-10*} are satisfied and $r =0$, we get
$\Br(\sX)\{p'\} = \Br(U)^\perp\{p'\} = 0$ by \cite[Cor.~6.1.4]{CTS}, where $p$ is the
residue characteristic of $k$. We shall show that
$\Br(\sU)\{p'\} \cong {\Q}/{\Z}\{p'\}$. By op. cit., it suffices to show the
following.

\begin{lem}\label{lem:Exm-2**}
$\Br(\G_{m,R})\{p'\} \cong {\Q}/{\Z}\{p'\}$. 
\end{lem}
\begin{proof}
We write $A = R[x]$ and $B = R[x, x^{-1}])$ so that $\G_{m,R} = \Spec(B)$.
We then have the exact sequence
\begin{equation}\label{eqn:Exm-2-0}
  \Br(A) \to \Br(B) \xrightarrow{\partial} H^3_{\{0\}}(\A^1_R, \G_m)
    \xrightarrow{\alpha} H^3_\et(\A^1_R, \G_m) \xrightarrow{\beta}
    H^3_\et(\G_{m,R}, \G_m).
\end{equation}
The homotopy invariance of the {\'e}tale cohomology with coefficients in
${\Q_\ell}/{\Z_\ell}(1)$ and the Kummer sequence tell us that
$\Br(A)\{p'\} \xrightarrow{\cong} \Br(R)\{p'\} = 0$.
We claim that $\beta$ is injective on the $p'$-primary torsion subgroup.

Indeed, we have a commutative diagram
\begin{equation}\label{eqn:Exm-2-1-0}
  \xymatrix@C1pc{
    H^3_\et(A, {\Q_\ell}/{\Z_\ell}(1)) \ar@{->>}[r] \ar[d] &
    H^3_\et(\A^1_R, \G_m)\{\ell\}
    \ar[d]^-{\beta} \\
    H^3_\et(B, {\Q_\ell}/{\Z_\ell}(1)) \ar@{->>}[r] & H^3_\et(\G_{m,R}, \G_m)\{\ell\},}
\end{equation}
where the horizontal arrows are induced by the Kummer sequence. Since
$\Br(A) \otimes {\Q_\ell}/{\Z_\ell} \cong \Br(B) \otimes {\Q_\ell}/{\Z_\ell} = 0$,
we see that the horizontal arrows are bijective.
On the other hand, using the homotopy invariance of the
{\'e}tale cohomology with coefficients in ${\Q_\ell}/{\Z_\ell}(1)$ and
the existence of $R$-points on $\G_{m,R}$, one finds that the left vertical arrow
is injective. This proves the claim.

After the above reductions, it remains to show that
$H^3_{\{0\}}(\A^1_R, \G_m)\{p'\} \cong {\Q}/{\Z}\{p'\}$.
To that end, we use the Kummer sequence again to get an exact sequence
\begin{equation}\label{eqn:Exm-2-1}
  0  \to H^2_{\{0\}}(\A^1_R, \G_m) \otimes {\Q_\ell}/{\Z_\ell} \to
    H^3_{\{0\}}(\A^1_R, {\Q_\ell}/{\Z_\ell}(1))  \to
      H^3_{\{0\}}(\A^1_R, \G_m)\{\ell\} \to 0
\end{equation}
for any $\ell \neq p$. 
The exact sequence
$0 = \Pic(\G_{m,R}) \to H^2_{\{0\}}(\A^1_R, \G_m) \to \Br(A)$ says that
$H^2_{\{0\}}(\A^1_R, \G_m)$ is a torsion group. In particular,
$ H^2_{\{0\}}(\A^1_R, \G_m) \otimes {\Q_\ell}/{\Z_\ell} =0$.
This yields $H^3_{\{0\}}(\A^1_R, \G_m)\{\ell\} \cong
H^3_{\{0\}}(\A^1_R, {\Q_\ell}/{\Z_\ell}(1)) \cong H^1_\et(R, {\Q_\ell}/{\Z_\ell})
\cong H^1_\et(\F, {\Q_\ell}/{\Z_\ell}) \cong {\Q_\ell}/{\Z_\ell}$, where the second
isomorphism follows from Gabber's purity. The proof is complete.
\end{proof}

\subsection{Example~3:}\label{sec:Exm-3}
This example shows that the $(\star \star_\ell)$ condition in
item (2) of \thmref{thm:Main-10*} is necessary.
We let $p$ be a prime and let $k$ be a $p$-adic field.
Let $Y \subset \P^n_k$ be a smooth projective curve and let
$\wh{Y} \subset \P^{n+1}_k$ be the projective cone over $Y$. Let
$X$ denote the blow-up of $\wh{Y}$ at its vertex $P$ and let
$\pi \colon X \to \wh{Y}$ denote the blow-up map. We let
$U = X \setminus E$, where $E \subset X$ is the
exceptional divisor. Then one checks using
\cite[Props.~6.2.7, 8.4.1]{CTS} that the inclusion $U \inj X$ induces an
isomorphism $\Br(X) \xrightarrow{\cong} \Br(U)$.

Using the commutative diagram
\begin{equation}\label{eqn:Exm-3-0}
  \xymatrix@C1.6pc{
    \Br(X) \ar[r]^-{\cong} \ar[d]_-{\Phi_X} & \Br(U) \ar[d]^-{\Phi_U} \\
    \CH_0(X)^\star \ar@{^{(}->}[r] & C_0(U)^\star,}
\end{equation}
we get that the map $\Br(X)^\perp \to \Br(U)^\perp$ is an isomorphism. If we let $\sX$
denote a semistable model  of $X$ (which exists by \cite{Cossart-Piltant}), it
follows that the map $\Br(\sX) \to \Br(U)^\perp$ is an isomorphism.
On the other hand, we can choose $Y$ such that its rank is positive so that
$\Ker(\rho'_E)\{p'\} \neq 0$. In particular, the map
$\Br(U)^\perp\{p'\} \to \Ker(\rho'_E)\{p'\}$ is not
surjective.

\vskip.2cm

The above examples motivate the following.

\begin{ques}\label{ques:Exm-3}
  With the notation as above, is there an example in which 
  $0 \neq \Br(\sX) \neq \Br(U)^\perp \neq \Br(\sU)$?
\end{ques}

\vskip.2cm

\subsection{Comments on Question~\ref{ques:Exm-3}}\label{sec:Comments-3}
It is believable that the above question has a positive answer. Here, we sketch a
strategy for producing example in which $\Br(\sX) \neq \Br(U)^\perp \neq \Br(\sU)$.
Let $p$ be a prime and let $k$ be a $p$-adic field with the ring of integers $R$
and residue field $\F$. We let $\sX = \P^2_R$ so that $X = \P^2_k$ and
$Y = \P^2_\F$. We let $\sD \subset \sX$ be an integral regular closed subscheme of
dimension two which is faithfully flat over $R$. We let $E \subset Y$ denote the
reduced special fiber of $\sD$ and let $D \subset X$ denote the generic fiber
of $\sD$. We let $U = X \setminus D$ and $\sU = \sX \setminus \sD$.
We refer to \cite[Defn.~2.5]{Saito-JNT} for the definition of the rank of $D$.

We have the commutative diagram
\begin{equation}\label{eqn:Exm-4-1}
    \xymatrix@C1pc{
      H^1_\et(\sD, {\Q_\ell}/{\Z_\ell}) \ar[r]^-{\cong} \ar[d] &
      H^1_\et(E, {\Q_\ell}/{\Z_\ell}) \ar@{.>}[d] \\
      H^3_\et(\sX, {\Q_\ell}/{\Z_\ell}(1)) \ar[r]^-{\cong} &
      H^3_\et(Y, {\Q_\ell}/{\Z_\ell}(1)),}
\end{equation}
where the left vertical arrow is the Gysin homomorphism. This yields a
`Gysin homomorphism $u_* \colon H^1_\et(E, {\Q_\ell}/{\Z_\ell}) \to 
H^3_\et(Y, {\Q_\ell}/{\Z_\ell}(1))$ such that the above diagram commutes, where
$u \colon E \inj Y$ is the inclusion.

\begin{prop}\label{prop:Exm-4}
  Suppose we can choose $p$ and $\ell \neq p$ such that the following hold.
  \begin{enumerate}
  \item
    The rank of $D$ is at least one.
  \item
    $H^1_{\rm {\etl}}(E, {\Q_\ell}/{\Z_\ell}) \to
    H^1_{\rm {\etl}}(E_n,  {\Q_\ell}/{\Z_\ell}) \bigoplus  H^1_{\rm {\etl}}(E_\sing,  {\Q_\ell}/{\Z_\ell})
    \bigoplus H^3_{\rm {\etl}}(Y, {\Q_\ell}/{\Z_\ell}(1))$
    is injective.
\end{enumerate}
  Then  we have
  $\Br(\sX)\{\ell\} \neq \Br(U)^\perp\{\ell\} \neq \Br(\sU)\{\ell\}$.
\end{prop}
\begin{proof}
  It follows from condition (1) of the proposition
  and \thmref{thm:Main-10*} (whose hypotheses are satisfied) that
  $\Br(U)^\perp\{\ell\} \neq 0$. In particular, $\Br(\sU) \{\ell\} \neq 0$.
Using the localization sequence
  \begin{equation}\label{eqn:Exm-4-2}
  0 \to \Br(\sU)\{\ell\} \xrightarrow{\partial} H^1_\et(\sD,  {\Q_\ell}/{\Z_\ell}) \to
  H^3_\et(\sX, {\Q_\ell}/{\Z_\ell}(1)),
  \end{equation}
  we can therefore find $\alpha \in \Br(\sU)\{\ell\}$ such that
  $\partial(\alpha) \neq 0$.
We choose one such $\alpha \in \Br(\sU) \{\ell\}$.

Using ~\eqref{eqn:Exm-4-2}, condition (2) of the proposition, it follows that
  the composite map
  \[
  \Br(\sU)\{\ell\} \to  H^1_\et(E,  {\Q_\ell}/{\Z_\ell})
  \to H^1_\et(E_n,  {\Q_\ell}/{\Z_\ell}) \bigoplus  H^1_\et(E_\sing,  {\Q_\ell}/{\Z_\ell})
  \]
  is injective.

If $\partial(\alpha)$ dies in $H^1_\et(E_\sing,  {\Q_\ell}/{\Z_\ell})$, then
  by the Chebotarev density theorem for $E_n$, we conclude that there exists
  $y \in (E_n)_{(0)}$ such that the pull-back of $\delta(\alpha)$ to
  $H^1_\et(k(y), {\Q_\ell}/{\Z_\ell})$ is not zero. Letting $x \in E$ denote the
  image of $y$, it follows that $\delta(\alpha)$ does not die in
  $H^1_\et(k(x), {\Q_\ell}/{\Z_\ell})$. We have thus concluded that there
  exists $x \in E_{(0)}$ such that $\delta(\alpha)$ does not die in
  $H^1_\et(k(x), {\Q_\ell}/{\Z_\ell})$. By going to a finite unramified extension
  and applying \cite[Lem.~6.2]{Ghosh-Krishna-Jussieu}), we can find a point
  $P \in U_{(0)}$ whose specialization is $x$.

The commutative diagram
  \begin{equation}\label{eqn:Exm-4-0}
    \xymatrix@C1pc{
      \Br(\sU)\{\ell\} \ar[r]^-{\partial} \ar[d] & H^1_\et(E,  {\Q_\ell}/{\Z_\ell})
      \ar[d] \\
      \Br(k(P))\{\ell\} \ar[r]^-{\inv}_-{\cong} &  H^1_\et(k(x), {\Q_\ell}/{\Z_\ell})}
    \end{equation}
  then tells us that $\alpha$ does not die in $\Br(k(P))$.
  But this is the same thing
  as saying that $\alpha \notin \Br(U)^\perp$.
To complete the proof, it remains to show that
  $\Br(\sX) =0$ but $\Br(U)^\perp\{\ell\} \neq 0$. But this is clear from
  condition (1) of the proposition and \thmref{thm:Main-10*}.
\end{proof}

\subsection{Proofs of Theorems~\ref{thm:Main-6} and ~\ref{thm:Main-7}}
  \label{sec:Normal-surface}
  We let $R$ be an excellent hdvr with quotient field
  $k$ and residue field $\F$. We let $X$ be an integral, normal,
  projective surface over $k$ and let $U = X_\reg$ denote the regular locus of $X$
  with inclusion $j \colon U \inj X$. We let $\pi \colon \wt{X} \to X$ be any
  resolution of singularities of $X$ such that the exceptional divisor $D \subset
  \wt{X}$ is an snc divisor on $\wt{X}$. We let $\iota \colon D \inj \wt{X}$ denote
  the inclusion. We let $\wt{j} \colon U \inj \wt{X}$ denote the inclusion.

  Recall that the Levine--Weibel Chow group $\CH^{\lw}_0(X)$ of $X$ is the quotient
  of $Z_0(U)$ by the subgroup $R^{\lw}_0(U)$ generated by $\divf(f)$, where
  $C \subset U$ is an integral curve which is closed in $X$ and $f$ is a nonzero
  rational function on $C$. This coincides with $\CH^{\Lci}_0(X)$ defined in
  \S~\ref{sec:Prf-2} (cf. \cite[Lem.~2.1]{Ghosh-Krishna-Jussieu}).
  It is easy to see that the identity map of $Z_0(U)$ induces a canonical
  surjection $\pi^* \colon \CH^{\lw}_0(X) \surj C_0(U)$. Using this map and
  \lemref{lem:BM-pair-functor}, we get a pairing
  $\CH^{\lw}_0(X) \times \Br(U) \to {\Q}/{\Z}$ which we call the Brauer--Manin
  pairing for $X$.
  We now prove Theorems~\ref{thm:Main-6} and ~\ref{thm:Main-7}.
 Let $\ell$ be a prime different from $\Char(\F)$.

\vskip.2 cm

  {{\bf(A)} \sl Proof of Theorem~\ref{thm:Main-6}:}
  We let $\CH_0(\wt{X}|mD)$ denote the Chow group of 0-cycles with modulus $mD$ on
  $\wt{X}$ for $m \ge 1$ (cf. \cite{Kerz-Saito-Duke}).
  The identity map of $Z_0(U)$ then induces canonical surjective maps
\begin{equation}\label{eqn:Main-6-0}
\CH^{\lw}_0(X) \xrightarrow{\pi^*} \CH_0(\wt{X}|mD)
  \surj  \CH_0(\wt{X}|D) \xrightarrow{\phi_{\wt{X}}} H^S_0(U),
\end{equation}
where the first map
 exists by \cite[Thm.~1.5]{Ghosh-Krishna-Jussieu} and the third map exists by
      \cite[Thm.~5.1]{Schmidt-ANT}. The second map is the canonical surjection for
      every $m \ge 1$. We let
      $\eta_X \colon \CH^{\lw}_0(X) \surj  H^S_0(U)$ denote the composite map.

In combination with \cite[Thm.~6.5]{Binda-Krishna-Ecole} (when $\Char(k) = 0$) and
      \cite[Lem.~8.3]{Ghosh-Krishna-Jussieu} (when $\Char(k) > 0$),
\propref{prop:Tame-part} says that
      the maps $\CH^{\lw}_0(X)^{(\ell)} \to C_0(U)^{(\ell)} \to H^S_0(U)^{(\ell)}$ 
      are isomorphisms. Meanwhile, the map
      $j_* \colon H^4_{\et,c}(U, \Z_\ell(2)) \to H^4_{\et}(X, \Z_\ell(2))$ is an
      isomorphism by the localization sequence ~\eqref{eqn:MC-03} and
      \thmref{thm:Tate-Duality-0}.
We now use \thmref{thm:BM-main-pf} to conclude that the
      cycle class map
      \begin{equation}\label{eqn:Main-6-1}
      \cyc_X \colon \CH^{\lw}_0(X)^{(\ell)} \to H^4_{\et}(X, \Z_\ell(2))
      \end{equation}
      is injective. Equivalently, the Brauer--Manin pairing for $X$ induces an
      injective map $\CH^{\lw}_0(X) \to \Hom(\Br(U)\{\ell\}, {\Q}/{\Z})$
      modulo the maximal
  $\ell$-divisible subgroup of $\CH^{\lw}_0(X)$.
  This proves Theorem~\ref{thm:Main-6}.
  \qed

\vskip.3 cm

  {{\bf(B)} \sl Proof of Theorem~\ref{thm:Main-7}:}
  We repeat the above argument verbatim (using \thmref{thm:BM-main-p5-f} instead of
  \thmref{thm:BM-main-pf}) to finish the proof.
If $\Char(k) = 0$ and $H^2_\zar(X, \sO_X) = 0$, then it easily follows from the
Leray spectral sequence for $\pi$ and the isomorphism $\sO_X \xrightarrow{\cong}
\pi_*\sO_{\wt{X}}$ (which one can deduce using the Stein factorization theorem)
that $H^2_\zar(\wt{X}, \sO_{\wt{X}}) = 0$. In particular, $H^2_\et(\wt{X}, \Q_\ell(1))$
is algebraic.
\qed

\vskip .4cm

\noindent\emph{Acknowledgments.}
Special cases of
Theorems~\ref{thm:Main-4} and~\ref{thm:Main-10} were previously obtained by
different methods in collaboration with Samiron Sadhukhan as a part of his
PhD thesis. The authors are grateful to Samiron for collaboration.
The authors would like to thank Olivier Wittenberg for helpful correspondence
concerning the Tate conjecture over finite fields. 
The success of this paper would not have been possible without the foundational
contributions of Saito, Saito--Sato and Esnault--Wittenberg to the study of
cycle class map and Brauer--Manin pairing for smooth projective varieties.

\end{document}